\documentclass{article}
\usepackage{graphicx, dsfont, hyperref, amssymb, natbib, color, datetime} 

\usepackage[a4paper,margin=1in]{geometry}

\usepackage{amsmath,amsthm, multicol,array,mathtools,enumerate, bbm, rotating,subcaption} 
\newtheorem{lemma}{Lemma}
\newtheorem{theorem}{Theorem}

\title{Semi-supervised learning in unmatched linear regression using an empirical likelihood approach}
\author{Fadoua Balabdaoui$^{\star}$ and Jinyu Chen$^{\star}$}

\date{$^\star$ Seminar for Statistics, D-MATH, ETH of Zurich \\ \today}

\begin{document}

\maketitle

\begin{abstract}
Knowing the link between observed predictive variables and outcomes is crucial for making inference in any regression model. When this link is missing, partially or completely, classical estimation methods fail in recovering the true regression function. Deconvolution approaches have been proposed and studied in detail in the unmatched setting where the predictive variables and responses are allowed to be independent. In this work, we consider linear regression in a semi-supervised learning setting where, beside a small sample of matched data, we have access to a relatively large unmatched sample. Using maximum likelihood estimation, we show that under some mild assumptions the semi-supervised learning empirical maximum likelihood estimator (SSLEMLE) is asymptotically normal and give explicitly its asymptotic covariance matrix as a function of the ratio of the matched/unmatched sample sizes and other parameters. Furthermore, we quantify the statistical gain achieved by having the additional large unmatched sample over having only the small matched sample. To illustrate the theory, we present the results of an extensive simulation study and apply our methodology to the \lq\lq combined cycle power plant\rq\rq \ data set.





\end{abstract}

\section{Background and contributions}

\subsection{An overview}
Consider the standard linear regression model 
\begin{align*}
    Y = \beta_0^\top X + \epsilon 
\end{align*}
with $X \in \mathbb R^p$ the predictive variable, $Y \in \mathbb R$ the response, $\beta_0$ the unknown regression vector and $\epsilon$ the unobserved noise. Based on observed pairs $\{(X_i, Y_i)\}_{i=1}^m$ of responses and their corresponding (matching) covariates estimation of $\beta_0$ can be done using the ordinary least squares estimator (OLSE):  
\[
\widehat{\beta}_{\text{OLS}} = \arg\min_{\beta\in\mathbb{R}^p} \frac{1}{m} \sum_{i=1}^m (Y_i - \beta^\top X_i)^2.
\]
The OLSE comes with many advantages; e.g. it is straightforward to compute and is asymptotically normal under classical regularity conditions. Also, it does not require any knowledge about the distribution of the noise, except for being centered and having a finite variance conditionally on the covariate.  Suppose now that knowledge about the noise distribution is available, for example through expertise. For simplicity of exposition, we further assume that this knowledge is complete and that the distribution has known density $f^\epsilon$ with respect to Lebesgue measure.  In this case, a more natural estimator of $\beta_0$ is the maximum likelihood estimator (MLE):
\[
\widehat{\beta}_{\text{MLE}} = \arg\min_{\beta\in\mathbb{R}^p} -\frac{1}{m}\sum_{i=1}^m \log{f^\epsilon(Y_i - \beta^\top X_i)}.
\]
When $\epsilon \sim \mathcal{N}(0,\sigma^2)$, the OLSE and MLE are equal  but the MLE  is more efficient when the distribution of $\epsilon$ deviates from Gaussianity; see \cite{FadouaandJustine}.  Having this in mind, the goal in this paper is not to explore efficiency of the MLE over the OLSE as done in \cite{FadouaandJustine} but to consider inference in a linear regression model based on a small matched sample in addition of a large unmatched sample. Therefore, we are in a semi-supervised learning setting where we aim to exploit the advantages of knowing the link between the responses and covariates in the matched sample and of the large size of the unmatched one. Before going into describing our approach, we would like to review the existing literature and the recent developments in unlinked regression.  In this setting, one only has access to independent samples of responses $\{\tilde Y_i\}_{i=1}^{n_Y}$ and covariates $\{\tilde X_i\}_{i=1}^{n_X}$, with possibly $n_X \neq n_Y$, such that 
$$
Y_j \stackrel{d}{=}  f_0(X_i)  +  \epsilon_{ij}       
$$
for $(i, j) \in \{1, \ldots, n_X \} \times \{1, \ldots, n_Y\}$. Here, $\epsilon_{ij}$ are i.i.d and independent of $X_i$ for all $(i, j) \in \{1, \ldots, n_X \} \times \{1, \ldots, n_Y\}$, and $f_0$ is an unknown regression function. When $f_0$ is a uni-variate monotone function, estimation based on deconvolution techniques has been studied in several papers; see \citep{carpentier2016learning, balabdaoui2021unlinked, rigollet2019uncoupled, meis} and the references therein. When $f_0(x) = \beta_0^\top x$; i.e., in the unlinked linear regression setting, \cite{monafadoua2024} constructed a deconvolution least squares estimator (DLSE) of $\beta_0$.   Under some regularity conditions satisfied by the distribution of noise and that of the covariate, and identifiability of $\beta_0$, the authors could establish consistency and asymptotic normality of the DLSE. A related but a more specific setting is that of the so-called shuffled or permuted regression: There, it is assumed that the responses and covariates were measured on the same subject but the link between them is not accessible because the covariates have been permuted; 
see \citep{hsu2017linear, pananjady2017linear, Unnikrishnan:2018gp, slawski2019linear, slawski2020two,  Tsakiris20, slawski2021pseudo, zhang2021benefits}. Note that the main focus in the aforementioned works is to recover the unknown permutation. 

It is important to mention that there is a whole stream of articles where authors use the term \lq\lq semi-supervised learning\rq\rq \ based on labeled and unlabeled data in a way that is very different from the one considered here. If labeled/matched means the same, unmatched/unlabeled in this work refers to responses $\tilde Y_i$ and covariates $ \tilde X_j$ which are unlinked (even in the scenario where they have been measured on the same individuals/items) but can be still put together in regression model. Of course, such a regression model is not a standard one: The responses and covariates can only be required to satisfy an equality in distribution.  More details are given below. Thus, in our semi-supervised learning, we use that the unmatched responses and covariates give  information about the regression vector even though this information is of a different nature from the one provided by the matched sample. In the common literature on semi-supervised settings, the unlabeled data consist only of the covariates or covariates and predictors which might be biased. In the first framework, we can refer for example to  \cite{ChakraborttyCai2018, Azriel02102022} in linear regression, and to \cite{Semi-supervised-inference-for-nonparametric-logistic-regression}, where logistic regression was considered. In the second one, we refer to the seminal work on prediction-powered prediction of \cite{Prediction-powered-inference}.  There, the authors consider the situation where one has access to a small golden matched/labeled sample of pairs responses and covariates  $\{(X_i, Y_i)\}_{i=1}^n$, a large unlabeled sample of covariates  $\{\tilde X_i\}_{i=1}^N$  ($N > > n$) and a prediction rule $f: \{\tilde X_i\}_{i=1}^N \mapsto \{\tilde Y_i\}_{i=1}^N$. Using the golden sample, the authors propose an approach of correcting the prediction error made by $f$ with the goal of constructing sharper and valid confidence intervals for a given parameter of interest.


\subsection{The SSL setting and contributions} 

As mentioned above, we will consider the setting where both a matched and an unmatched sample are available. More specifically, consider a pair $(X, Y) \in \mathbb R^p \times \mathbb R$  such that 
\begin{eqnarray*}
Y  =  \beta_0^\top X+ \epsilon
\end{eqnarray*}
where $\beta_0 \in \mathbb R^p$ is unknown and $\epsilon$ is independent of $X$ and has density $f^\epsilon$.  We want to make inference about $\beta_0$ based on 

\begin{itemize}

\item independent unmatched covariates and responses (for simplicity assumed to be of the same size), $\tilde{X}_i, i=1, \ldots, n$ and $\tilde{Y}_j, j =1, \ldots, n$,  such that 
\begin{eqnarray*}
\tilde X_1, \ldots, \tilde X_n \stackrel{i.i.d.}{\sim}  X, \  \ \text{and}  \ \ \tilde Y_1, \ldots, \tilde Y_n \stackrel{i.i.d.}{\sim}  Y,
\end{eqnarray*}

\item i.i.d matched pairs of covariates and responses of size $m$:  
$$( X_1, Y_1), \ldots, (X_m, Y_m)  \stackrel{i.i.d.}{\sim} (X, Y). $$
\end{itemize}
If only the unmatched sample is available, then one can use the DLSE introduced by \cite{monafadoua2024}: \begin{eqnarray*}
\widehat{\beta}_{\text{DLSE}} &=& \arg\min_{\beta\in\mathbb{R}^p} \int \Big(\mathbb F^{\tilde Y}_n(y)-\frac{1}{n}\sum_{i=1}^n F^\epsilon(y-\beta^\top \tilde X_i)\Big)^2  d\mathbb F^{\tilde Y}_n(y), \\
&& \ \ \  \textrm{where  $\mathbb F^{\tilde Y}_n$ is the empirical distribution of $\tilde Y_i, i =1, \ldots, n$} \\
& = &  \arg\min_{\beta\in\mathbb{R}^p} \frac{1}{n}\sum_{j=1}^n \Big( \frac{j}{n}-\frac{1}{n}\sum_{i=1}^n F^\epsilon(\tilde{Y}_{(j)}-\beta^\top \tilde X_i)\Big)^2, \\
&& \ \ \  \textrm{where  $\tilde Y_{(j)}$ are the order statistics of $\tilde Y_i, i =1, \ldots, n$}.
\end{eqnarray*}
In this work, we will take a different approach, namely that of maximizing a likelihood function obtained by combining both the matched and unmatched samples. More specifically, if we denote $f^\epsilon$ by $f$, our goal is to study the properties of any maximizer of
\begin{eqnarray}\label{pseudologlik}
\ell_{n, m}(\beta) &= &  \frac{1}{n+m}  \sum_{j=1}^n \log \left( \frac{1}{n} \sum_{i=1}^n f(\tilde{Y}_j - \beta^\top \tilde{X}_i)\right)  + \frac{1}{n+m}   \sum_{k=1}^m \log f(Y_k - \beta^\top X_k)  \notag \\
&& 
\end{eqnarray}
over $\beta \in \mathbb R^p$, under some specific assumptions.  Some remarks are in order.   The function $\ell_{n, m}$  should be seen an empirical log-likelihood since the distribution of the covariate $X$ is unknown and hence has to be estimated using the observations $\tilde X_j, j =1, \ldots, n$ of the unmatched sample. In other words, if the distribution function of $X$, $F^X$ say, were known, then we would maximize the \lq\lq true\rq\rq \ log-likelihood
\begin{eqnarray*}
\beta \mapsto  \frac{1}{n+m}  \sum_{j=1}^n \log \left( \int f(\tilde{Y}_j - \beta^\top x) dF^X(x)\right)  + \frac{1}{n+m}   \sum_{k=1}^m \log f(Y_k - \beta^\top X_k).   
\end{eqnarray*}
The second aspect to be mentioned is the fact that we can re-write  $\ell_{n, m}(\beta)$ as 
\begin{eqnarray*}
 w_{n, m} \ \frac{1}{n} \sum_{j=1}^n \log \left( \frac{1}{n} \sum_{i=1}^n f(\tilde{Y}_j - \beta^\top \tilde{X}_i)\right)  + (1-w_{n, m}) \ \frac{1}{m} \sum_{k=1}^m \log f(Y_k - \beta^\top X_k)     
\end{eqnarray*}
with $w_{n,m} = \frac{n}{n+m}$. Therefore, the contributions of the unmatched and matched samples have natural weights which, under the assumption that the ratio $m/n \to \lambda \in (0,1)$, as $ n, m \to \infty$, converge to the limiting weights $1/(1+ \lambda)$ and $\lambda/(1+\lambda)$ respectively. More details about the samples ratio will be given in the next section.  \\

\par \noindent Our main contributions in this work can be listed as follows. 
\begin{itemize}
\item We study existence of a maximizer of $\ell_{n, m}$ and show that the optimization problem admits at least a solution for any finite $n$ and $m$ or with probability 1 for $n$ and $m$ large enough under the condition that $\lim_{n, m \to \infty} m/n = \lambda \in (0,1)$. We refer to any such a maximizer as the Semi-Supervised Learning Empirical Maximum Likelihood Estimator (SSLEMLE).

\item Using the theory of empirical processes, we show that the SSLEMLE is consistent. Under specific assumptions, we prove that is asymptotically normal and give the expression of its asymptotic covariance matrix as a function of $\lambda$, the density of the noise, the density of the covariate vector, and that of the marginal density of the response variable.

\item We study the statistical gain achieved by adding the large unmatched sample to the small matched one. Defined as the square root of the ratio of the determinants of the asymptotic covariances of the SSLEMLE and the MLE based on the small matched sample, we are able to derive its explicit formula in the case where the covariate vector and noise are Gaussian.  Although the Gaussian distribution for the covariate violates one of our assumptions,  we were able to validate it through simulations. The formula was obtained using matrix diagonalization and involved algebra, and hence is one of the main highlights of this work.

\item We illustrate the theory through simulations in several scenarios with the goal of showing that the Monte Carlo estimation of the statistical gain is close to the true one or providing an empirical approximation thereof in cases where an explicit formula is hard to derive. Furthermore, we showcase our methodology by applying it to the \lq\lq combined cycle power plant \rq\rq \ data set used for predicting the net hourly electrical energy output.
\end{itemize}

\subsection{Outline of the paper}

The paper is organized as follows. In the next section, we establish existence and consistency of a SSLEMLE. In section 3 we prove asymptotic normality and give the explicit formula for the statistical gain when both the covariates and noise are Gaussian. Although Gaussianity of the covariate violates one of our assumptions, we believe that obtained formula is a very interesting result in its own right. Section 4 shows simulations on synthetic and real data. We provide the proofs of all lemmas and theorems in the Appendix.

\section{The semi-supervised learning empirical maximum likelihood estimator (SSLEMLE)}\label{SSLEMLE}

Let $\tilde{X}_i, i=1, \ldots, n$ and $\tilde{Y}_j, j =1, \ldots, n$ denote again the independent unmatched data. Let $(X_k, Y_k), k =1, \ldots, m$ be the matched pairs and $f$ the density of the noise distribution. The empirical likelihood of the aggregated data is given by
\begin{eqnarray*}
\ell_{n, m}(\beta) =  \frac{1}{n+m}  \sum_{j=1}^n \log \left( \frac{1}{n} \sum_{i=1}^n f(\tilde{Y}_j - \beta^\top \tilde{X}_i)\right)  + \frac{1}{n+m}   \sum_{k=1}^m \log f(Y_k - \beta^\top X_k)   
\end{eqnarray*}
for $\beta \in \mathbb R^p$. The first goal in this section is to show under some appropriate conditions that a SSLEMLE; i.e., a maximizer of $\ell_{n,m}$ over $\mathbb R^p$, exists.

\subsection{Existence}

\par \noindent In the sequel, we will use the following assumption: \\

\medskip

\par \noindent (A0) The density of the noise $\epsilon$ admits the form
\begin{eqnarray*}
f(t)  =  c_\alpha  \exp\left( - d^{-\alpha} \vert t \vert^\alpha \right), \ t  \in \mathbb R,    
\end{eqnarray*}
for some $\alpha \ge 1$, $c_\alpha > 0$ and $d > 0$.  For $\alpha = 1$, $d = \mu$ and $c_1 = 1/(2\mu)$ for some $\mu > 0$, $f$ is the density of a Laplace distribution with intensity $\mu$.  For $\alpha =2$, $c_2 = 1/\sqrt{2\pi \sigma^2}$ and $d = \sqrt 2  \sigma$ for some $ \sigma > 0$, $f$ is the density of a centered Gaussian with standard variation $\sigma$. To derive $c_\alpha$ as a function of $d_\alpha$, note that
\begin{eqnarray*}
c^{-1}_\alpha  & =  &  2 d \int_0^\infty \exp( - t^\alpha)  dt \\
& =  & \frac{2 d}{\alpha} \int_0^\infty u^{1/\alpha -1} \exp(- u)  du, \ \   \text{using the variable change $u = t^\alpha$} \\
& = & \frac{2 d \Gamma(1/\alpha)}{\alpha} 
\end{eqnarray*}
and hence
$$
c_\alpha  = \frac{\alpha}{2 d \Gamma(1/\alpha)}.
$$
In the following,  we will prove existence of a maximizer  when $n$ and $m$ are fixed and when they are let to increase to $\infty$. The arguments are based on showing that if $\Vert \beta \Vert$ is too large, then $\beta$ is not a good candidate for maximizing $\ell_{n,m}$.

\subsubsection{The case of finite \texorpdfstring{$m$}{m} and \texorpdfstring{$n$}{n}}

\begin{lemma} \label{Lemma: Existence of MLE for fixed sample case}
Suppose that (A0) holds. If the matched design matrix $M = \begin{pmatrix}
    X_1^\top \\
    \vdots\\
     X_m^\top \\
\end{pmatrix} \in \mathbb{R}^{m\times p}$ has $\operatorname{rank}(M) = p$, then the total empirical log-likelihood function $\ell_{n, m}(\beta)$ admits at least a maximizer. 
    \end{lemma}

\subsubsection{The asymptotic case: When \texorpdfstring{$n, m \to \infty$}{n, m -> infty}}

In this asymptotic case, we need the following additional assumptions.  

\begin{description}

\item  (A1) $\lim_{n, m \to \infty}  m/n = \lambda \in (0,1)$,  




\item (A2) For any vector $v \in \mathbb R^p \setminus \{0\}$, 
$$
\mathbb P(v^\top X = 0)  =0  
$$
\end{description}
Under the condition that the design matrix has rank $p$, the proof of Lemma \ref{Lemma: Existence of MLE for fixed sample case} implies that for fixed $n$ and $m$ we have that $-\ell_{n,m}(\beta) > -\ell_{n,m}(0)$ for all $\beta\in\{b\in\mathbb{R}^p:\left\lVert b \right\rVert > R\}$ with 
\begin{eqnarray}\label{R}
R = \left( \frac{2^{\alpha-1}}{A^\ast} \sum_{j=1}^n \left\lvert \Tilde{Y}_j \right\rvert^\alpha +\frac{2^{\alpha}}{A^\ast} \sum_{k=1}^m \left\lvert Y_k \right\rvert^\alpha \right)^{\frac{1}{\alpha}}
\end{eqnarray}
and $A^\ast = \inf_{u \in \mathcal{S}^{p-1}} \sum_{k=1}^m \left\lvert u^\top X_k\right\rvert^\alpha$, where $ \mathcal{S}^{p-1}$ is the $(p-1)$-dimensional unit sphere. The proof of the next lemma is mainly based on showing that $R$ can be bounded above by some $R^\ast$ with probability 1 for large enough $n$ and $m$.

\medskip

\begin{lemma}\label{Lemma: Existence of MLE for asymptotic sample case}
Suppose that (A0)-(A2) hold. Then,  with probability 1, there exists  $m^\ast$ such that for all $n,m \ge m^\ast$ a maximizer of $\beta \mapsto \ell_{n,m}(\beta)$ belongs to $\overline {\mathcal{B}} (0, R^\ast)$ where $R^\ast> 0$ is given in (\ref{R*}).  
\end{lemma}

\medskip

\subsection{Consistency}
Recall $R^\ast$ from the previous section; see also (\ref{R*}). In the following, and for the sake of convenience, we will redefine $R^\ast$ to be equal to $\max(R^\ast, \lVert \beta_0\rVert )$.  In the following, we will use the notation:
\begin{itemize}

\item For any measure $Q$ and function $f$ which is integrable with respect to $Q$, $Qf:= \int f dQ$,

\item $\mathbb P_m$ the empirical probability measure based on the matched sample $(X_k, Y_k), k=1, \ldots, m$,

\item $\mathbb P$ the probability measure of the pair $(X, Y)$ such that $Y= \beta^\top_0 X +  \epsilon$. In particular, $d \mathbb P(x, y) =  f(y - \beta^\top_0 x)  f^X(x) dx dy$,

\item $\mathbb P^{\tilde X}$ the probability measure of $\tilde{X}$. Note that it is equal to $\mathbb P^{X}$. In particular, we have $d \mathbb P^{\tilde X}(x) =d\mathbb P^{X}(x) = f^X(x) dx$,

\item $\mathbb P^{\tilde Y}$ the (marginal) probability measure of the response $\tilde{Y} \stackrel{d}{=}  \beta^\top_0 \tilde{X} + \epsilon$.  In particular $d \mathbb P^{\tilde Y}(y)  =  \left(\int f(y - \beta^\top_0 x) f^X(x) dx\right) dy$, 

\item $\mathbb G_m:=  \sqrt m (\mathbb P_m - \mathbb P)$ and $\mathbb G_n = \sqrt n (\mathbb P_n - \mathbb P)$.
\end{itemize}


\medskip
Consider the population criterion
\begin{eqnarray}\label{popcrit}
\ell(\beta) & =  &   \frac{1}{\lambda + 1} \int \log \left( \int f(y - \beta^\top x) d \mathbb P^{\tilde{X}}(x)  \right)  d\mathbb P^{\tilde Y}(y)  +    \frac{\lambda}{\lambda + 1}  \int \log f(y - \beta^\top x)  d\mathbb P(x, y) \notag \\
&&
\end{eqnarray}
for $\beta \in \mathbb R^p$. \\

\par  \noindent We will make the following assumption:
\begin{description}
\item (A3) The covariate $X$ is compactly supported. Hence, we assume that there exists $B > 0$ such that
$$
\mathbb P( \Vert X \Vert \le B)  = 1.
$$
\end{description}

Such an assumption might be too strong and is mainly needed so that the arguments used in the proof of consistency and later of weak convergence work.  Below we show that the expression of the statistical gain obtained under the assumption that $X$ is Gaussian (note that this violates (A3)) can be validated through Monte Carlo simulations. \\

\par \noindent The following theorem is key to showing consistency of the MLE.
\begin{theorem}\label{Consis}
Suppose that (A0)-(A3) are satisfied. Then, it holds that
\begin{eqnarray}\label{UC}
\sup_{\beta \in \overline{\mathcal{B}}(0, R^\ast)}  \left \vert \ell_{n, m}(\beta) -  \ell(\beta)  \right \vert  =o_{\mathbb P \otimes\mathbb P^{\tilde X} \otimes \mathbb P^{\tilde Y}}(1)
\end{eqnarray}
and for any $r > 0$
\begin{eqnarray}\label{Iden}
\sup_{\beta \in \mathcal O_r}  \ell(\beta)  <  \ell(\beta_0)
\end{eqnarray}
where $\ell$ is the same population criterion defined in (\ref{popcrit}) and $\mathcal O_r = \mathcal{B}(0, R^\ast) \cap \{\beta:  \Vert \beta - \beta_0 \Vert > r \}$.  In particular, this implies that
\begin{eqnarray*}
\widehat \beta_{n, m}  \to_{\mathbb P \otimes \mathbb P^{\tilde X} \otimes \mathbb P^{\tilde Y}}  \beta_0
\end{eqnarray*}
as $n, m \to \infty$, where $\widehat \beta_{n, m}$ is any maximizer of the empirical log-likelihood $\ell_{n, m}$.
\end{theorem}

The proof of Theorem \ref{Consis} is very much involved due to the \lq\lq unmatched\rq\rq \ part in the log-likelihood function. Handling this part requires non-trivial use of empirical processes. This was done through defining several classes of functions which admit desirable properties; e.g. having a finite bracketing integral.    It is worth noting that putting the matched and unmatched samples together solves any non-identifiability issue that may originally be there with the unmatched sample alone. For example, it is known that that the set of $\beta_0$ such that $\tilde Y \stackrel{d}{=}  \beta^\top_0 X + \epsilon$ when $X \sim \mathcal{N}(0, \Sigma)$ for a positive definite covariance matrix $\Sigma \in \mathbb R^{p \times p} $ is the ellipsoid $\{\beta \in \mathbb R^r:  \beta^\top \Sigma  \beta = c \} $ for some constant $c > 0$; see \cite{monafadoua2024} and  \cite{fadmartjohn}.  Hence, the regression model in non-identifiable in this case, and with unmatched data alone, any reasonable estimator should belong to a set which approximates the true ellipsoid. When matched data from the same model are added, the ellipsoid reduces to a unique element, which is also the regression vector in the regression model from which the matched pairs are observed.   

\section{Weak convergence and statistical gain}\label{WC}

\subsection{Asymptotic normality}
In this section, we need the additional assumption: \\

\begin{description}
    \item (A4)  The power $\alpha \ge 1$ in $f(t) = c_\alpha \exp(-d^{-\alpha} \vert t \vert^\alpha), t \in \mathbb R,$ is an integer. 
\end{description}

Assumption (A4) is made to ensure that the density  has a bounded $k$-th derivative for any integer $k \ge 1$. In the following theorem, which is one very important contribution of this work, we derive the asymptotic normality of a SSLEMLE  $\widehat \beta_{n,m}$. In the proof, we use a first order Taylor expansion of the gradient of $\ell_{n,m}$ around $\widehat \beta_{n,m}$ and its consistency proved in the previous section. Two key elements in the proof are showing that
$$
\frac{1}{\sqrt n}\sum_{j=1}^n \frac{\int x f'(\widetilde Y_j - \beta_0^\top x) d\mathbb P^{\tilde X}_n(x)}{\int f(\widetilde Y_j - \beta_0^\top x) d\mathbb P^{\tilde X}_n(x)}   \to_d \mathcal{N}(0, \Gamma_1 + \Gamma_2)
$$
where $\Gamma_1$ and $\Gamma_2$ are given in (\ref{Gamma1}) and (\ref{Gamma2}) respectively, and that the value of the Hessian matrix of $\ell_{n, m}$ at $\beta_0$ converges in probability to 
$$
-\frac{1}{1+\lambda} \Gamma_1 -  \frac{\lambda}{1+\lambda} \Sigma_2,
$$
where we recall that $\lambda = \lim_{n,m \to \infty} m/n$.  See also Theorem \ref{Remainder} and Theorem \ref{Remainder2}. One can see these results as non-standard Central Limit Theorem and Law of large numbers that require use of additional care since the quantities under study are ratios with random denominators. It is worth noting that the asymptotic variance is not of the form of a Fisher information matrix. As opposed to the usual setting in the weak convergence of the MLE under the classical regularity condition where the asymptotic variance of the score is equal to the limit of the negative of the Hessian, the inverse of the matrix in the middle of the product defining the asymptotic variance of $\widehat \beta_{n,m}$ is not equal to the matrix on the right and left. This asymmetry mainly stems from the fact that the variability of the unmatched part in the score function, that is the gradient of the log-likelihood function, is due to the variability of $\tilde X_i, i =1, \ldots, n$ in addition of that of $\tilde Y_j, j=1, \ldots, n$. This is the reason why the asymptotic variance of the value of the score at $\beta_0$ involves $\Gamma_2$, the variance of an empirical process of the unmatched covariates $\tilde X_i, i =1, \ldots, n$. On the other hand, the limit in probability of the unmatched part in the Hessian matrix of the log-likelihood does not depend on the variability of $\tilde X_i, i =1, \ldots, n$ since it turns out that 
$$
\frac{1}{n} \sum_{j=1}^n \frac{\int x x^\top f''(\tilde Y_j - \beta_0^\top x) d\mathbb P^{\tilde X}_n(x)}{\int f(\tilde Y_j - \beta_0^\top x) d\mathbb P^{\tilde X}_n(x)} \to_{\mathbb P^{\tilde X} \otimes \mathbb P^{\tilde Y}}  0,
$$
see also (\ref{FormerConv}) in the proof of Theorem \ref{Remainder2}.
\bigskip

\begin{theorem}\label{AsympNorm}
Under the Assumptions (A0-A4), it holds that 
\begin{eqnarray*}
\sqrt{m +n} (\widehat \beta_{n, m} - \beta_0)  \to_d \mathcal{N}(0, \Sigma_{\text{SSL}})
\end{eqnarray*}
where
\begin{eqnarray*}
\Sigma_{\text{SSL}} =\left(\frac{1}{1+\lambda} \Gamma_1 + \frac{\lambda}{1+\lambda}\Sigma_2\right)^{-1} \left(\frac{1}{1+\lambda}\left(\Gamma_1+\Gamma_2\right)+\frac{\lambda}{1+\lambda}\Sigma_2\right)\left(\frac{1}{1+\lambda} \Gamma_1 + \frac{\lambda}{1+\lambda}\Sigma_2\right)^{-1}
\end{eqnarray*}
with $\Gamma_1, \Gamma_2$ are given in (\ref{Gamma1}) and (\ref{Gamma2}) respectively and 
$$
\Sigma_2 = \left(\int\frac{( f'(t))^2}{f(t)}dt\right) \mathbb E[X X^\top] =  \frac{\alpha^2}{d^2}  \frac{\Gamma\left(2 - \frac{1}{\alpha}\right)}{\Gamma\left(\frac{1}{\alpha}\right)} \ \mathbb E[X X^\top].
$$
\end{theorem}

\medskip

\subsection{The statistical gain of adding unmatched samples} \label{subsection: gain}
From the results of asymptotic normality of the SSLEMLE $\widehat{\beta}_{n,m}$ proved in Theorem \ref{AsympNorm} and that of the MLE based on matched samples (mMLE) $\widehat{\beta}_m$, we know that
\begin{align*}
    \sqrt{m+n}(\widehat{\beta}_{n,m} - \beta_0) \overset{d}{\rightarrow} \mathcal{N}(0,\Sigma_{\text{SSL}})
\end{align*}
with $\Sigma_{\text{SSL}}$ as in Theorem \ref{AsympNorm} and  for $\widehat{\beta}_m$ we have
\begin{align*}
    \sqrt{m}(\widehat{\beta}_{m} - \beta_0) \overset{d}{\rightarrow} \mathcal{N}(0,\Sigma_{\text{mMLE}})
\end{align*}
with $\Sigma_{\text{mMLE}} = \left(\int\frac{\left( f'(t)\right)^2}{f(t)}dt \cdot \mathbb E[XX^\top]\right)^{-1} = \Sigma_2^{-1}$. 
Note that the gain is directly related to the asymptotic efficiency of the SSLEMLE over the matched MLE. To evaluate the asymptotic efficiency,  we need first to scale the estimation errors with the same factor.  We have that
\begin{align*}
     \sqrt{m}(\widehat{\beta}_{n,m} - \beta_0) = \sqrt{\frac{m}{n+m}}\sqrt{m+n}(\widehat{\beta}_{n,m} - \beta_0)\overset{d}{\rightarrow} \mathcal{N}(0,\frac{\lambda}{1+\lambda}\Sigma_{\text{SSL}}).
\end{align*}
Now, define 
$$\Tilde{\Sigma}_{\text{SSL}} = \frac{\lambda}{1+\lambda}\Sigma_{\text{SSL}} =  \left(\frac{1}{\lambda} \Gamma_1 + \Sigma_2\right)^{-1} \left(\frac{1}{\lambda}\left(\Gamma_1+\Gamma_2\right)+\Sigma_2\right)\left(\frac{1}{\lambda} \Gamma_1 +\Sigma_2\right)^{-1}.$$ 
Then, we have that
\begin{align*}
    (\widehat{\beta}_{n,m} - \beta_0)^\top \left(m \Tilde{\Sigma}_{\text{SSL}}^{-1}\right) (\widehat{\beta}_{n,m} - \beta_0) \overset{d}{\rightarrow} \chi_{(p)}^2
\end{align*}
and 
\begin{align*}
    (\widehat{\beta}_{m} - \beta_0)^\top \left(m{\Sigma}_{\text{mMLE}}^{-1}\right) (\widehat{\beta}_{m} - \beta_0) \overset{d}{\rightarrow} \chi_{(p)}^2.
\end{align*}
Consider $q_{1-\alpha, p}$ to be the $(1-\alpha)-$quantile  of $\chi_{(p)}^2$-distribution. Then, the confidence regions of $\beta_0$ associated with the SSLEMLE and matched MLE are given by:
\begin{align*}
    \text{CR}_{\text{SSL}} = \{\beta\in\mathbb{R}^p:(\widehat \beta_{n,m} - \beta)^\top \Tilde{\Sigma}_{\text{SSL}}^{-1}(\widehat{\beta}_{n,m} - \beta) \leq \frac{q_{1-\alpha, p}}{m}\}
\end{align*}
and 
\begin{align*}
    \text{CR}_{\text{mMLE}} = \{\beta\in\mathbb{R}^p:(\widehat{\beta}_{m} - \beta)^\top \Sigma_{\text{mMLE}}^{-1}(\widehat{\beta}_{m} - \beta) \leq \frac{q_{1-\alpha, p}}{m}\}
\end{align*}
which  are ellipsoids. Thus,  the ratio of the volumes of the ellipsoids can be computed in order to assess the statistical gain due to the added unmatched sample.  Since the volume of the ellipsoid $\{x \in \mathbb R^p:   x^\top M^{-1} x  \le 1 \}$, for some positive definite matrix $M$, is equal to $\sqrt{\text{det}(M)} \cdot \text{vol}(B^p)$, with $B^p$ the $p$-dimensional unit ball, we can show the following result.

\medskip

\begin{theorem}\label{gain}
Suppose that $\epsilon \sim \mathcal{N}(0, \sigma_\epsilon^2)$ with $\sigma_\epsilon > 0$, and that (A1) holds. Also, assume that $X \sim \mathcal{N}(\mu_X, \Sigma_X)$ with $\mu_X \in \mathbb R^p: \mu_X \ne 0$ and  $\Sigma_X$ is positive definite. Define
\begin{eqnarray*}
\alpha_0 &= & \Sigma_X^{1/2} \beta_0  \\
\theta_X &=  & \Sigma_X^{-1/2}  \mu_X \\
\zeta & = &  \frac{\theta^\top_X \alpha_0}{\Vert \theta_X \Vert \Vert \alpha_0 \Vert}  =   \frac{\mu_X^\top \beta_0}{\sqrt{\beta_0^\top \Sigma_X \beta_0} \sqrt{\mu_X^\top \Sigma_X^{-1} \mu_X}} \\
\rho & = & \frac{1}{\Vert \theta_X \Vert^2}  =  \frac{1}{\mu_X^\top \Sigma_X^{-1} \mu_X} \\
\eta & = & \frac{\Vert \alpha_0 \Vert^2}{\sigma_\epsilon^2}  =  \frac{\beta_0^\top \Sigma_X \beta_0}{\sigma_\epsilon^2}.
\end{eqnarray*}
Then, the OLSE and the matched MLE are equal and the statistical gain $G$ determined by 
\begin{eqnarray}\label{equation: Gain with mean neq 0}
 &&G= \left(\sqrt{\frac{\det{\Tilde{\Sigma}_{\text{SSL}}}}{\det{\Sigma_{\text{OLS}}}}} \right)^{-1} \notag\\
 &&=   \frac{1+\frac{1}{\lambda}\left(\frac{1}{1+\eta}\frac{1}{1+\rho}+\frac{2\eta}{(\eta+1)^2}(1-\frac{\zeta^2}{1+\rho})\right) + \frac{1}{\lambda^2}\frac{1-\zeta^2}{1+\rho}\frac{2\eta}{(\eta+1)^3}} {\sqrt{1+\frac{1}{\lambda}\left(\frac{1}{1+\eta}\frac{1}{1+\rho}(1+\frac{\eta}{\eta+1})+\frac{2\eta}{(\eta+1)^2}(1-\frac{\zeta^2}{1+\rho})(1+\frac{\eta^2}{(\eta+1)^2})\right) + \frac{1}{\lambda^2}\frac{1-\zeta^2}{1+\rho}\frac{2\eta}{(\eta+1)^3}(1+\frac{\eta^2}{(\eta+1)^2})(1+\frac{\eta}{\eta+1})}}.
 \notag \\ 
\end{eqnarray}
If $\mu_X = 0 \in \mathbb R^p$, then
\begin{eqnarray} \label{equation: Gain with mean = 0}
G= \frac{1+\frac{2}{\lambda}\frac{\eta}{(\eta+1)^2}}{\sqrt{1+\frac{2}{\lambda}\left(\frac{\eta}{(\eta+1)^2} + \frac{\eta^3}{(\eta+1)^4}\right)}}. 
\end{eqnarray}
\end{theorem}

\medskip
\medskip

Strictly speaking, Theorem \ref{gain} is a conjecture about the statistical gain of the SLLEMLE when the covariate is Gaussian. In fact, assumption (A3) is not fulfilled in this case since it requires the covariate distribution to be compactly supported. However, the formulas  derived in the theorem can be validated through Monte Carlo simulations; see Figures \ref{fig:s1} and \ref{fig:s2}.  

Some remarks about the formulas obtained in Theorem \ref{gain} are in order.
\paragraph{The case $\mu_X = 0$.} It follows from (\ref{equation: Gain with mean = 0}) that for small $\lambda$ 
\begin{eqnarray*}
G \approx \frac{\sqrt{2\eta}}{\sqrt{\lambda}\,\sqrt{\,2\eta^{2} + 2\eta + 1\,}}.
\end{eqnarray*}
Then, $G$ achieves its maximum at $\eta^\ast = 1/\sqrt 2$, or equivalently at $\text{SNR}^\ast = \sqrt{\eta^\ast} =  1/2^{1/4} \approx  0.840$, and the maximum gain is $G^\ast \approx  0.643/\sqrt{\lambda} $.  
In general, it can be shown that the gain is a unimodal function of $\eta$ for any value of $\lambda \in (0,1)$. Before presenting a formal proof, this behavior can be explained as follows.  When the SNR is too small, it might be too hard for the unmatched sample to boost the performance of the matched MLE but its effect starts to be visible for larger values of the SNR. On the other hand, if the SNR is too large, then the matched sample is enough to achieve a good estimation quality, which also means that the contribution of the unmatched sample becomes less \lq\lq impressive\rq\rq \ in this case. Although this explanation seems plausible, we would like to note that unimodality of the gain might be distribution-specific. In fact, a different behavior is observed when the noise has a Laplace distribution; see Figure \ref{fig:s6}. However, for all noise distributions, we should expect that the gain tends to $1$ for large values of SNR. In the Gaussian case, the limit $1$ can be easily recovered from both formulas by tending $\eta$ to $\infty$. 

To show that the gain is always unimodal function when $X$ is Gaussian with $\mu_X =0$,  let us fix $\lambda \in (0,1)$ and define the gain $G$ as a function of $\eta$. In other words, $G$ is given by 
\begin{align*}
    G_\lambda(\eta) & =  \frac{1+\frac{2}{\lambda}\frac{\eta}{(\eta+1)^2}}{\sqrt{1+\frac{2}{\lambda}\left(\frac{\eta}{(\eta+1)^2} + \frac{\eta^3}{(\eta+1)^4}\right)}} \\
    &= \frac{\eta^2 + 2(1 + \lambda^{-1})\eta  +1}{\sqrt{\eta^4 + 4 (1+ \lambda^{-1}) \eta^3 + 2 (3 + 2\lambda^{-1}) \eta^2 + 2 (2 + \lambda^{-1}) \eta + 1}}
\end{align*}
for $\eta > 0$. After some algebra we find that the first derivative of $G_\lambda$ is equal to 
\begin{align*}
    G_\lambda'(\eta) = \frac{- 4 (1+\lambda) \eta^3 - 3 \lambda \eta^2 + 2 (1+\lambda) \eta + \lambda}{\lambda^2 \left(\eta^4 + 4 (1+ \lambda^{-1}) \eta^3 + 2 (3 + 2\lambda^{-1}) \eta^2 + 2 (2 + \lambda^{-1}) \eta + 1\right)^{\frac{3}{2}}}
\end{align*}
whose sign is that of 
\[
P_\lambda(\eta) := - 4 (1+\lambda) \eta^3 - 3 \lambda \eta^2 + 2 (1+\lambda) \eta + \lambda, \ \eta \in (0, \infty).
\]
In order to investigate the sign of $P_\lambda$, we view it here as a   polynomial defined on $\mathbb R$. Note that  $P_\lambda(-1) = 2 >0$, $P_\lambda(-\frac{1}{2}) = -\frac{1}{2} - \frac{1}{4}\lambda <0$ , $P_\lambda(0) = \lambda > 0$ and $P_\lambda(1) = -2-4\lambda < 0$. By the intermediate value theorem, the three roots $\eta_1^\ast$, $\eta_2^\ast$ and $\eta_3^\ast$ of $P_\lambda$ are such that $\eta_1^\ast\in (-1, -\frac{1}{2})$, $\eta_2^\ast\in (-\frac{1}{2}, 0)$ and $\eta_3^\ast\in (0, 1)$. Since there are no more roots than these three, we can then conclude that $\eta_3^\ast$ is the unique root of $P_\lambda$ in $(0,\infty)$.  This also means that $P_\lambda(\eta) > 0$ for $\eta \in (0, \eta_3^\ast)$ and $P_\lambda(\eta) < 0$ for $\eta \in  (\eta_3^\ast, \infty)$. This in turn implies $G_\lambda$ is increasing on $(0, \eta_3^\ast)$ and decreasing on $(\eta_3^\ast, \infty)$. This shows that the gain is in fact unimodal for any $\lambda \in (0,1)$.

Now, if $\eta$ or equivalently the SNR is fixed, one can view the the gain as a function of $\lambda$. Put $t = 2/\lambda$. Then, we can write that
\begin{eqnarray*}
G = G_\eta(t) & = &  \frac{1  +  t  \frac{\eta}{(\eta+ 1)^2}}{\sqrt{1 + t  \left(\frac{\eta}{(\eta+1)^2} + \frac{\eta^3}{(\eta+1)^4}\right)}}.
\end{eqnarray*}
Putting $A  = \eta/(\eta+ 1)^2$ and $B = \eta/(\eta+1)^2 + \eta^3/(\eta+1)^4$
\begin{eqnarray*}
G'_\eta(t) =  \frac{2A - B + AB t}{2 (1+ t B)^{3/2}}
\end{eqnarray*}
with $2A - B = \eta/(\eta +1)^2 (1- \eta^2/(\eta+1)^2) > 0$ and $ AB > 0$. It follows that $G$ is an increasing function of $1/\lambda$. This means that the smaller $\lambda$ the bigger is the gain. This is an expected result since smaller values of $\lambda$  occur when the size of the unmatched sample is significantly larger that of the matched one.

When $\beta_0 = 0 \in\mathbb{R}^p$, then (\ref{equation: Gain with mean neq 0}) implies that
\begin{align*}
   G= \sqrt{1+\frac{1}{\lambda}\frac{1}{1+\rho}} \quad \text{with $\rho  =\frac{1}{\mu^\top_X \Sigma_X^{-1}\mu_X}$}
\end{align*}
and hence if $\mu_X = \beta_0 = 0$, then $G=1$, which means that there is no statistical gain from adding the unmatched part by our proposed estimator. This remains true for any distribution of $X$ and $\epsilon$.  In fact, when $\mu_X = \beta_0= 0$, it is elementary to check that $\Gamma_1 = \Gamma_2 = 0$, and hence $\widetilde{\Sigma}_{SSL} = \Sigma_2^{-1} = \Sigma_{OLS}$.

\paragraph{The case $\mu_X \ne 0$.} The formula in (\ref{equation: Gain with mean neq 0}) describes a more complex relationship between $G$ and the SNR due to the additional dependence on the parameters $\rho$ and $\zeta$. If $\Sigma_X$ is a diagonal,  $\zeta$ can be viewed as the cosine of the angle between $\beta_0$ and $\mu_X$. Hence, when $\mu_X \ne 0$, the gain is affected by the orientation of $\beta_0$ with respect to $\mu_X$. The exact variations of the gain as a function of $\zeta$, when all the other parameters are held fixed, seem to be very hard to study. However, when $\lambda$ is small then it holds that
\begin{eqnarray*}
G \approx \frac{1}{\lambda}  \frac{\sqrt{1 - \zeta^2} \sqrt{2\eta}}{\sqrt{(1+\rho) (\eta^2 + (\eta +1)^2) (2 \eta +1)}}
\end{eqnarray*}
which shows that $G$ is decreasing function in $\vert \zeta \vert$ and hence the more $\beta_0$ and $\mu_X$ are aligned the smaller $G$ is. See also Figure \ref{fig:G(zeta)} where we plot $G$ versus $\zeta$ for $\rho = \eta = 1$ and $\lambda = 0.1$.

\begin{figure}[h]
\centering
   \includegraphics[width=0.65\textwidth]{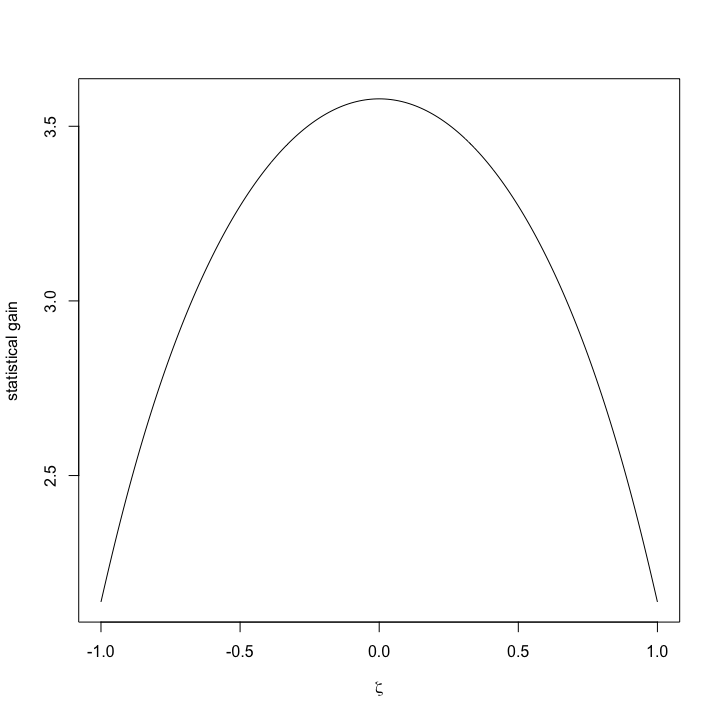}
    \caption{Statistical gain versus $\zeta$ with fixed $\rho =1$, $\eta = 1$ and $\lambda = 0.1$.}
    \label{fig:G(zeta)}
\end{figure}

\section{Simulation results and data application}
\subsection{Simulation settings}
In the following, we present the results obtained for synthetic data sets. We used \textsf{R} by \cite{Rstats} to implement the simulations. In the codes, we also used the libraries \texttt{sfsmisc} by \cite{Rsfsmics}, \texttt{VGAM} by \cite{RVGAM} and \texttt{MASS} by \cite{RMASS}. One of the main goals of the simulation study is to validate the theoretical expression for the statistical gain obtained in subsection \ref{subsection: gain} with its empirical Monte Carlo estimates.  

We investigate the relationship between the gain and $\eta$ as shown in (\ref{equation: Gain with mean = 0}). To simplify the simulation settings, we fixed the standard deviation $\sigma_\epsilon$ of the noise variable and also the covariance matrix $\Sigma_X$ of the covariate. Thus, $\eta = \beta_0^\top \Sigma_X \beta_0 /\sigma_\epsilon^2$ depends only on the true regression vector $\beta_0$. Since the value $\sqrt{\eta}$ is equal to signal-to-noise ratio, we denote it by SNR. 

We used 15 different regression vectors $\beta_0$'s which yield different values of $\eta$, and hence different statistical gains. To generate these $\beta_0$'s, we implemented the following steps. We considered the vector $\begin{pmatrix} 2 & 2 & 2 \end{pmatrix}^\top$, and added to it the 8 corners and the centers of the 6 faces of the cube $[-1,1]^3$. To these 15 vectors, we added a small Gaussian perturbation  $ \sim \mathcal{N}(0,0.1^2\cdot\mathbbm{1}_{3\times 3})$. The obtained vectors were then divided by their norms and then rescaled so that the final norms are equal to 15 real numbers which are equally spaced between 0 and 8.
For each vector $\beta_0$, we generated the following data:
\begin{enumerate}[(1)]
    \item $n$ unmatched covariates sampled from some given distribution of $X$,
    \item $n$ unmatched responses sampled separately from the distribution $Y \overset{d}{=} \beta_0^\top X + \epsilon$,
    \item $m$ matched pairs for some fixed $\lambda$.
\end{enumerate}
Then, we computed the SSLEMLE $\widehat{\beta}_{n,m}$ and $\widehat{\beta}_m$  the MLE based the matched sample, which is equal to the OLSE in the settings where the noise is Gaussian. For the settings where $\epsilon$ is chosen to follow a Laplace distribution, we also compute the OLSE. For each $\beta_0$, we repeated the above process 2500 times, and hence obtained 2500 estimators for each of the 15 vectors $\beta_0$. In the final step, we calculated for each $\beta_0$ the empirical covariance matrix based on the   
the obtained 2500 scaled estimation errors. 
The estimate of the statistical gain is simply set to be the ratio of the square root of the determinants of the obtained sample covariance matrices. 
We then compared these results with the theoretical ones given in (\ref{equation: Gain with mean neq 0}) and (\ref{equation: Gain with mean = 0}), when both the covariate variable $X$ and the noise $\epsilon$ follow a Gaussian distribution as assumed in subsection \ref{subsection: gain}.

In Table \ref{Table: simulations}, we indicated the 6 simulation settings considered in this section.

\begin{sidewaystable}
\centering
\begin{tabular}{|c|c|c|c|c|c|}
\hline
\textbf{index} & \textbf{noise $\epsilon$} & \textbf{covariate $X$} & \textbf{ratio $\lambda$} & \textbf{\# of unmatched data} & \textbf{iterations} \\
\hline
1 & $\mathcal N(0,\sigma_\epsilon^2)$   & $\mathcal N(\mathbf{0}, \sigma_X^2\mathbbm{1}_{3\times 3})$    & $0.2$ and $0.6$ & $200$, $1000$ and $5000$  & 2500\\
\hline
2 & $\mathcal N(0,\sigma_\epsilon^2)$   & $\mathcal N(\mu_X\mathbf{1}_{3},\sigma_X^2\mathbbm{1}_{3\times 3})$    & $0.2$ and $0.6$ & $200$, $1000$ and $5000$  & 2500\\
\hline
3 & $\mathcal N(0,\sigma_\epsilon^2)$   & $\text{U}([ - \sqrt{3}\sigma_X,\sqrt{3}\sigma_X]^3)$    & $0.2$ and $0.6$ & $200$, $1000$ and $5000$  & 2500\\
\hline
4 & $\mathcal N(0,\sigma_\epsilon^2)$   & $\text{U}([\mu_X - \sqrt{3}\sigma_X,\mu_X + \sqrt{3}\sigma_X]^3)$    & $0.2$ and $0.6$ & $200$, $1000$ and $5000$  & 2500\\
\hline
5 & $\text{Laplace}(0,\frac{\sigma_\epsilon}{\sqrt{2}})$   & $\mathcal N(\mathbf{0}, \sigma_X^2\mathbbm{1}_{3\times 3})$   & $0.2$ and $0.6$ & $200$, $1000$ and $5000$  & 2500\\
\hline
6 & $\text{Laplace}(0,\frac{\sigma_\epsilon}{\sqrt{2}})$   & $\mathcal N(\mu_X\mathbf{1}_{3},\sigma_X^2\mathbbm{1}_{3\times 3})$   & $0.2$ and $0.6$ & $200$, $1000$ and $5000$  & 2500\\
\hline
\end{tabular}
\vspace{0.5em} 
\caption{$\mathbf{1}_3 \in \mathbb{R}^3$ denotes the all-ones vector, and $\mathbbm{1}_{3\times 3}$ the identity matrix. We used $\sigma_\epsilon = 0.8\sqrt{10}$, $\mu_X = 5$ and $\sigma_X = 1$.  }\label{Table: simulations}
\end{sidewaystable}

\subsection{Simulation results}

When $X$ is Gaussian, the results are shown in Figures \ref{fig:s1} and \ref{fig:s2}. Figure \ref{fig:s1} corresponds to the case where $\mathbb E(X) = \mu_X = 0$. One can see that the empirical gain approximates quite well the theoretical one for large sample sizes of unmatched data. Note that in this case the gain depends only on the SNR (or more precisely on $\eta$) as derived in (\ref{equation: Gain with mean = 0}). Thus, connecting the points through a linear interpolation gives an accurate illustration of the actual dependence of the gain $G$ on SNR between the points. 

\begin{figure}[h]
    \centering
    \begin{subfigure}{0.48\textwidth}  
        \includegraphics[width=\textwidth]{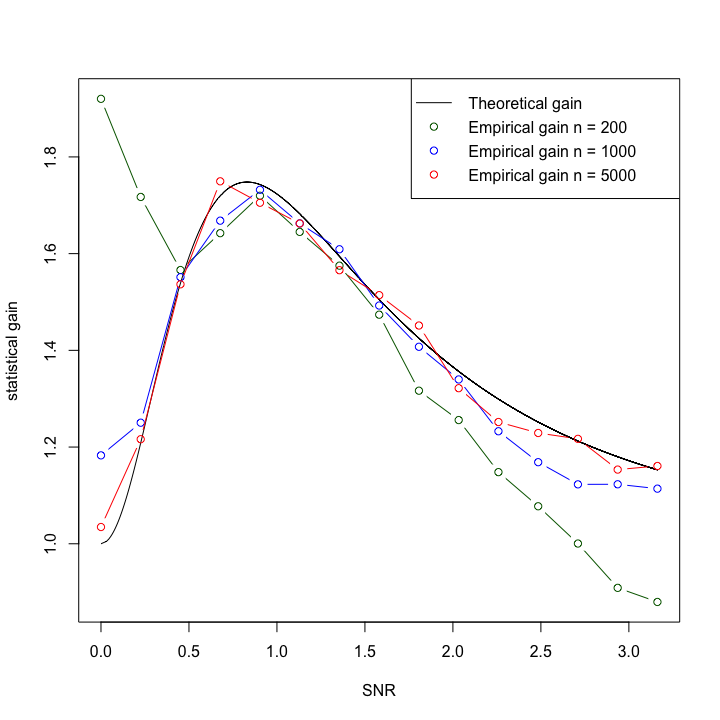} 
        \caption{$\lambda = 0.2$}
    \end{subfigure}
    \hfill
    \begin{subfigure}{0.48\textwidth}
        \includegraphics[width=\textwidth]{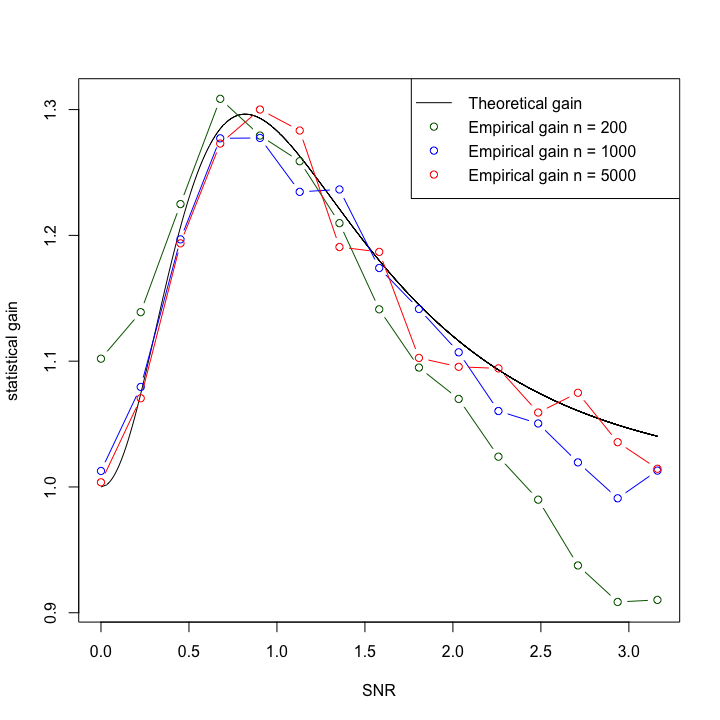} 
        \caption{$\lambda = 0.6$}
    \end{subfigure}
    \caption{Results of simulation with $\epsilon \sim \mathcal{N}(0, (0.8\sqrt{10})^2)$ and $X \sim \mathcal{N}(0,\mathbbm{1}_{3\times 3})$.}
    \label{fig:s1}
\end{figure}

As already mentioned above, in the case where $\mu_X \neq 0$, the gain does not depend solely on the SNR, but also on $\zeta$ and $\rho$ as one can see from the formula in (\ref{equation: Gain with mean neq 0}). Thus, the behavior of $G$ is expected to be different as for the case $\mu_X = 0$.  In Figure \ref{fig:s2}, one can see that the gain drops for the value $\text{SNR} \approx 1.8$ before increasing again at $\text{SNR} =2$. This is not an artifact due to finite sample sizes because the same behavior is exhibited by the theoretical gain.  As $\mu_X$, $\Sigma_X$ and $\rho$ are fixed, the behavior can be explained through the aforementioned dependence of $G$ on $\zeta^2$ which is equal to the squared cosine of the angle between $\beta_0$ and $\mu_X$ since $ \Sigma_X$ is diagonal. Hence, the observed decrease of the gain in this simulation setting is due to the fact that $\beta_0$ and $\mu_X$ are nearly co-linear. We refer the reader to the discussion above for the case $\mu_X \ne 0$ and also to Figure \ref{fig:G(zeta)}.


\begin{figure}[h]
    \centering
    \begin{subfigure}{0.48\textwidth}  
        \includegraphics[width=\textwidth]{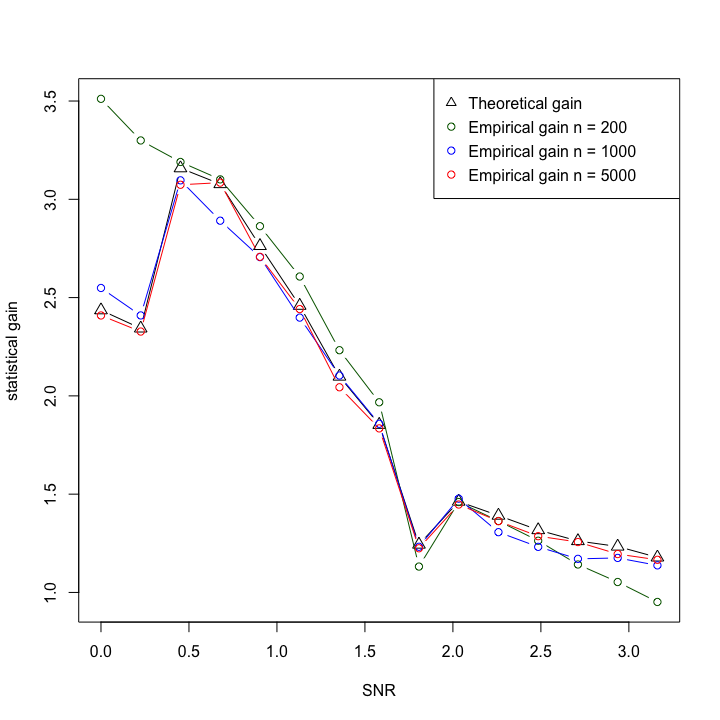} 
        \caption{$\lambda = 0.2$}
    \end{subfigure}
    \hfill
    \begin{subfigure}{0.48\textwidth}
        \includegraphics[width=\textwidth]{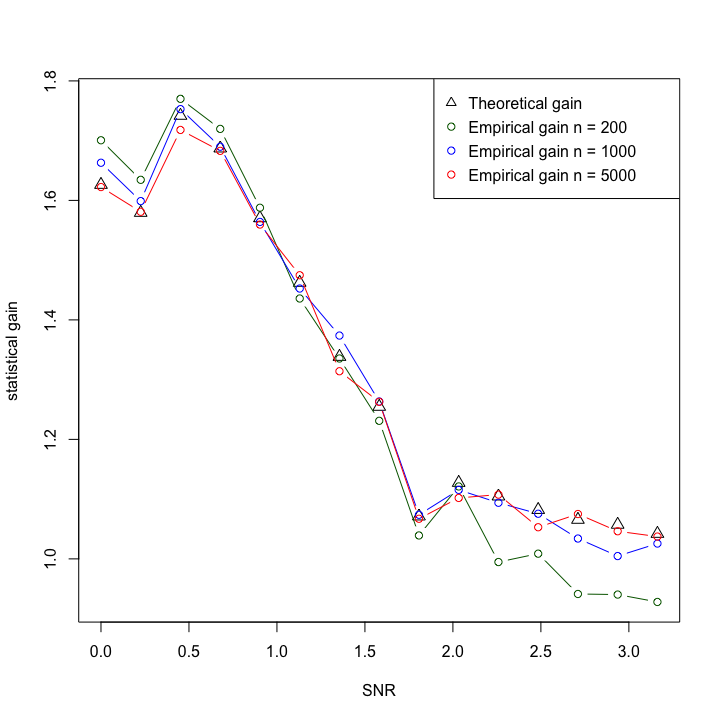} 
        \caption{$\lambda = 0.6$}
    \end{subfigure}
    \caption{Results of simulation with $\epsilon \sim \mathcal{N}(0, (0.8\sqrt{10})^2)$ and $X \sim \mathcal{N}(5 \cdot \mathbf{1}_3,\mathbbm{1}_{3\times 3})$.}
    \label{fig:s2}
\end{figure}


When $X$ follows a uniform distribution, we have no explicit formula for the statistical gain as the expression of $\Gamma_1$ and $\Gamma_2$ are intractable in this case. Figures \ref{fig:s3} and \ref{fig:s4} illustrate curves that are similar to those in the case of a Gaussian covariate but without their theoretical counterpart. With increasing SNR, the statistical gain increases first and then reaches its maximum. Then, with larger SNR, the statistical gain starts to decrease. In both Figures \ref{fig:s3} and \ref{fig:s4}, we can also observe that the gain drops at SNR $\approx 1.8$ to have a local peak at SNR $\approx 2$. Although $\mu_X = 0$ in this case, we suspect that the direction of $\beta_0$ in $\mathbb{R}^3$ plays an additional role when the distribution of $X$ is not Gaussian.

\begin{figure}[h]
    \centering
    \begin{subfigure}{0.48\textwidth}  
        \includegraphics[width=\textwidth]{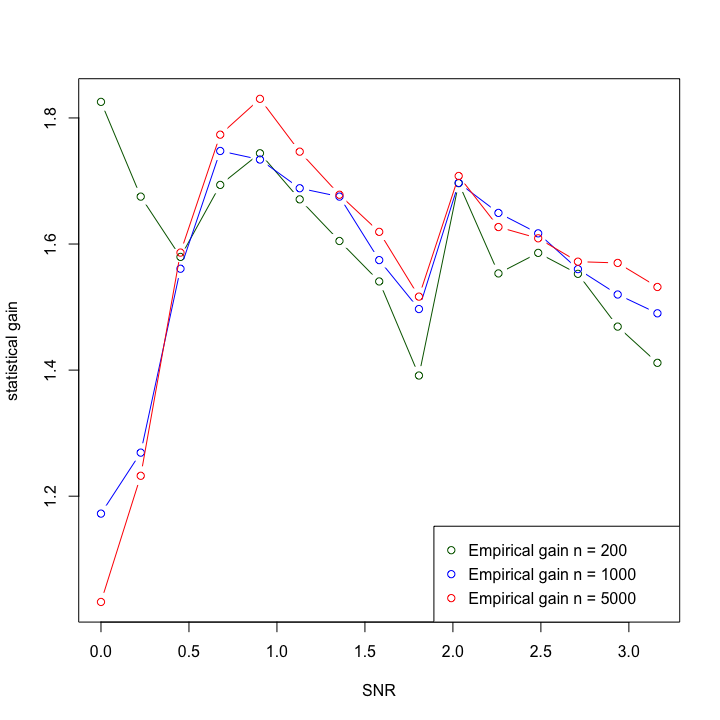} 
        \caption{$\lambda = 0.2$}
    \end{subfigure}
    \hfill
    \begin{subfigure}{0.48\textwidth}
        \includegraphics[width=\textwidth]{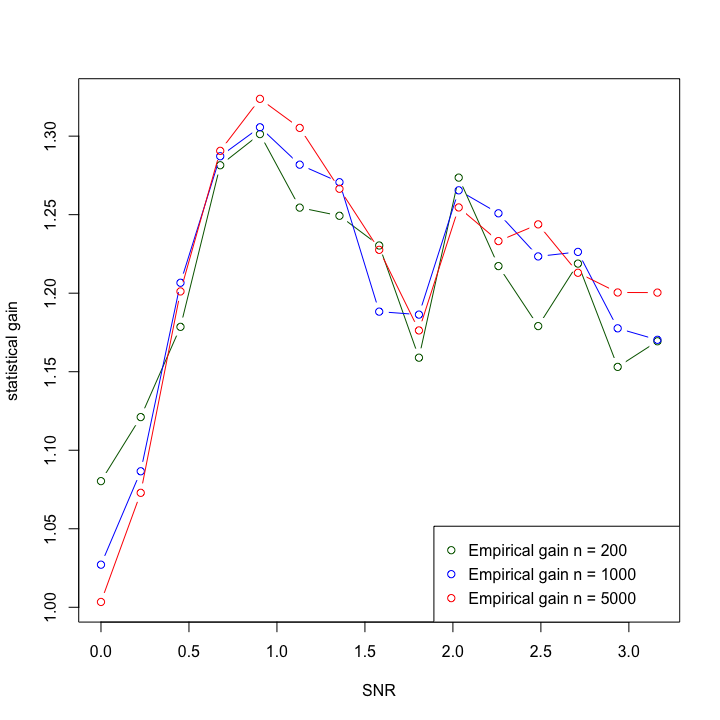} 
        \caption{$\lambda = 0.6$}
    \end{subfigure}
    \caption{Results of simulation with $\epsilon \sim \mathcal{N}(0, (0.8\sqrt{10})^2)$ and $X \sim \text{U}([-\sqrt{3}, \sqrt{3}]^3)$.}
    \label{fig:s3}
\end{figure}

\begin{figure}[h]
    \centering
    \begin{subfigure}{0.48\textwidth}  
        \includegraphics[width=\textwidth]{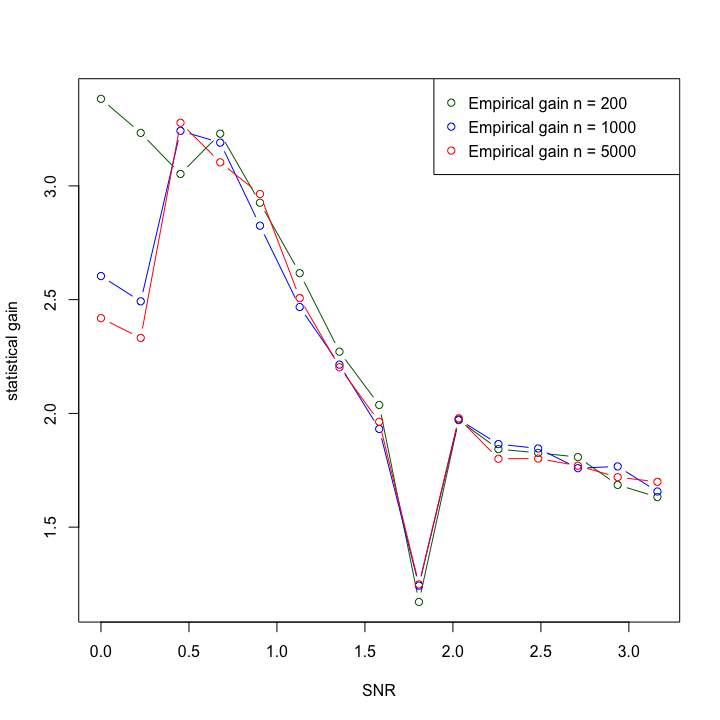} 
        \caption{$\lambda = 0.2$}
    \end{subfigure}
    \hfill
    \begin{subfigure}{0.48\textwidth}
        \includegraphics[width=\textwidth]{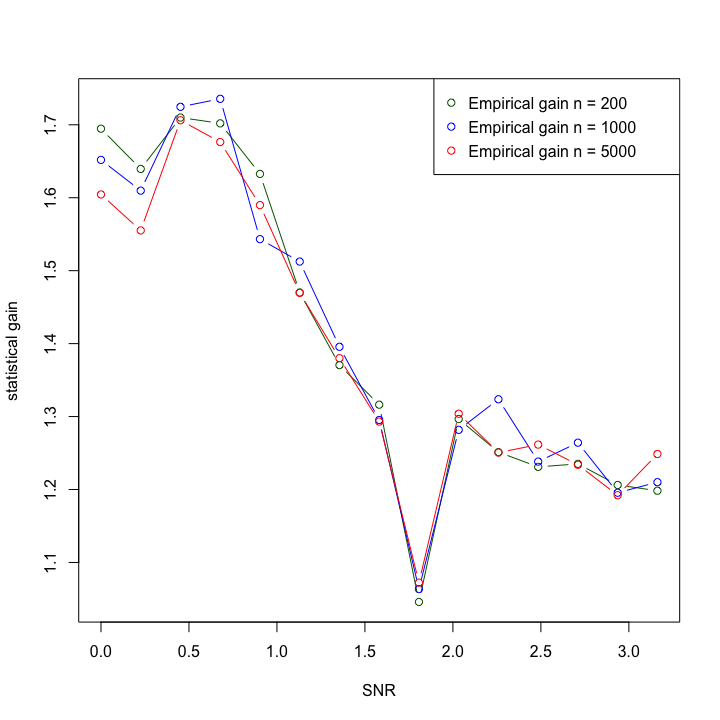} 
        \caption{$\lambda = 0.6$}
    \end{subfigure}
    \caption{Results of simulation with $\epsilon \sim \mathcal{N}(0, (0.8\sqrt{10})^2)$ and $X \sim \text{U}([5-\sqrt{3}, 5+\sqrt{3}]^3)$.}
    \label{fig:s4}
\end{figure}

In the case where $\epsilon$ follows a Laplace distribution, the OLSE is no longer equal to the MLE of the matched sample. Thus, for simulations \#5 and \#6 (see Table \ref{Table: simulations}) we compute not only the empirical gain $\left(\sqrt{\frac{\det{(\widehat{\Sigma}_{\text{SSL}})}}{\det{(\widehat{\Sigma}_{\text{OLS}})}}}\right)^{-1}$, but also $\left(\sqrt{\frac{\det{(\widehat{\Sigma}_{\text{SSL}})}}{\det{(\widehat{\Sigma}_{\text{mMLE}})}}}\right)^{-1}$.  As shown in Figures $\ref{fig:s5}$ and $\ref{fig:s6}$, the curves behave similarly as in the Gaussian noise case. The proposed estimator performs better than both the OLSE and the matched MLE. As expected, the improvement relatively to the OLSE is larger than that to the matched MLE.  Finally, we would like to note that in all simulations the gain is larger for smaller values $\lambda$.  This fact is to be expected as more unmatched data should have a better contribution to the performance of the SSLEMLE. In the Gaussian case, this can be seen explicitly from the formulas of Theorem \ref{gain} which imply that $G$ is monotone increasing in $1/\lambda$ when all the other quantities are fixed. 

\begin{figure}[h]
    \centering
    \begin{subfigure}{0.82\textwidth} 
        \centering
        \includegraphics[width=\textwidth]{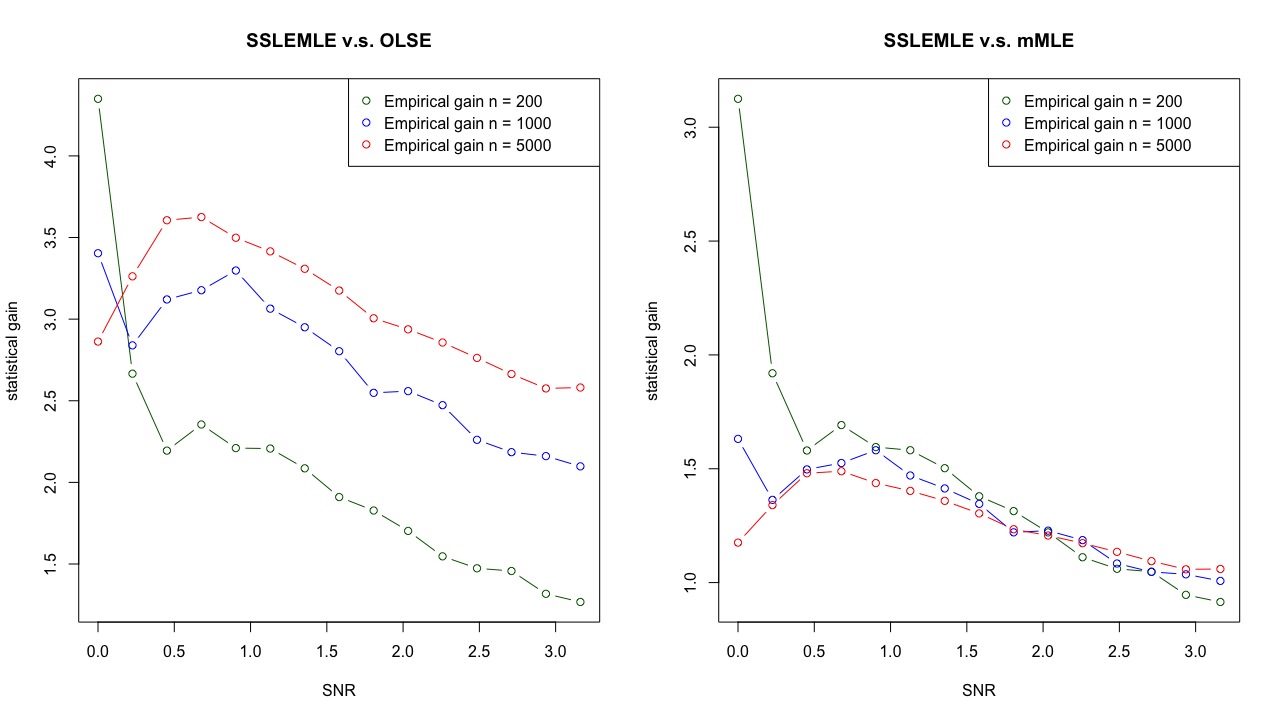}
        \caption{$\lambda = 0.2$}
    \end{subfigure}
    \vspace{0.4em}
    \begin{subfigure}{0.82\textwidth}
        \centering
        \includegraphics[width=\textwidth]{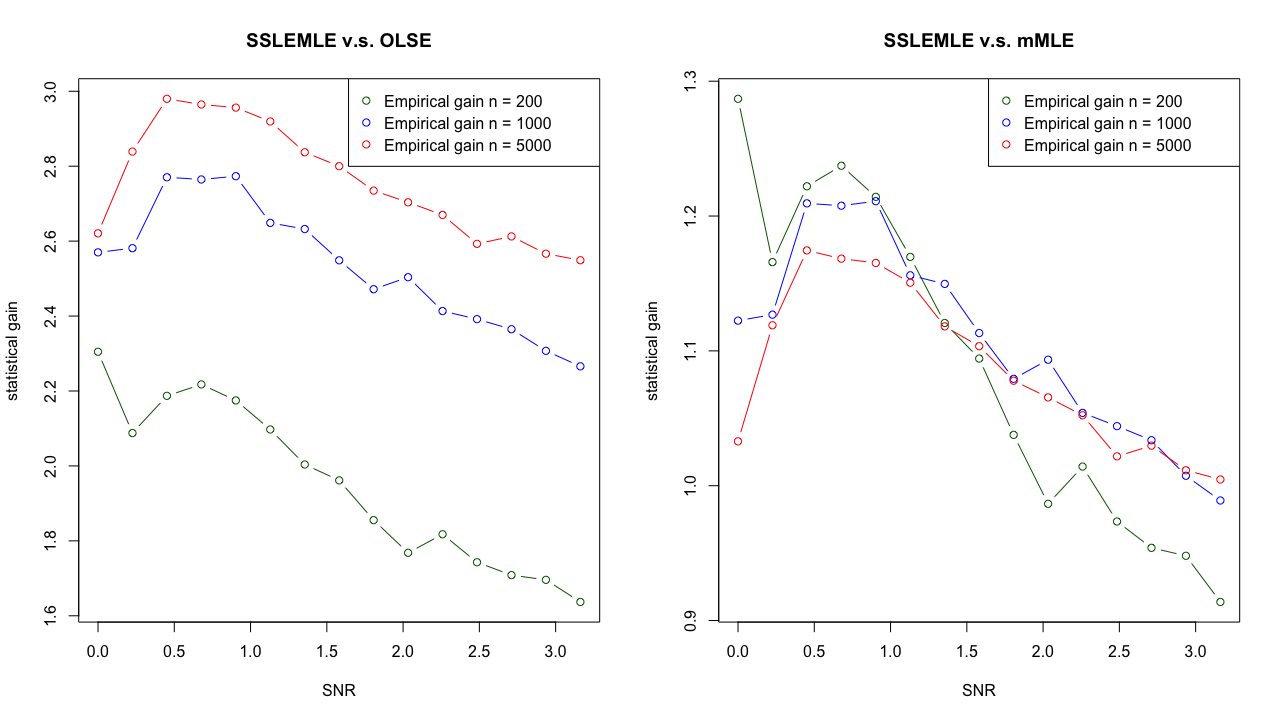}
        \caption{$\lambda = 0.6$}
    \end{subfigure}
\caption{Results of simulation with $\epsilon \sim \text{Laplace}(0, \frac{0.8\sqrt{10}}{\sqrt{2}})$, $X \sim \mathcal{N}(0,\mathbbm{1}_{3\times 3})$.}
    \label{fig:s5}
\end{figure}

\begin{figure}[h]
    \centering
    \begin{subfigure}{0.82\textwidth} 
        \centering
        \includegraphics[width=\textwidth]{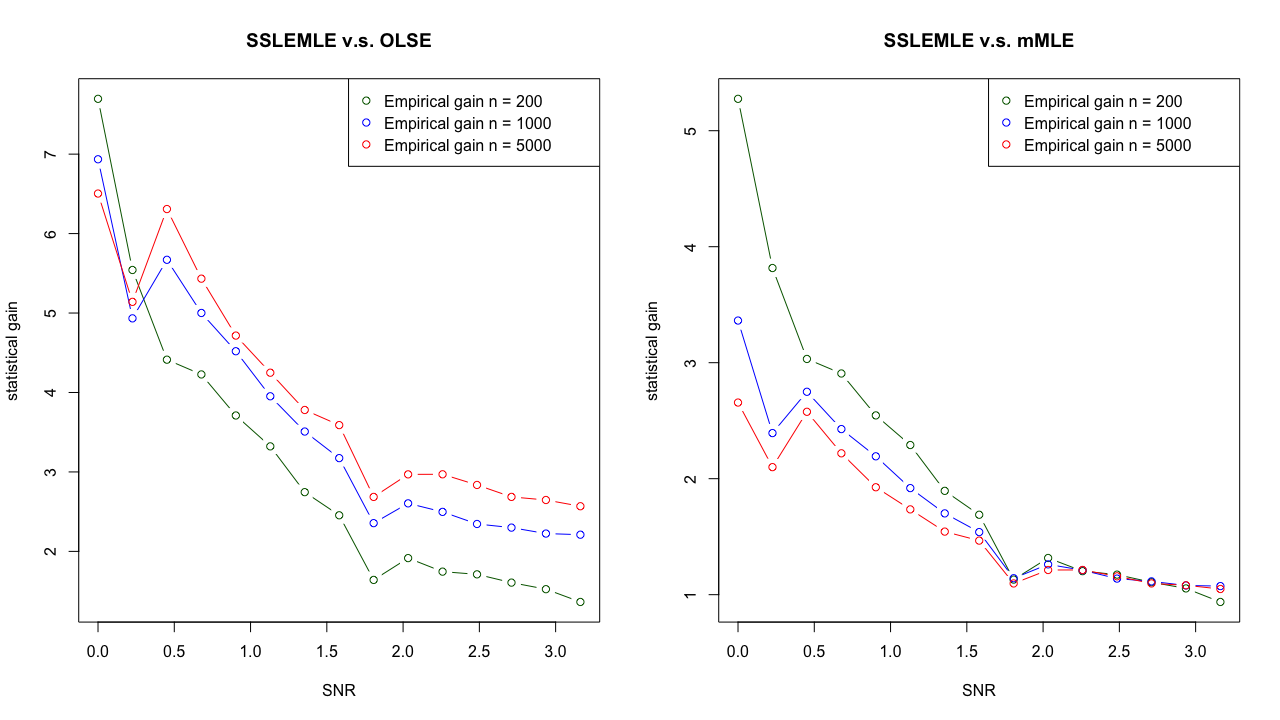}
        \caption{$\lambda = 0.2$}
    \end{subfigure}
    \vspace{0.4em}
    \begin{subfigure}{0.82\textwidth}
        \centering
        \includegraphics[width=\textwidth]{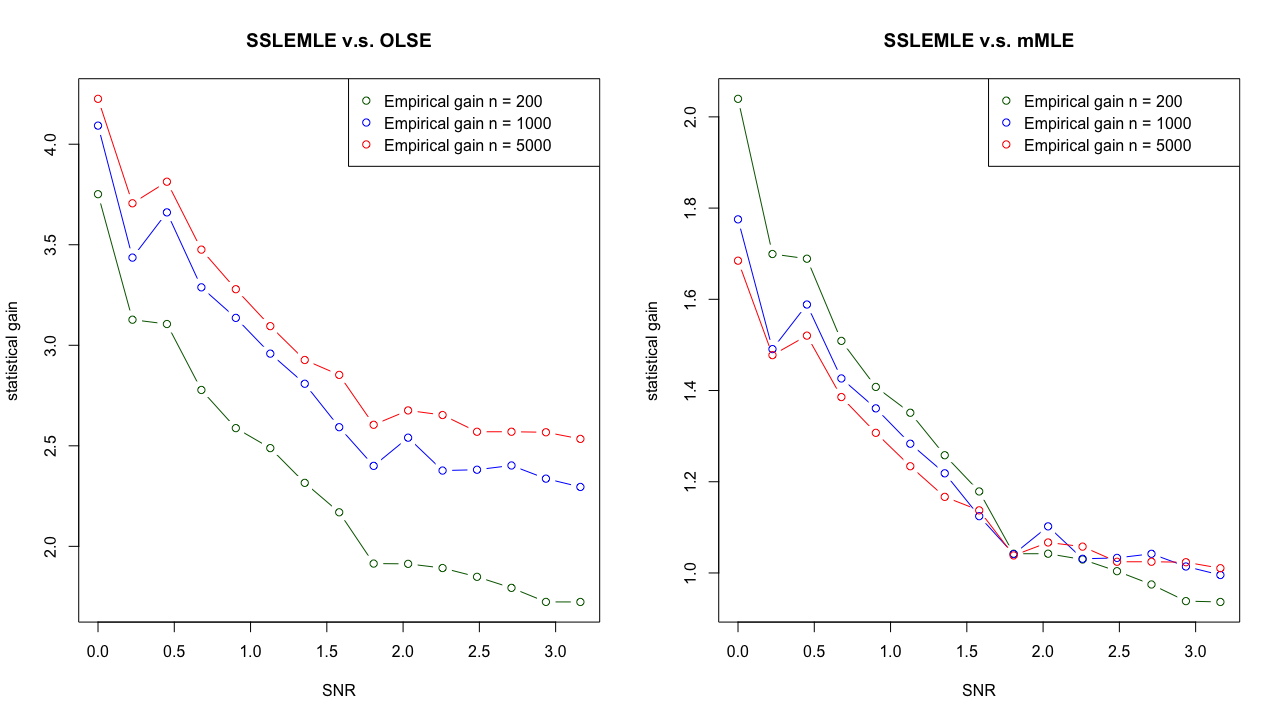}
        \caption{$\lambda = 0.6$}
    \end{subfigure}
     \caption{Results of simulation with $\epsilon \sim \text{Laplace}(0, \frac{0.8\sqrt{10}}{\sqrt{2}})$, $X \sim \mathcal{N}(5 \cdot \mathbf{1}_3,\mathbbm{1}_{3\times 3})$.}
    \label{fig:s6}
\end{figure}

\subsection{Data application}

In this section, we consider the combined cycle power plant data set from \cite{combined_cycle_power_plant_294}, also studied in \cite{monafadoua2024}. This data set consists of 9568 data points with ambient temperature (AT), atmospheric pressure (AP), relative humidity (RH) and exhaust vacuum (V) as the covariates and the net hourly electrical energy output (PE) as the response variable. Assuming that the conditional expectation of PE is a linear function of AT, V, AP and RH, we compute the OLSE with intercept using all the 9568 data points. The obtained multiple $R^2$ value is 0.9287 indicating that the model accounts for a significant part of the variability of the response. As shown in Figure \ref{fig:ccpp_TA_QQ}, the linear model with Gaussian noise seems to be quite suitable. From  the obtained residuals, the standard deviation of the noise can be estimated as $4.558$. Additionally, Figure \ref{fig:ccpp_residual} shows  that $\mathcal{N}(0, 4.558^2)$ provides a very good description of the noise distribution. Therefore, if $\varphi$ denotes the pdf of $\mathcal{N}(0,1)$, we treat the density of the noise as known and equal to $1/4.558\, \varphi(\cdot /4.558)$ in all the subsequent simulations. 

\begin{figure}[h]
    \centering
    \begin{subfigure}{0.49\textwidth}  
        \includegraphics[width=\textwidth]{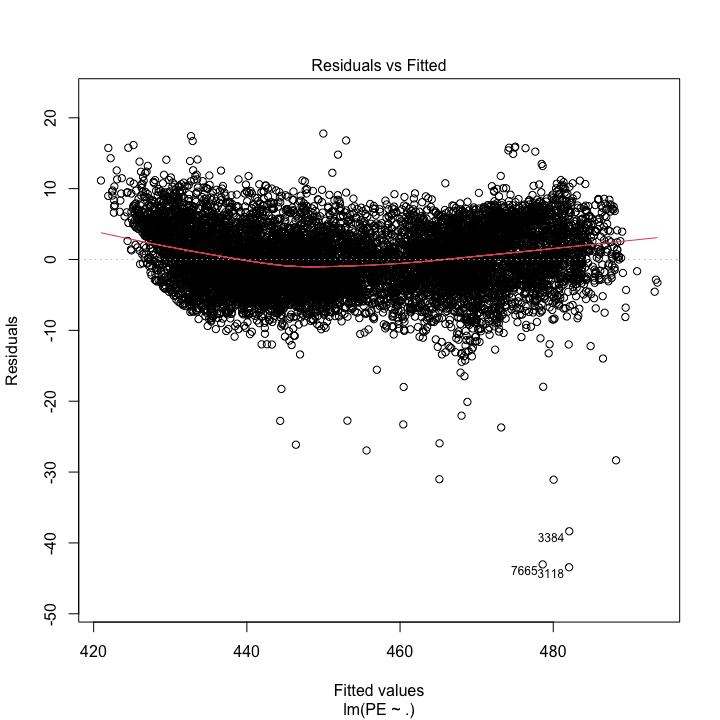} 
        \caption{TA plot}
    \end{subfigure}
    \hfill
    \begin{subfigure}{0.49\textwidth}
        \includegraphics[width=\textwidth]{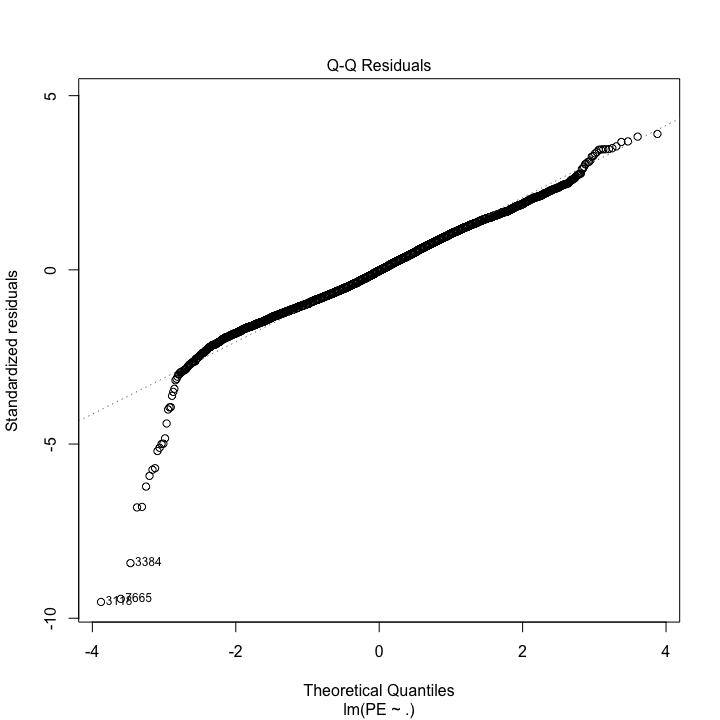} 
        \caption{QQ plot}
    \end{subfigure}
    \caption{The Tukey-Anscombe (TA)- and QQ-plots of the OLSE-fitted model using all 9568 data points from the power plant data set.}
    \label{fig:ccpp_TA_QQ}
\end{figure}

\begin{figure}[h]
    \centering
    \includegraphics[width=0.6\textwidth]{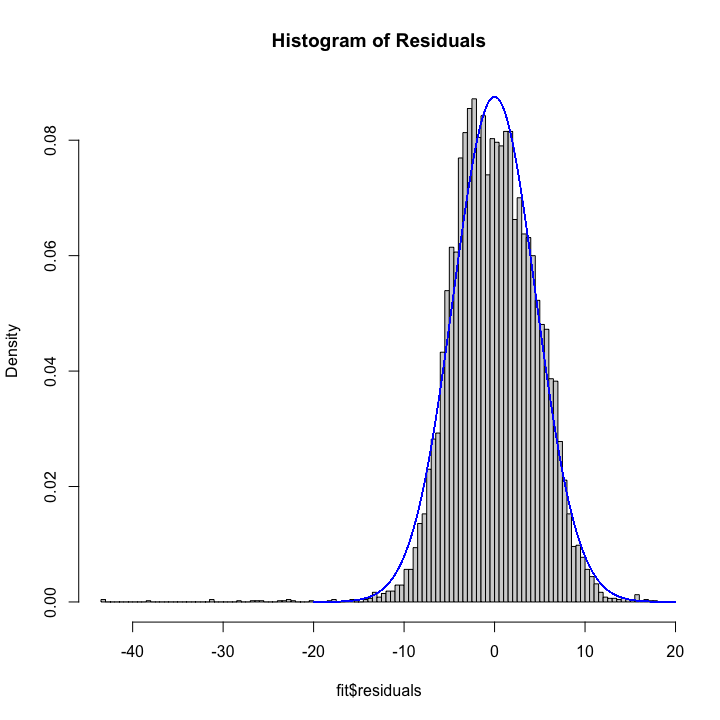} 
    \caption{Histogram of residuals from the OLSE-fitted model with the density function of $\mathcal{N}(0, 4.558^2)$ overlaid (blue).}
    \label{fig:ccpp_residual}
\end{figure}


In what follows, the OLSE will considered as the ground truth. We assume that the link between the response PE and the covariates is unknown, except for a very small subset of data points. In each of 1000 independent replications, we randomly split the whole data set into a training set of 7176 data points and a testing set of 2392 data points. Then, we randomly sample 10 matched data $(m=10)$ from the training set and $n \in \{50,100,200,400,800,1600\}$ unmatched data from the remaining data points in the training set. We compute SSLEMLE $\widehat{\beta}_{n,m}$'s and matched OLSE $\widehat{\beta}_m$'s, and evaluate the predictive performance of these estimators on the testing set by calculating the mean squared errors (MSEs). Note that the matched OLSE and matched MLE are equal in this case since the noise is assumed to be Gaussian.   Table \ref{Table: ccpp_MSE} shows that the predictive performance of SSLEMLE improves as $n$ increases. Figure \ref{fig:ccpp_MSE_ratios}  shows also that the ratio of the average MSE of SSLEMLE to that of the OLSE over the 1000 replications decreases as $n$ increases. Since the true model and the true parameter are unknown in any real data application, it is reasonable to use MSE to evaluate the estimators. A lower MSE in different testing sets generally indicates a better fit and stronger predictive power, and can be interpreted as an improvement gained by using the unmatched data compared to using only a small \lq\lq golden\rq\rq \ matched sample. In addition, Figure \ref{fig: ccpp_boxplots} shows that SSLEMLE becomes increasingly centered as $n$ grows. Note that in this figure, the results for the intercept are not shown as it is of order $\sim 450$, and hence it is hard to display it along with the other four and much smaller coefficients. We observe that the median and mean values of SSLEMLEs and OLSEs begin to show some misalignment as $n$ increases, although the differences are not too large. We suspect that this is due to a combination of model misspecification and minor computational errors in the optimization.

In the simulations we used a model with an intercept. The log-likelihood in this case is given by 
\begin{align*}
\ell_{n,m}(\beta) &=  \frac{1}{n+m}\sum_{j=1}^n \log\left(\frac{1}{n}  \sum_{i=1}^n f(\tilde Y_j - \beta_1 - \beta^\top_{2:(p+1)} \tilde{X}_i)  \right) \\
& \ + \frac{1}{n+m}\sum_{k=1}^m \log f(Y_k - \beta_1 - \beta^\top_{2:(p+1)} X_k)
\end{align*}
for $\beta \in \mathbb R^{p+1}$. Our asymptotic analysis can certainly be extended provided that the covariate admits an absolutely continuous distribution. However, the arguments will be much more involved because the intercept $\beta_1$ cannot be handled in the same way as the remaining coefficients in $\beta_{2:(p+1)}$ (which plays the same role as $\beta$ above). In fact, while the asymptotics for the matched MLE remain the same after replacing $X$ by  $(1, X^\top)^\top$,   $\beta_1$ has to be separated from the remaining coefficients when studying the SSLEMLE. In fact, adding $1$ as a covariate will violate absolute continuity, a very crucial assumption in all the proofs. 


\begin{table}[h]
    \centering
    \begin{tabular}{|c|cccccc|}
    \hline
    $n$  & 50 & 100 & 200 & 400 & 800 & 1600\\
    \hline
    count & 644  & 750 & 828 & 866 & 908  & 934 \\
    \hline
    \end{tabular}
    \vspace{0.5em} 
    \caption{Number of times (out of 1000 replications) that the SSLEMLE achieves lower values of the MSE than the OLSE on the testing set for different values of $n$ with fixed $m=10$.}
    \label{Table: ccpp_MSE}
\end{table}

\begin{figure}[h]
    \centering
    \includegraphics[width=0.65\textwidth]{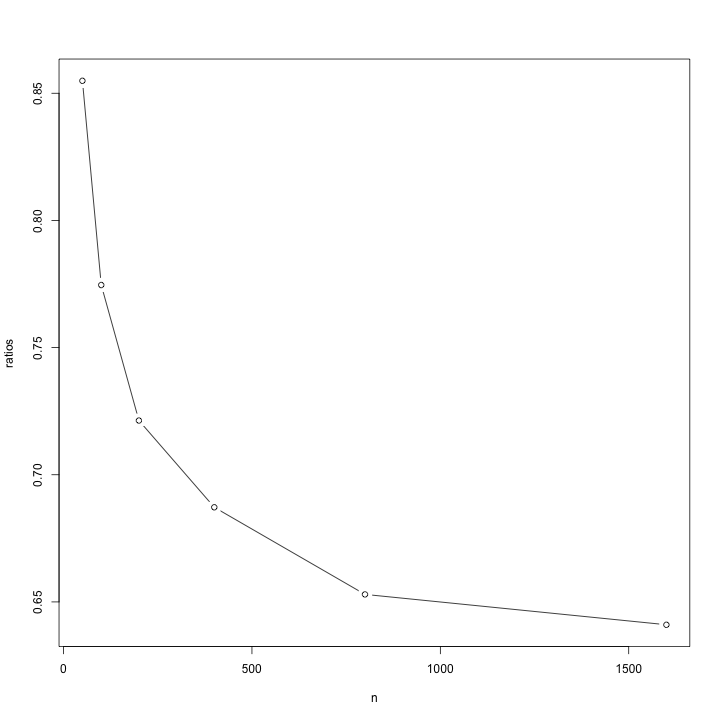} 
    \caption{Ratio of the average MSE of SSLEMLE to that of OLSE, where each average is computed over 1000 replications for different values of $n$ with fixed $m=10$.}
    \label{fig:ccpp_MSE_ratios}
\end{figure}

\begin{figure}[h]
    \centering
    \begin{subfigure}[b]{0.45\textwidth}
        \includegraphics[width=\textwidth]{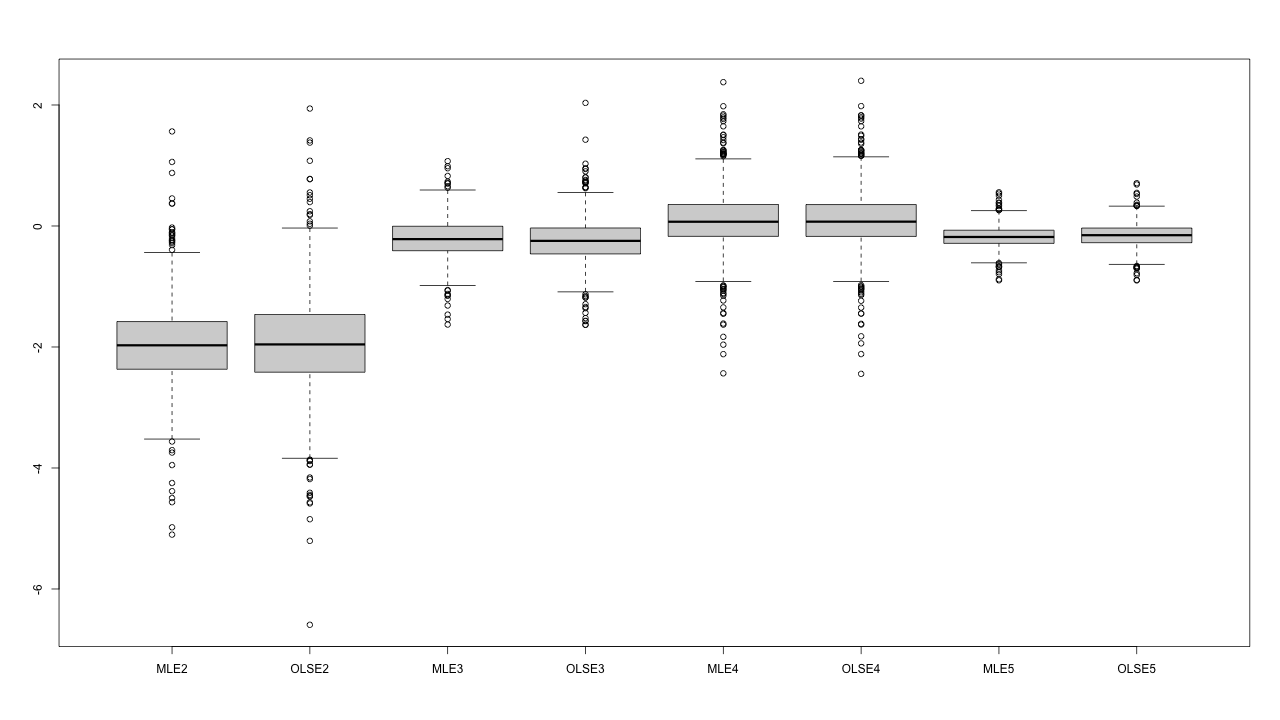}
        \caption{$n=50$}
    \end{subfigure}
    \hfill
    \begin{subfigure}[b]{0.45\textwidth}
        \includegraphics[width=\textwidth]{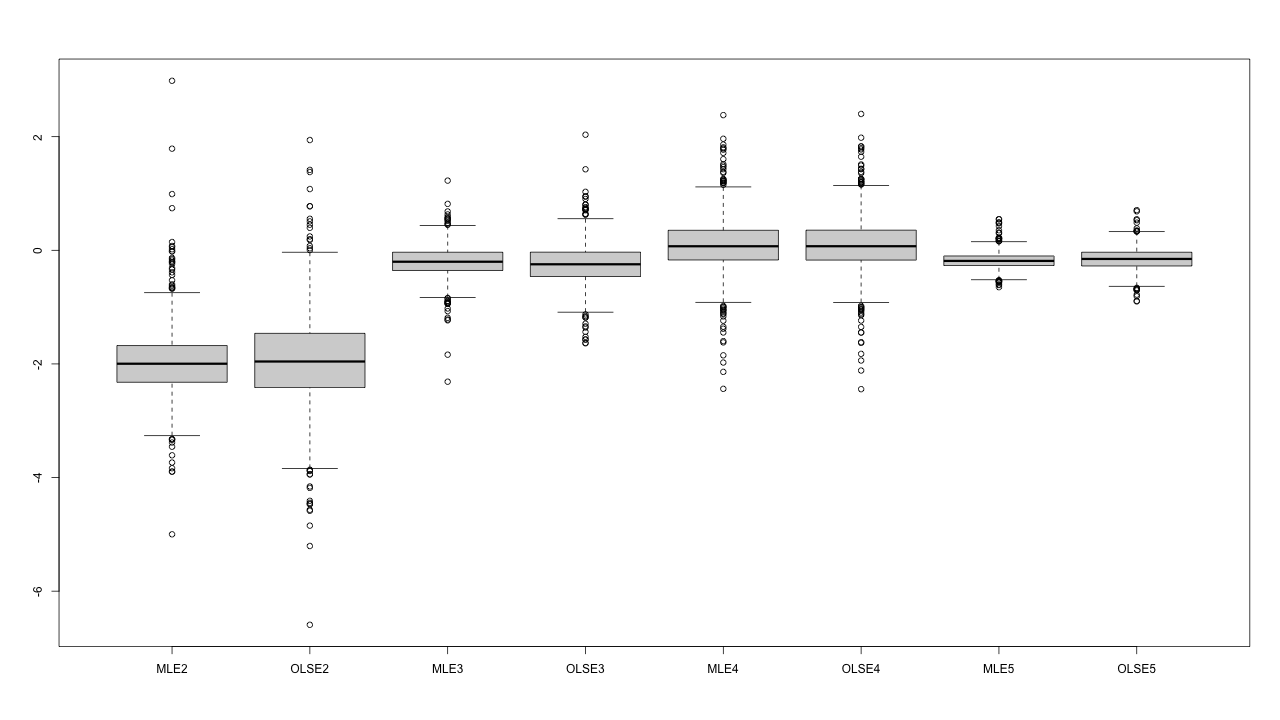}
        \caption{$n=100$}
    \end{subfigure}
    \vskip 0.5cm  
    \begin{subfigure}[b]{0.45\textwidth}
        \includegraphics[width=\textwidth]{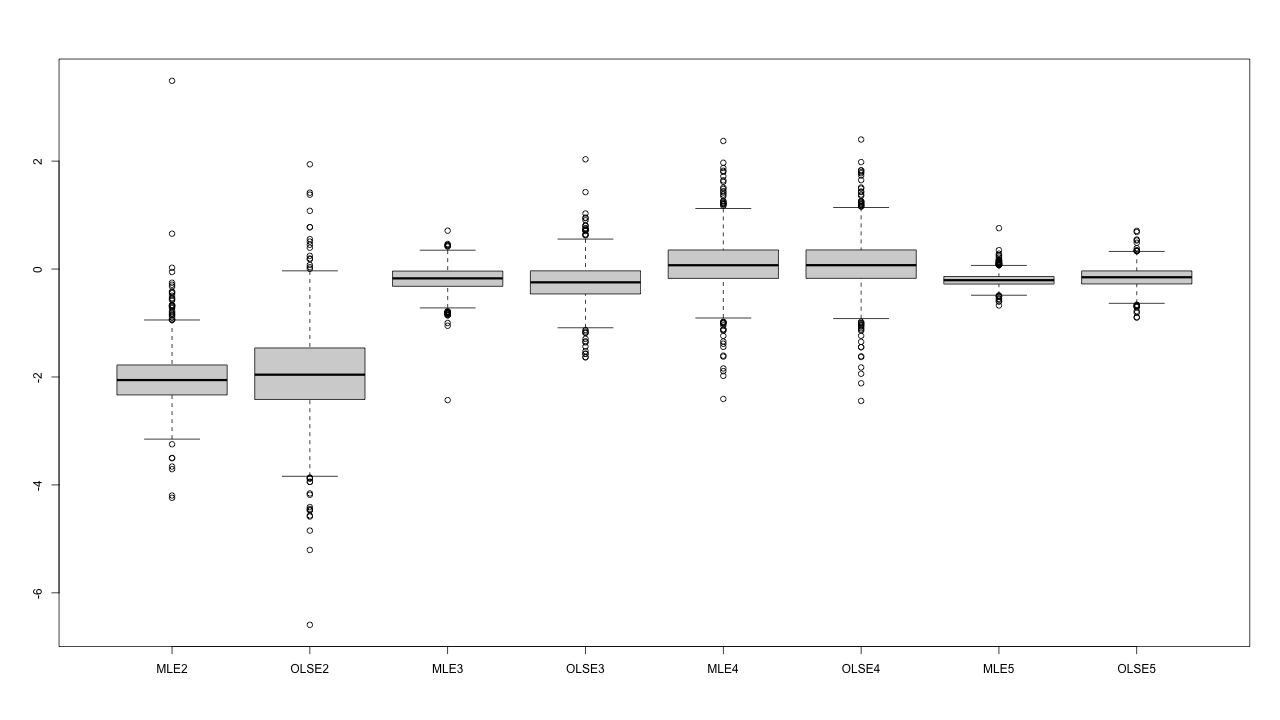}
        \caption{$n=200$}
    \end{subfigure}
    \hfill
    \begin{subfigure}[b]{0.45\textwidth}
        \includegraphics[width=\textwidth]{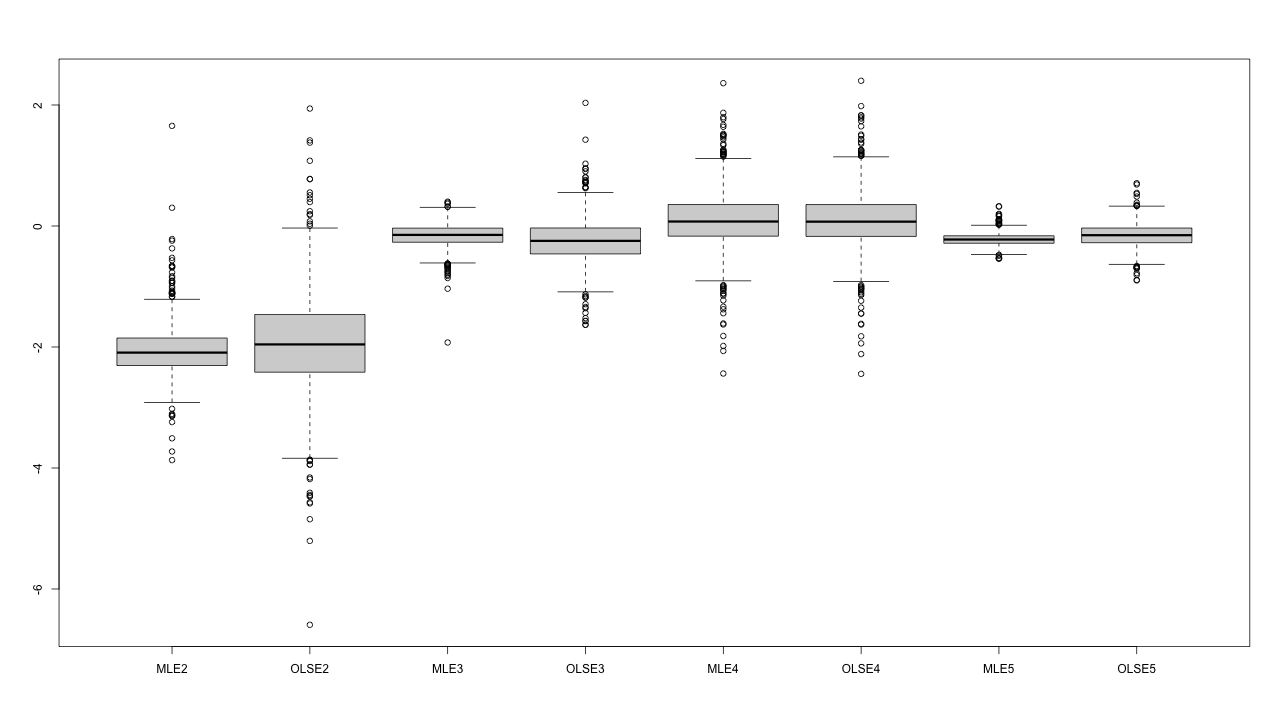}
        \caption{$n=400$}
    \end{subfigure}
    \vskip 0.5cm
    \begin{subfigure}[b]{0.45\textwidth}
        \includegraphics[width=\textwidth]{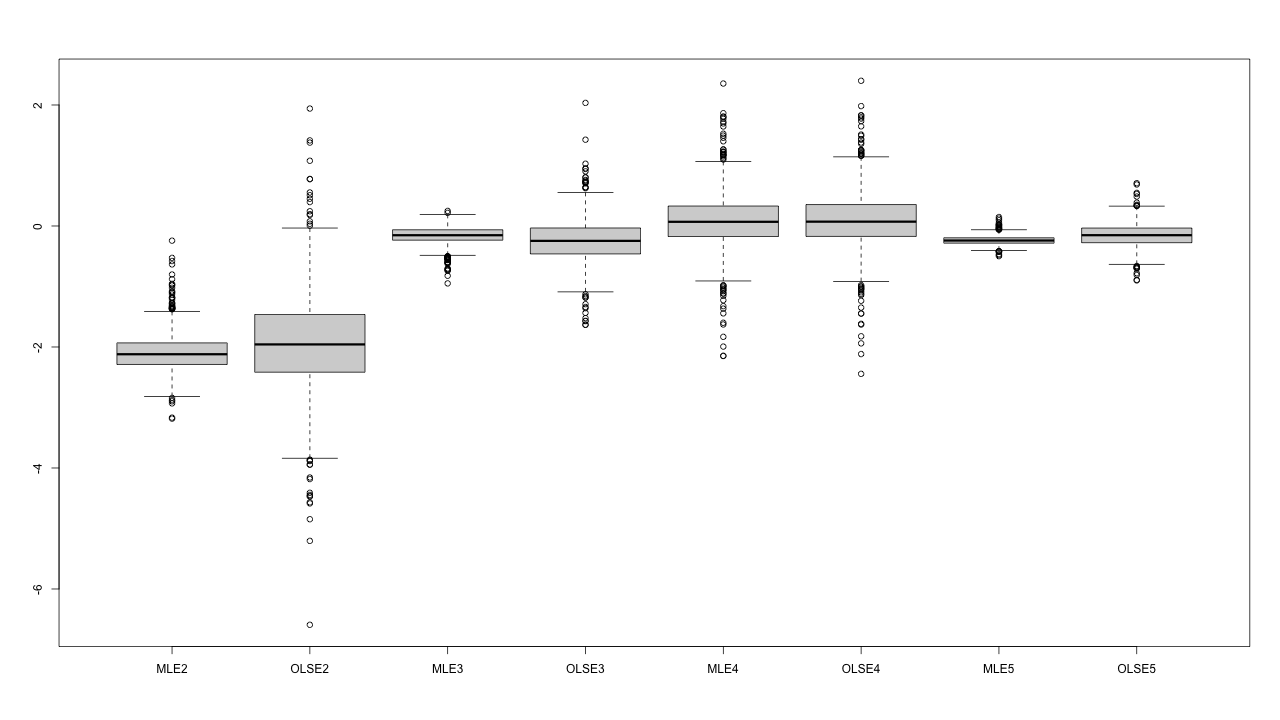}
        \caption{$n=800$}
    \end{subfigure}
    \hfill
    \begin{subfigure}[b]{0.45\textwidth}
        \includegraphics[width=\textwidth]{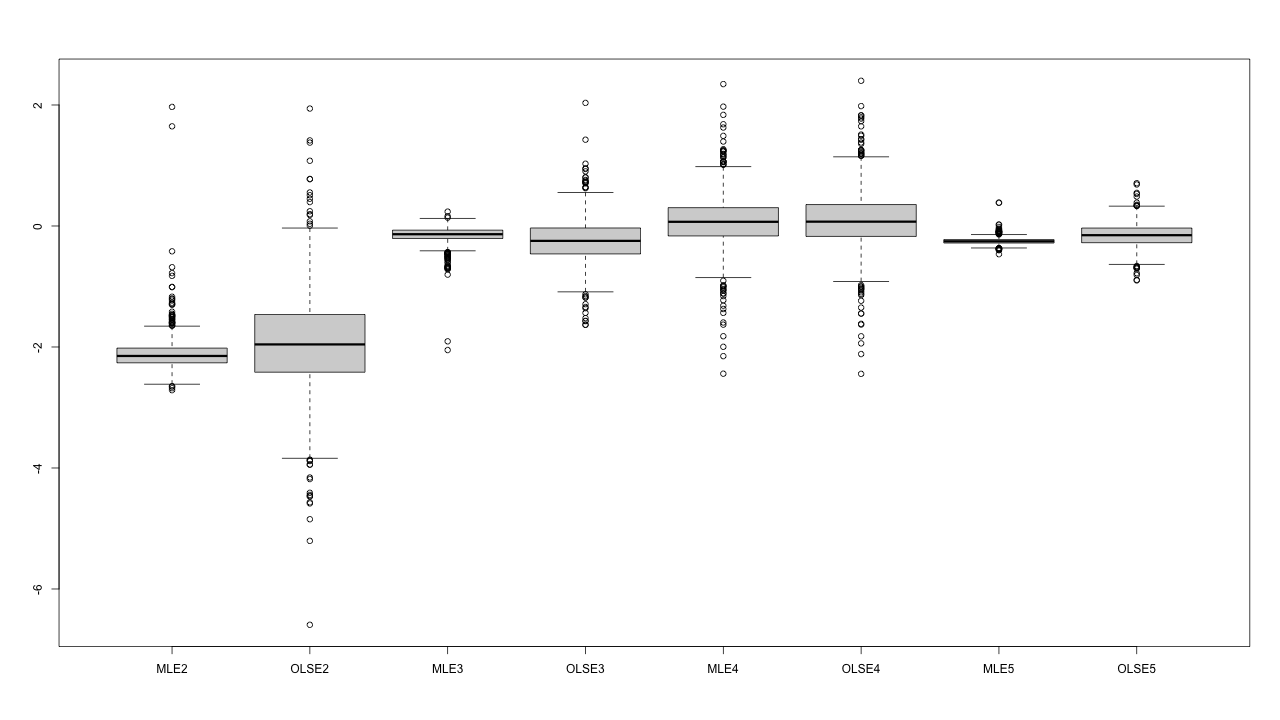}
        \caption{$n=1600$}
    \end{subfigure}
    \caption{Boxplots of SSLEMLE and OLSE coefficients (excluding the intercept) for different values of n.}
    \label{fig: ccpp_boxplots}
\end{figure}


\section{Conclusions and discussion}

In this paper, we introduced an estimator for the regression vector which is based on the empirical likelihood function constructed from both matched and unmatched samples.  The resulting SSLEMLE solves the semi-supervised learning problem in the scope of the linear model. The likelihood function simply combines the matched samples with the unmatched ones by taking the product of the likelihood functions for each part. 

Under the assumption that the ratio of the matched and unmatched sample sizes converges in the limit to some $\lambda \in (0,1)$ we were able to prove  existence and consistency of the proposed estimator and analyze its asymptotic properties. We also compared the proposed estimator to the MLE based only on the matched data and explicitly calculated the statistical gain when $X \sim \mathcal{N}(\mu,\Sigma_X)$ and $\epsilon\sim\mathcal{N}(0,\sigma_\epsilon^2)$. However, the proofs of consistency and asymptotic normality are based on some more strict assumptions.  Assumption (A3) requires that the covariate variable $X$ is compactly supported. Our Theorem \ref{gain} actually violates this assumption. However, we could show that the formulas of the statistical case, obtained by using involved algebra and diagonalization techniques, match very closely with our Monte Carlo simulations and hence are validated numerically.


Assumption (A0) imposes that the noise density is proportional to $\exp(- d^{-\alpha} \vert t \vert^\alpha),$ $\alpha > 0$. Such an assumption can be relaxed at the cost of much more complicated proofs.  Note that (A0) still encompasses many well-known probability models, including Laplace and Gaussian distributions. Assumption (A0) entails that the variance of the noise distribution is known. In practice, this is rarely the case. When the standard deviation of the noise variable is unknown, it needs to be included as a parameter of the combined likelihood function. In this case, the log-likelihood function is given by
\begin{align*}
    \ell_{n, m}(\beta,\sigma) & =  \frac{1}{n+m}  \sum_{j=1}^n \log \left( \frac{1}{n} \sum_{i=1}^n \frac{1}{\sigma}f\left(\frac{\tilde{Y}_j - \beta^\top \tilde{X}_i}{\sigma}\right)\right) \\
    & \ + \ \frac{1}{n+m}  \sum_{k=1}^m \log{\left( \frac{1}{\sigma}f\left(\frac{Y_k - \beta^\top X_k}{\sigma}\right)\right)},
\end{align*}
where $f(t) = \frac{\alpha}{2\Gamma(\frac{1}{\alpha})}\sqrt{\frac{\Gamma(3/\alpha)}{\Gamma(1/\alpha)}}\exp{\left(-\lvert t\rvert^\alpha \left(\frac{\Gamma(3/\alpha)}{\Gamma(1/\alpha)}\right)^\frac{\alpha}{2}\right)}$ is the standardized density function of the noise variable with variance of 1. 
In a future work, one may focus on studying this more complex problem. In practice, estimating the unknown variance using the matched data could be an easier approach since a natural estimate is the standard deviation of the residuals.  Alternatively, the obtained estimate can be used as an initial value for maximizing the new log-likelihood. 

It is worth noting that one important and tacit assumption in the problem we consider is that the parameters in the unmatched regression model are assumed to be the same as in the matched one. In fact, we assume that the covariates in the unmatched sample have the same distribution as those in the matched one. Also, it is assumed that we have the same regression vector $\beta_0$ in both models as well as the same noise distribution.  Thus, it is crucial that the unmatched sample reasonably reflects similar dependencies as in the matched one. Formal tests can be conducted in order to check whether this basic assumption is possibly violated.  For example,  one can first test whether the covariates and responses in the unmatched and matched samples have the same distributions.

Finally, we think that the theoretical findings of this paper may be extended to more general regression settings. For instance, one can consider combining matched and unmatched samples in the scope of a logistic regression model. In this case, the corresponding SSLEMLE is obtained by by maximizing the log-likelihood 
\begin{eqnarray}\label{logliklogistic}
&&\ell_{n, m}(\beta)\notag \\ 
&& = \frac{1}{n+m} \sum_{j=1}^n  \left \{ \tilde Y_j \log \left( \int \frac{e^{\beta^\top x}}  {1+e^{\beta^\top x}} d\mathbb F^{\tilde X}_n(x)\right) +   (1-\tilde Y_j) \log \left( \int \frac{1}{1+e^{\beta^\top x}} d\mathbb F^{\tilde X}_n(x)\right) \right \} \notag \\
&& \ \ + \ \frac{1}{n+m} \sum_{k=1}^m \left \{ Y_k \log\left(  \frac{e^{\beta^\top  X_k}}{1+e^{\beta^\top X_k}} \right)  + (1-Y_k) \log\left(\frac{1}{1+e^{\beta^\top X_k}} \right) \right\}. \notag \\
&&
\end{eqnarray}
where $\mathbb F^{\tilde X}_n$ is the empirical distribution of the unmatched covariates $\tilde X_i, i=1, \ldots, n$, $\tilde Y_j \in 
\{0,1\}, j =1, \ldots, n$ are the unmatched responses, and $(X_k, Y_k) \in \mathbb R^p \times \{0,1\}, \ k=1, \ldots, m$ the pairs in the matched sample.  While the matched part in $\ell_{n, m}(\beta)$ corresponds to the classical log-likelihood in a logistic regression model, the unmatched comes from writing the marginal distribution of $\tilde Y_j$:  For $\delta \in \{0,1\}$ we have that
\begin{eqnarray*}
\mathbb P(\tilde Y_j = \delta)  & = &  \int \mathbb P(\tilde Y_j = \delta | \tilde X = x) dF^{\tilde X}(x)   \\
& = &  \int \left(\frac{e^{\beta_0^\top x}}{1 + e^{\beta_0^\top x}}  \right)^\delta  \left(\frac{1}{1 + e^{\beta_0^\top x}}  \right)^{1-\delta} dF^{\tilde X}(x)
\end{eqnarray*}
where $F^{\tilde X}$ is the true distribution of  $\tilde X \stackrel{d}{=} \tilde X_i, i =1, \ldots, n$. As in the linear model investigated here, a maximizer of the log-likelihood in (\ref{logliklogistic}) for logistic regression is expected to be associated with a better performance than the MLE based on the matched sample alone. In the very simple setting of dimension 1, we computed the SSLEMLE for Gaussian covariates with mean and variance both equal to $1$ and true regression coefficient $\beta_0 =2$. In Figure \ref{fig: gainlogistic}, we plot the estimated statistical gain based on 100 replications versus $\log_{10}(m/n)$, where $m=100$ is held fixed and $n \in \{100, 500, 1000, 5000, 10 000, 50 000, 100 000 \}$. The magnitude of improvement for large $n$ or equivalently small ratios $m/n$ is certainly quite promising. 
\begin{figure}[!h]
    \centering
    \begin{subfigure}{0.48\textwidth} 
        \centering
        \includegraphics[width=\textwidth]{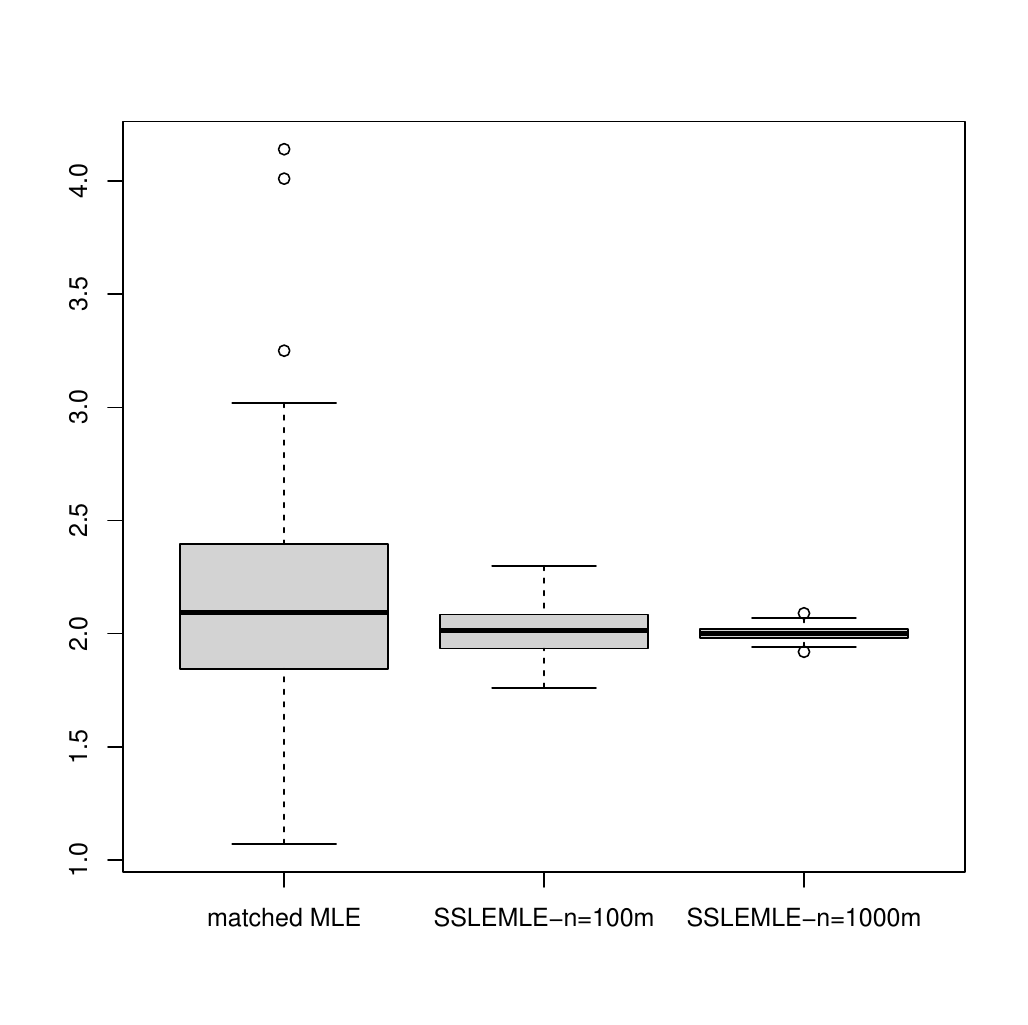}
        \caption{}
    \end{subfigure}
    \vspace{0.4em}
    \begin{subfigure}{0.48\textwidth}
        \centering
        \includegraphics[width=\textwidth]{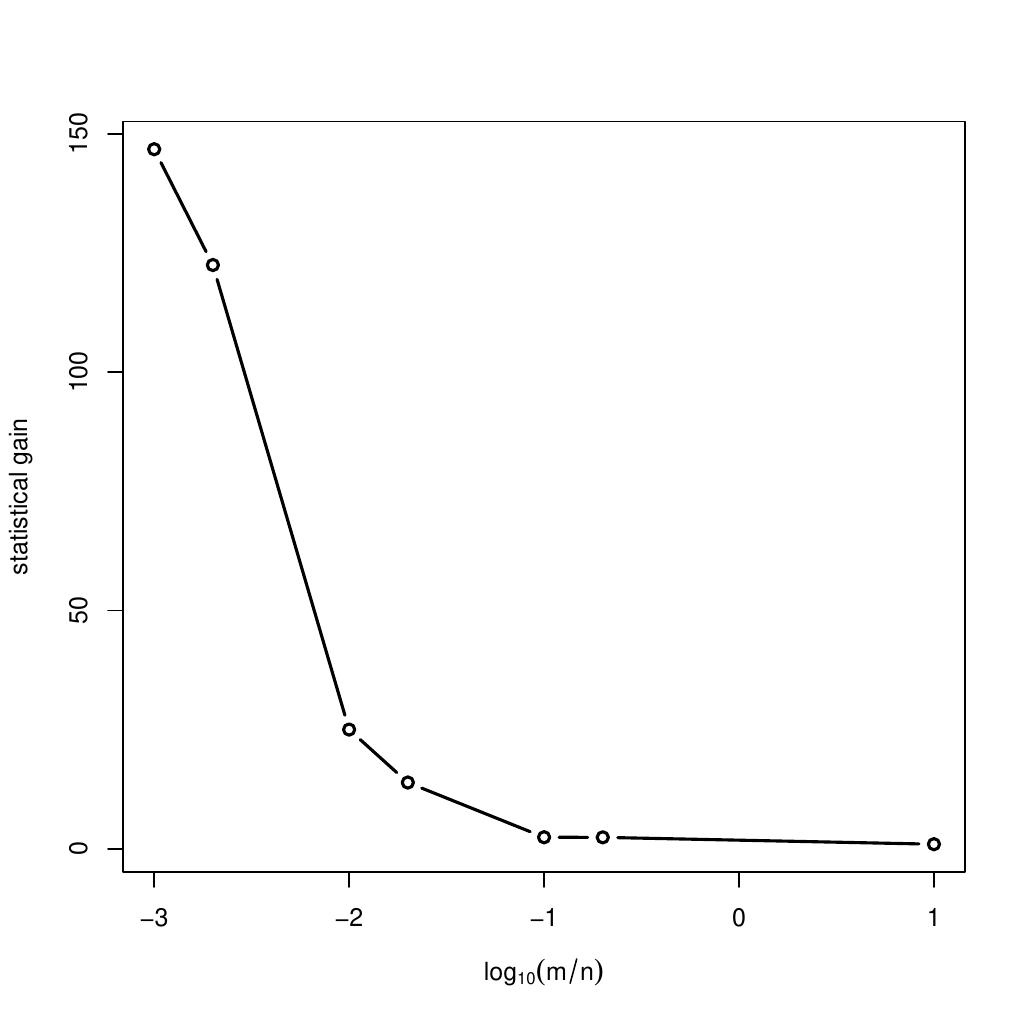}
        \caption{}
    \end{subfigure}
     \caption{ Boxplots of the matched MLE and SSLEMLE's for $n =100m$ and $n=1000m$ (a) and the empirical statistical gain for the SSLEMLE over the matched MLE (b) in a univariate logistic regression model with Gaussian covariates with mean and variance equal to $1$. The true regression coefficient is $\beta_0 =2$ and the number of replications used to estimate the gain was chosen to be 100.}
    \label{fig: gainlogistic}
\end{figure}
Investigating the asymptotic properties of the SSLEMLE and the associated statistical gain may require in this case developing different techniques from empirical process theory as those employed for the linear model. This will be studied in detail in the scope of another research work.

\section*{Appendix A: Proofs of the main results}

\subsection{Proofs for Section \ref{SSLEMLE}}

\begin{proof}[Proof of Lemma \ref{Lemma: Existence of MLE for fixed sample case}]
        Under the assumption that the density function of $\epsilon$ has the form $f(t) = c_\alpha\exp{(-d^{-\alpha} \left\lvert t\right\rvert^\alpha)}$, we have that
\begin{align*}
-\ell_{n, m}(\beta) &= \frac{1}{n+m} \left(-\sum_{j=1}^n \log{\left(\sum_{i=1}^n \frac{1}{n}f^\epsilon(\Tilde{Y}_j - \beta^\top \Tilde{X}_i) \right)} - \sum_{k=1}^m \log{f^\epsilon(Y_k - \beta^\top X_k)}\right) \\
    &= \frac{1}{n+m} \left(-\sum_{j=1}^n \log{\left(\sum_{i=1}^n \frac{1}{n} c_\alpha \exp{(-d^{-\alpha} \lvert \Tilde{Y}_j - \beta^\top \Tilde{X}_i \rvert ^\alpha)}\right)} - \sum_{k=1}^m \log{\left(c_\alpha \exp{(-d^{-\alpha} \lvert Y_k - \beta^\top X_k \rvert ^\alpha)}\right)}\right) \\
    &= \frac{1}{n+m} \left(-\sum_{j=1}^n \log{\left(\sum_{i=1}^n \frac{c_\alpha}{n}  \exp{(-d^{-\alpha} \lvert \Tilde{Y}_j - \beta^\top \Tilde{X}_i \rvert ^\alpha)}\right)} - m \log{c_\alpha} + d^{-\alpha} \sum_{k=1}^m \lvert Y_k - \beta^\top X_k \rvert ^\alpha\right) \\
\end{align*}
We have $-d^{-\alpha} \lvert \Tilde{Y}_j - \beta^\top \Tilde{X}_i \rvert ^\alpha \leq 0$. Then it follows $0 < \exp{(-d^{-\alpha} \lvert \Tilde{Y}_j - \beta^\top \Tilde{X}_i \rvert ^\alpha)} \leq 1$. Thus, it holds that
\[
-\sum_{j=1}^n \log{\left(\sum_{i=1}^n \frac{c_\alpha}{n}  \exp{(-d^{-\alpha} \lvert \Tilde{Y}_j - \beta^\top \Tilde{X}_i \rvert ^\alpha)}\right)} \geq -n\log{c_\alpha}
\]
and hence
\[
-\ell_{n, m}(\beta) \geq - \log{c_\alpha} + \frac{d^{-\alpha}}{n+m}\sum_{k=1}^m \lvert Y_k - \beta^\top X_k \rvert ^\alpha
\]
Define $A^\ast = \inf_{u \in \mathcal{S}^{p-1}} \sum_{k=1}^m \left\lvert u^\top X_k\right\rvert^\alpha$ with $\mathcal{S}^{p-1} = \{u \in \mathbb{R}^p: \left\lVert u\right\rVert_2=1\}$. $A^\ast$ exists and is attained on $\mathcal{S}^{p-1}$ since the latter is a compact subset of $\mathbb{R}^p$ and the function $u \mapsto \sum_{k=1}^m \left\lvert u^\top X_k\right\rvert^\alpha$ is continuous. Since $\operatorname{rank}(M) = p$,  $Mu = 0$ if and only if $u = 0$. This means that $\sum_{k=1}^m \left\lvert u^\top X_k\right\rvert^\alpha = \left\lVert Mu\right\rVert_\alpha^\alpha= 0$ if and only if $u = 0$. So we have $A^\ast > 0$. 
Now, for all $\beta\in\{b \in \mathbb{R}^p: \left\lVert b\right\rVert > R\}$ with 
\begin{eqnarray*}
R := \left( \frac{2^{\alpha-1}}{A^\ast} \sum_{j=1}^n \left\lvert \Tilde{Y}_j \right\rvert^\alpha +\frac{2^{\alpha}}{A^\ast} \sum_{k=1}^m \left\lvert Y_k \right\rvert^\alpha \right)^{\frac{1}{\alpha}}
\end{eqnarray*}
we have that
\begin{align*}
    \left\lVert\beta\right\rVert^\alpha >  \frac{2^{\alpha-1}}{A^\ast} \sum_{j=1}^n \left\lvert \Tilde{Y}_j \right\rvert^\alpha +\frac{2^{\alpha}}{A^\ast} \sum_{k=1}^m \left\lvert Y_k \right\rvert^\alpha 
\end{align*}
and therefore
\begin{align*}
    \left\lVert\beta\right\rVert^\alpha A^\ast >  2^{\alpha-1}\sum_{j=1}^n \left\lvert \Tilde{Y}_j \right\rvert^\alpha + 2^{\alpha} \sum_{k=1}^m \left\lvert Y_k \right\rvert^\alpha
\end{align*}
Since $\beta = \left\lVert\beta\right\rVert u_{\beta}$ for some $u_\beta \in \mathcal{S}^{p-1}$, we have then 
\begin{align*}
    \sum_{k=1}^m \left\lvert \beta^\top X_k\right\rvert^\alpha &= \sum_{k=1}^m \left\lvert \left\lVert\beta\right\rVert u_\beta^\top X_k\right\rvert^\alpha \\
    &= \left\lVert\beta\right\rVert^\alpha \sum_{k=1}^m \left\lvert u_\beta^\top X_k\right\rvert^\alpha \geq \left\lVert\beta\right\rVert^\alpha \inf_{u \in \mathcal{S}^{p-1}} \sum_{k=1}^m \left\lvert u^\top X_k\right\rvert^\alpha  =\left\lVert\beta\right\rVert^\alpha A^\ast
\end{align*} 
which implies that
\begin{align*}
    \sum_{k=1}^m \left\lvert \beta^\top X_k\right\rvert^\alpha > 2^{\alpha-1}\sum_{j=1}^n \left\lvert \Tilde{Y}_j \right\rvert^\alpha + 2^{\alpha} \sum_{k=1}^m \left\lvert Y_k \right\rvert^\alpha
\end{align*}
or equivalently
\begin{align*}
    2^{1-\alpha}\sum_{k=1}^m \left\lvert \beta^\top X_k\right\rvert^\alpha > \sum_{j=1}^n \left\lvert \Tilde{Y}_j \right\rvert^\alpha + 2\sum_{k=1}^m \left\lvert Y_k \right\rvert^\alpha.
\end{align*}
It follows that
\begin{align*}
    \sum_{k=1}^m  \Big(2^{1-\alpha}\left\lvert\beta^\top X_k\right\rvert^\alpha - \left\lvert Y_k \right\rvert^\alpha \Big) > \sum_{j=1}^n \left\lvert \Tilde{Y}_j \right\rvert^\alpha + \sum_{k=1}^m \left\lvert Y_k \right\rvert^\alpha.
\end{align*}
Convexity of the function $u \mapsto \vert u \vert^\alpha$ implies that $\left\vert Y_k - \beta^\top X_k\right\rvert^\alpha \geq 2^{1-\alpha}\left\lvert\beta^\top X_k\right\rvert^\alpha - \left\lvert Y_k \right\rvert^\alpha$, and therefore
\begin{align*}
    \sum_{k=1}^m  \left\vert Y_k - \beta^\top X_k\right\rvert^\alpha > \sum_{j=1}^n \left\lvert \Tilde{Y}_j \right\rvert^\alpha + \sum_{k=1}^m \left\lvert Y_k \right\rvert^\alpha.
\end{align*}
This is equivalent to writing that
\begin{align*}
    - \log{c_\alpha} + \frac{d^{-\alpha}}{n+m}\sum_{k=1}^m \lvert Y_k - \beta^\top X_k \rvert ^\alpha > - \log{c_\alpha} + \frac{d^{-\alpha}}{n+m}\left(\sum_{j=1}^n \left\lvert \Tilde{Y}_j \right\rvert^\alpha + \sum_{k=1}^m \left\lvert Y_k \right\rvert^\alpha\right).
\end{align*}
On the other hand, we have that
\begin{align*}
    -\ell_{n, m}(0) &= \frac{1}{n+m} \left(-\sum_{j=1}^n \log{\left(\sum_{i=1}^n \frac{c_\alpha}{n}  \exp{(-d^{-\alpha} \lvert \Tilde{Y}_j \rvert ^\alpha)}\right)} - m \log{c_\alpha} + d^{-\alpha} \sum_{k=1}^m \lvert Y_k \rvert ^\alpha\right) \\
    &=\frac{1}{n+m} \left(-\sum_{j=1}^n \left(\log{c_\alpha}-d^{-\alpha} \lvert \Tilde{Y}_j \rvert ^\alpha \right)- m \log{c_\alpha} + d^{-\alpha} \sum_{k=1}^m \lvert Y_k  \rvert ^\alpha\right) \\
    &= \frac{1}{n+m} \left(- n \log{c_\alpha}+\sum_{j=1}^n d^{-\alpha} \lvert \Tilde{Y}_j \rvert ^\alpha - m \log{c_\alpha} + d^{-\alpha} \sum_{k=1}^m \lvert Y_k  \rvert ^\alpha\right) \\
    &= - \log{c_\alpha} + \frac{d^{-\alpha}}{n+m}\left(\sum_{j=1}^n \left\lvert \Tilde{Y}_j \right\rvert^\alpha + \sum_{k=1}^m \left\lvert Y_k \right\rvert^\alpha\right).
\end{align*}
It follows from the calculations above that
\begin{eqnarray*}
-\ell_{n,m}(\beta) \geq - \log{c_\alpha} + \frac{d^{-\alpha}}{n+m}\sum_{k=1}^m \lvert Y_k - \beta^\top X_k \rvert ^\alpha &>& - \log{c_\alpha} + \frac{d^{-\alpha}}{n+m}\left(\sum_{j=1}^n \left\lvert \Tilde{Y}_j \right\rvert^\alpha + \sum_{k=1}^m \left\lvert Y_k \right\rvert^\alpha\right) \\
&= & -\ell_{n,m}(0).
\end{eqnarray*}
We conclude that $-\ell_{n,m}(\beta) > -\ell_{n,m}(0)$ for any $\beta: \Vert \beta \Vert > R$.
Now, consider the closed ball  $\overline{\mathcal{B}}(0,R) := \{\beta\in\mathbb{R}^p: \left\lVert \beta\right\rVert \leq R\}$. Since $\beta \mapsto -\ell_{n,m}(\beta)$ is continuous, this function attains its  minimum on $\overline{\mathcal{B}}(0,R)$ at some vector $\beta^\ast \in \overline{\mathcal{B}}(0,R)$ such that 
\[
-\ell_{n,m}(\beta^\ast) \leq  -\ell_{n,m}(\beta) \quad \text{for all $\beta \in \overline{\mathcal{B}}(0,R)$}
\]
Since $0 \in \overline{\mathcal{B}}(0,R)$,  we have that
\[
-\ell_{n,m}(\beta^\ast) \leq  -\ell_{n,m}(0).
\]
Since for all $\beta \in \{b \in\mathbb{R}^p: \left\lVert b \right\rVert >R\}$
\[
-\ell_{n,m}(\beta^\ast) \leq  -\ell_{n,m}(0) < -\ell_{n,m}(\beta)
\]
it follows that $\beta^\ast$ is a minimizer of $\beta \mapsto -\ell_{n,m}(\beta)$ over $\mathbb{R}^p$. Thus, a maximizer of $\beta \mapsto \ell_{n,m}(\beta)$ exists.
    \end{proof}

\begin{proof}[Proof of Lemma \ref{Lemma: Existence of MLE for asymptotic sample case}]
Recall $R$ from (\ref{R}).  Note that
\begin{align*}
    R &= \left( \frac{2^{\alpha-1}}{\inf_{u \in \mathcal{S}^{p-1}} \sum_{k=1}^m \left\lvert u^\top X_k\right\rvert^\alpha} \sum_{j=1}^n \left\lvert \Tilde{Y}_j \right\rvert^\alpha +\frac{2^{\alpha}}{\inf_{u \in \mathcal{S}^{p-1}} \sum_{k=1}^m \left\lvert u^\top X_k\right\rvert^\alpha} \sum_{k=1}^m \left\lvert Y_k \right\rvert^\alpha \right)^{\frac{1}{\alpha}} \\
    &=\left(\frac{2^{\alpha-1}}{\inf_{u \in \mathcal{S}^{p-1}} \frac{1}{n+m}\sum_{k=1}^m \left\lvert u^\top X_k\right\rvert^\alpha}\right)^{\frac{1}{\alpha}}\left(\frac{1}{n+m}\sum_{j=1}^n \left\lvert \Tilde{Y}_j \right\rvert^\alpha  + \frac{2}{n+m}\sum_{k=1}^m \left\lvert Y_k \right\rvert^\alpha\right)^{\frac{1}{\alpha}}
\end{align*}
Since $\lim_{m, n \to \infty} \frac{m}{n} = \lambda \in (0,1)$, we have then by SLLN and the continuous mapping theorem
\begin{align*}
    \inf_{u \in \mathcal{S}^{p-1}} \frac{1}{n+m}\sum_{k=1}^m \left\lvert u^\top X_k\right\rvert^\alpha &= \frac{1}{1+\frac{n}{m}}\inf_{u \in \mathcal{S}^{p-1}} \frac{1}{m}\sum_{k=1}^m \left\lvert u^\top X_k\right\rvert^\alpha\\
    &\overset{m, n \rightarrow \infty}{\longrightarrow}\frac{1}{1+\frac{1}{\lambda}}\inf_{u \in \mathcal{S}^{p-1}} \mathbb{E}\left[\left\lvert u^\top X\right\rvert^\alpha\right],  \ \text{with probability 1} \\
    &=\frac{1}{1+\frac{1}{\lambda}}\inf_{u \in \mathcal{S}^{p-1}} \int\left\lvert u^\top x\right\rvert^\alpha f^X(x)\,dx.
\end{align*}
This means that with probability 1 there exists an integer $m_0\geq 1$ such that for all $n,m \geq m_0$:
\begin{align*}
    \inf_{u \in \mathcal{S}^{p-1}} \frac{1}{n+m}\sum_{k=1}^m \left\lvert u^\top X_k\right\rvert^\alpha \geq \frac{1}{1+\frac{1}{\lambda}}\inf_{u \in \mathcal{S}^{p-1}} \int\left\lvert u^\top x\right\rvert^\alpha f^X(x)\,dx - \delta
\end{align*}
for a fixed $\delta > 0$.  If we choose $\delta = \frac{1}{2}\cdot\frac{1}{1+\frac{1}{\lambda}}\inf_{u \in \mathcal{S}^{p-1}} \int\left\lvert u^\top x\right\rvert^\alpha f^X(x)\,dx$, then for all $n, m\geq m_0$ we have that
\[
\inf_{u \in \mathcal{S}^{p-1}} \frac{1}{n+m}\sum_{k=1}^m \left\lvert u^\top X_k\right\rvert^\alpha \geq \frac{\lambda}{2(1+\lambda)}\inf_{u \in \mathcal{S}^{p-1}} \int\left\lvert u^\top x\right\rvert^\alpha f^X(x)\,dx
\]
Applying again the SLLN we have
\begin{align*}
    \frac{1}{n+m}\sum_{j=1}^n \left\lvert \Tilde{Y}_j \right\rvert^\alpha &= \frac{1}{1+\frac{m}{n}} \frac{1}{n} \sum_{j=1}^n \left\lvert \Tilde{Y}_j \right\rvert^\alpha\\
    &\overset{m, n \rightarrow\infty}{\longrightarrow}\frac{1}{1+\lambda} \mathbb{E}\left[|\Tilde{Y}|^\alpha\right] \\
    &=\frac{1}{1+\lambda} \int|y|^\alpha f^{\Tilde{Y}}(y)\,dy = \frac{1}{1+\lambda} \int|y|^\alpha f^\epsilon(y-\beta_0^\top x)f^X(x)\,dxdy
\end{align*}
with probability 1. This means that with probability 1, there exists an integer $m_1 \ge 1$ such that for all $n,m\geq m_1$:
\[
\frac{1}{n+m}\sum_{j=1}^n \left\lvert \Tilde{Y}_j \right\rvert^\alpha \leq \frac{3}{1+\lambda} \int|y|^\alpha f^\epsilon(y-\beta_0^\top x)f^X(x)\,dxdy
\]
Also, by the same theorem we have that
\begin{align*}
    \frac{1}{n+m}\sum_{k=1}^m \left\lvert Y_k \right\rvert^\alpha &= \frac{1}{1+\frac{n}{m}} \frac{1}{m} \sum_{k=1}^m \left\lvert Y_k \right\rvert^\alpha\\
    &\overset{m, n \rightarrow\infty}{\longrightarrow}\frac{1}{1+\frac{1}{\lambda}} \mathbb{E}\left[|Y|^\alpha\right] \\
    &= \frac{1}{1+\frac{1}{\lambda}}\int|y|^\alpha f^\epsilon(y-\beta_0^\top x)f^X(x)\,dxdy
\end{align*}
with probability 1. This means that with probability 1 there exists an integer $m_2 \ge 1$ such that for all $n, m \geq m_2$:
\begin{eqnarray*}
\frac{1}{n+m}\sum_{k=1}^m \left\lvert Y_k \right\rvert^\alpha \leq \frac{3 \lambda}{2(1+\lambda)}\int|y|^\alpha f^\epsilon(y-\beta_0^\top x)f^X(x)\,dxdy
\end{eqnarray*}
Put 
$$
B = \inf_{u \in \mathcal{S}^{p-1}} \int\left\lvert u^\top x\right\rvert^\alpha f^X(x)\,dx
$$
and
$$
C = \int|y|^\alpha f^\epsilon(y-\beta_0^\top x)f^X(x)\,dxdy.
$$
Then, for all $n, m \geq m^\ast :=\max(m_0,m_1,m_2)$, it holds that 
\begin{align}\label{R*}
    R &= \left(\frac{2^{\alpha-1}}{\inf_{u \in \mathcal{S}^{p-1}} \frac{1}{n+m}\sum_{k=1}^m \left\lvert u^\top X_k\right\rvert^\alpha}\right)^{\frac{1}{\alpha}}\left(\frac{1}{n+m}\sum_{j=1}^n \left\lvert \Tilde{Y}_j \right\rvert^\alpha  + \frac{2}{n+m}\sum_{k=1}^m \left\lvert Y_k \right\rvert^\alpha\right)^{\frac{1}{\alpha}} \notag \\
    & \le \left( \frac{2^{\alpha}  (1+\lambda)}{\lambda  B}\right)^{1/\alpha} (3C)^{1/\alpha} = 2\cdot 3^{1/\alpha} \left(\frac{1+\lambda}{\lambda}\cdot\frac{C}{B}\right)^{1/\alpha}:= R^\ast.
\end{align}
This means that with probability 1 there exists an integer $m^{\ast} \ge 1$ such that for all $n,m \geq m^\ast$ we have that 
$$\{\beta \in\mathbb{R}^p:\left\lVert \beta\right\rVert > R^\ast\} \subseteq \{\beta\in\mathbb{R}^p:\left\lVert\beta\right\rVert > R\}. $$
By Lemma \ref{Lemma: Existence of MLE for fixed sample case} we have: for all $n, m \geq m^\ast$: $-\ell_{n,m}(\beta) > -\ell_{n,m}(0)$ for all $\beta$ such that $\Vert \beta\Vert > R^\ast$. Using again continuity of the function $\beta \mapsto \ell_{n,m}(\beta)$ and compactness of closed balls, we conclude that with probability 1, there exists $m^\ast$ such that for all $n,m \ge m^\ast$ a maximizer of $\beta \mapsto \ell_{n,m}(\beta)$ belongs to the closed ball $\overline{\mathcal{B}}(0, R^\ast)$.

\end{proof}

\begin{proof}[Proof of Theorem \ref{Consis}]

To show the convergence in (\ref{UC}), we will start with proving that
\begin{eqnarray}\label{UC1}
\sup_{\beta \in \mathcal{B}(0, R^\ast)}  \left \vert \int \log f(y - \beta^\top x)  d(\mathbb P_m - \mathbb P)(x,y) \right \vert  = o_{\mathbb P}(1).
\end{eqnarray}
First, note that 
\begin{eqnarray*}
\log f(t) & = &  \log f(t) \mathds{1}_{t \ge 0}  +   \log f(t) \mathds{1}_{t < 0}  \\
& =  &  g_+(t)  + g_{-}(t)  -  \log(c_\alpha) 
\end{eqnarray*}
where $g_+(t)  =  \log f(t) \mathds{1}_{t \ge 0} + \log(c_\alpha) \mathds{1}_{t < 0} $ and $g_{-}(t) = \log f(t) \mathds{1}_{t < 0} + \log(c_\alpha) \mathds{1}_{t \ge 0}$.  The functions $g_+$ and $g_-$ are monotone non-increasing and non-decreasing respectively. Consider now the class of functions
\begin{eqnarray*}
\mathcal L=  \{(x, y) \mapsto  l_\beta(x, y) = y - \beta^\top x, \ \beta \in \overline{\mathcal B}(0, R^\ast)  \}.   
\end{eqnarray*}
The class $\mathcal{L}$ is indexed by $\beta$ and hence is a subset of the finite dimensional vector space
\begin{eqnarray*}
\{(x, y) \mapsto  l_\beta(x, y) = y - \beta^\top x, \ \beta \in \mathbb R \}.   
\end{eqnarray*}
From \cite[Lemma 2.6.16]{aadbookE2} it follows that $\mathcal L$ is a VC-subgraph of dimension $V \le p+2$. The convergence result in (\ref{UC1}) can be re-written as 
\begin{eqnarray*}
\sup_{l \in \mathcal{L}} \left \vert \int \left(g_+ \circ l + g_{-} \circ l - \log(c_\alpha)\right) d (\mathbb P_m - \mathbb P)  \right \vert  = o_{\mathbb P}(1).
\end{eqnarray*}
Using the fact that $\int d (\mathbb P_m - \mathbb P) =0$, it follows that 
$$
\int \left(g_+ \circ l + g_{-} \circ l - \log(c_\alpha)\right) d (\mathbb P_m - \mathbb P) =  \int \left(g_+ \circ l + g_{-} \circ l \right) d (\mathbb P_m - \mathbb P)
$$
and hence
\begin{eqnarray*}
\sup_{l \in \mathcal{L}} \left \vert \int \left(g_+ \circ l + g_{-} \circ l - \log(c_\alpha)\right) d (\mathbb P_m - \mathbb P)  \right \vert  &\le &  \sup_{l \in \mathcal{L}} \left \vert \int g_+ \circ l \  d (\mathbb P_m - \mathbb P)  \right \vert \\
&& \ +  \sup_{l \in \mathcal{L}} \left \vert \int g_{-} \circ l \  d (\mathbb P_m - \mathbb P)  \right \vert.
\end{eqnarray*}
\cite[Lemma 2.6.20 - (viii)]{aadbookE2} implies that the class $g_+ \circ \mathcal L$ is a VC-subgraph.  Let us denote by $V_+$ its index. Also, for all $t  \in \mathbb R$  $\vert g_+ \vert(t) \le \vert \log(c_\alpha) \vert + d^{-\alpha} \vert t \vert^\alpha$ and hence
\begin{eqnarray*}
\vert g_+(y - \beta^\top x)  \vert & \le &   \vert \log(c_\alpha) \vert + d^{-\alpha} \vert y - \beta^\top x \vert^\alpha  \\
& \le &  \vert \log(c_\alpha) \vert + 2^{\alpha -1} d^{-\alpha} \left(\vert y \vert^\alpha  + C^\alpha   \right):= G_+(x, y) 
\end{eqnarray*}
a.e. $\mathbb P$ with $C = R^\ast B$. Above, we used the Cauchy-Schwarz inequality and convexity of $t \mapsto \vert t \vert^\alpha$.  Denoting the class $g_+ \circ \mathcal L$ by $\mathcal{G}_+$, it follows from \cite[Theorem 2.6.7]{aadbookE2} that for $\delta \in (0,1)$ and any probability measure $Q$ such that $\Vert G_+ \Vert^2_Q = \int G_+(x, y)^2  dQ(x, y) > 0$
\begin{eqnarray}\label{CN}
N(\delta \Vert G_+ \Vert_Q, \mathcal G_+, L_2(Q)) \le D \left(\frac{1}{\delta}\right)^{2 V_+} 
\end{eqnarray}
for some universal constant $D > 0$, which can be taken to be larger than $ 1/e$ without loss of generality. Define now the uniform entropy integral 
\begin{eqnarray*}
J(\eta, \mathcal{G}_+)  = \sup_{Q}   \int_0^\eta  \sqrt{1 + \log N(\delta \Vert G_+ \Vert_Q, \mathcal G_+, L_2(Q)) } d\delta
\end{eqnarray*}
for $\eta > 0$. In the above definition, the supremum is taken again over $Q$ such that $\Vert G_+ \Vert_Q > 0$. Using the bound given in (\ref{CN}), it follows that
\begin{eqnarray*}
J(1, \mathcal{G}_+) & \le  & \int_0^1  \left(\sqrt{1 + \log D} +  \sqrt{2 V_+} \sqrt{\log\left(\frac{1}{\delta} \right)} \right) d\delta  \\
& \le &  \sqrt{1 + \log D}  +  \sqrt{2 V_+} \int_0^1 \frac{1}{\sqrt{\delta}} d \delta \\
& =  & \sqrt{1 + \log D}  +  2\sqrt{2 V_+}.
\end{eqnarray*}
Also, it is easy to show that the envelope $G_+$ is square integrable with respect to $\mathbb P$ since
\begin{eqnarray*}
\mathbb E[\vert Y \vert^{2\alpha}] & \le  &  2^{2\alpha -1} \left( C^{2\alpha} + \mathbb E[ \vert \epsilon\vert^{2\alpha} ] \right) \\
& = &2^{2\alpha -1} \left( C^{2\alpha} + c_\alpha \int \vert t \vert^{2\alpha} e^{- d^{-\alpha} \vert t \vert^\alpha} dt  \right)  \\
& =  &  2^{2\alpha -1} \left( C^{2\alpha} +  \frac{2}{\alpha} c_\alpha  d^{2 \alpha +1}  \Gamma\left( 2 + \frac{1}{\alpha}\right)   \right)  < \infty.
\end{eqnarray*}
Now, by \cite[Theorem 2.14.1]{aadbookE2}, it follows that
\begin{eqnarray*}
\mathbb E[\Vert \mathbb G_m \Vert^2_{\mathcal{G}_+}]^{1/2} \lesssim J(1, \mathcal{G}_+) \Vert G_+ \Vert_{\mathbb P} 
\end{eqnarray*}
where $\Vert \mathbb G_m \Vert_{\mathcal{G}_+} =  \sup_{g \in \mathcal{G}_+} \vert \mathbb G_m g \vert$ and $\Vert G_+ \Vert_{\mathbb P} =  \left(\int G_+(x, y)^2 d\mathbb P(x, y)\right)^{1/2}$.  Using the Markov's inequality, it follows that 
\begin{eqnarray*}
\Vert \mathbb G_m \Vert_{\mathcal{G}_+} = O_{\mathbb P}(1)
\end{eqnarray*}
which in turn implies that 
\begin{eqnarray*}
\sup_{l \in \mathcal{L}} \left \vert \int g_+ \circ l \ d (\mathbb P_m - \mathbb P)  \right \vert  = O_{\mathbb P}(m^{-1/2}) = o_{\mathbb P}(1).
\end{eqnarray*}
Since a similar reasoning can be used for the class of functions $g_- \circ \mathcal L$ we conclude that (\ref{UC1}) is true.

Next, we will turn the unmatched part of the empirical log-likelihood. We will show that
\begin{eqnarray}\label{UC2}
&& \sup_{\beta \in \overline{\mathcal B}(0, R^\ast)} \left \vert \int  \log \left(\int f(y -\beta^\top x) d\mathbb P^{\tilde X}_n(x) \right) d \mathbb P^{\tilde Y}_n(y)  -  \int  \log\left(\int f(y -\beta^\top x) d\mathbb P^{\tilde X}(x) \right) d \mathbb P^{\tilde Y}(y) \right \vert  \notag \\
&&= o_{\mathbb P^{\tilde X} \otimes \mathbb P^{\tilde Y}}(1). 
\end{eqnarray}
We have that
\begin{eqnarray*}
&& \int  \log\left(\int f(y -\beta^\top x) d\mathbb P^{\tilde X}_n(x) \right) d \mathbb P^{\tilde Y}_n(y)  - \int  \log\left(\int f(y -\beta^\top x) d\mathbb P^{\tilde X}(x) \right) d \mathbb P^{\tilde Y}(y)   \\
&& = \int  \log\left(\int f(y -\beta^\top x) d\mathbb P^{\tilde X}_n(x) \right) d \mathbb P^{\tilde Y}_n(y) - \int  \log \left(\int f(y -\beta^\top x) d\mathbb P^{\tilde X}_n(x) \right) d \mathbb P^{\tilde Y}(y) \\
&& + \  \int  \log \left(\int f(y -\beta^\top x) d\mathbb P^{\tilde X}_n(x) \right) d \mathbb P^{\tilde Y}(y) - \int  \log\left(\int f(y -\beta^\top x) d\mathbb P^{\tilde X}(x) \right) d \mathbb P^{\tilde Y}(y) \\
&& := \mathbb A_n(\beta)  + \mathbb B_n(\beta).
\end{eqnarray*}
We start with $\mathbb A_n$. Using the notation from empirical processes, we can rewrite $\mathbb A_n(\beta)$ as 
\begin{eqnarray*}
\mathbb A_n(\beta)  &=&   \int \log \left(\int f(y -\beta^\top x) d\mathbb P^{\tilde X}_n(x) \right) \  d (\mathbb P^{\tilde Y}_n - \mathbb P^{\tilde{Y}})(y) \\
& = & \int h_\beta(y) \  d (\mathbb P^{\tilde Y}_n - \mathbb P^{\tilde{Y}})(y)
\end{eqnarray*}
where for $y \in \mathbb R$
$$
h_{\beta}(y)  =  \log\left(\sum_{i=1}^k w_i f(y- \beta^\top x_i) \right)
$$
for $k = n$, $x_i = \tilde X_i$, and $w_i = 1/n$.  We will denote by $\mathcal H$ the class of functions to which $h_\beta$ belongs.  More formally we can take $\mathcal H$ to be the class 
\begin{eqnarray*}
\mathcal H & =  &  \Big \{ h_\beta:  h_\beta(y) =  \log\left(\sum_{i=1}^k w_i f(y- \beta^\top x_i) \right), \beta \in \overline{\mathcal B}(0, R^\ast), 
\\ && \  \  \    k \in \mathbb N, x_i \in \overline{\mathcal B}(0, B), w_i \in [0,1]: \sum_{i=1}^k w_i = 1  \Big \}.
\end{eqnarray*}
Note that the functions in $\mathcal H$ depend also on $x_1, \ldots, x_k$ but we omit to write this dependence explicitly. For fixed $y$ and $\beta$ (for $\alpha =1$, $y$ and $\beta$ are such that  $y \ne \beta^\top x_i$ for all $i=1, \ldots, k$), the gradient of the function $b \mapsto h_b(y)$ at $\beta$ is given by 
\begin{eqnarray*}
\nabla h_\beta(y) & = & -\frac{ \sum_{i=1}^k w_i f'(y- \beta^\top x_i) x_i}{ \sum_{i=1}^k w_i f(y- \beta^\top x_i)}  \\
& =  &  \frac{ \alpha  d^{-\alpha} \sum_{i=1}^k w_i \vert y- \beta^\top x_i \vert^{\alpha -1}  \operatorname{sgn}(y-\beta^\top x_i) \exp(- d^{-\alpha} \vert y- \beta^\top x_i\vert^\alpha ) x_i}{ \sum_{i=1}^k w_i \exp(- d^{-\alpha}\vert y- \beta^\top x_i\vert^\alpha )} 
\end{eqnarray*}
where $\operatorname{sgn}(z) = 1$ if $z > 0$, $= -1$ if $z < 0$ and $=0$ otherwise.  Then, using the fact that $x_i \in \overline{\mathcal{B}}(0, B)$ it follows
\begin{eqnarray*}
\Vert \nabla h_\beta(y) \Vert \le  B \alpha d^{-\alpha} \sum_{i=1}^k \tilde{w}_i \vert y- \beta^\top x_i \vert^{\alpha - 1}
\end{eqnarray*}
with 
\begin{eqnarray*}
\tilde w_i =  \frac{w_i \exp(- d^{-\alpha} \vert y - \beta^\top x_i \vert^\alpha) }{\sum_{j=1} ^k w_j \exp(- d^{-\alpha} \vert y - \beta^\top x_j \vert^\alpha) }   
\end{eqnarray*}
for $i=1, \ldots, k$, which sum up to 1.  This implies that 
\begin{eqnarray*}
\Vert \nabla h_\beta(y) \Vert & \le & B \alpha d^{-\alpha} \sup_{1 \le j \le k}  \vert y - \beta^\top x_j \vert^{\alpha -1} \sum_{i=1}^k \tilde{w}_i \\
& \le & B \alpha d^{-\alpha} (2^{\alpha -2} \vee 1) (\vert y \vert^{\alpha -1} +  C^{\alpha -1}) : = H(y)    
\end{eqnarray*}
where we recall that $C = B R^\ast$. In the preceding inequality, we used convexity of the function $t \mapsto \vert t \vert^{\alpha -1}$ for $\alpha \ge 2$.

Using the mean value theorem for multivariate functions it holds that for any $\beta$ and $\beta'$ in $\overline{\mathcal B}(0, R^\ast)$
\begin{eqnarray*}
\vert h_\beta(y)  - h_{\beta'}(y) \vert \le  \Vert \beta - \beta' \Vert   H(y).
\end{eqnarray*}
Define now the class $\widetilde{\mathcal H}$ of functions $\tilde{h}_\beta$ given by
\begin{align*}
    \tilde{h}_\beta(y) : =  \frac{h_\beta(y) - h_0(y)}{(R^\ast \vee 1)H(y)}
\end{align*}
where $h_0$ denotes the element in $\mathcal H$ corresponding to $\beta=0$, that is $h_0 = \log{f}$. It is clear that the class $\widetilde{\mathcal H}$ admits $1$ as an envelope function, since
\begin{align}\label{Env1}
    \lvert \tilde{h}_\beta(y)\rvert &= \frac{1}{(R^\ast\vee 1)H(y)} \lvert h_\beta(y)-h_{0}(y)\rvert \notag \\
    &\leq \frac{\lVert\beta-0\rVert}{R^\ast \vee 1}\leq 1.
\end{align}
Also, 
\begin{eqnarray*}
\nabla \tilde{h}_\beta(y)   =  \frac{\nabla h_\beta(y)}{(R^\ast \vee 1) H(y)}
\end{eqnarray*}
and hence
\begin{eqnarray}\label{EnvGrad1}
\Vert \nabla  \tilde{h}_\beta(y) \Vert  \le \frac{1}{R^\ast \vee 1} \le 1.
\end{eqnarray}
Now, note that
\begin{align}\label{An}
    \mathbb A_n(\beta)  & =   (R^\ast\vee1) \int  \tilde{h}_\beta(y)  H(y) \  d(\mathbb P^{\tilde Y}_n  - \mathbb P^{\tilde Y})(y)  + \int \log{f(y)} \  d(\mathbb P^{\tilde Y}_n  - \mathbb P^{\tilde Y})(y) \notag  \\
& =  (R^\ast \vee 1) \frac{1}{\sqrt n} \int \tilde{h}_\beta(y)  H(y) \ d\mathbb G^{\tilde Y}_n(y)+ \int \log{f(y)} \  d(\mathbb P^{\tilde Y}_n  - \mathbb P^{\tilde Y})(y) \notag \\
& =  (R^\ast \vee 1)  \widetilde{\mathbb A}_{n}(\beta) + \int \log{f(y)} \  d(\mathbb P^{\tilde Y}_n  - \mathbb P^{\tilde Y})(y).
\end{align}
By \cite[Theorem 2.7.17]{aadbookE2} and (\ref{EnvGrad1}) it follows that  for all $\eta > 0$
\begin{eqnarray*}
N_{[ \ ]}(2 \eta , \widetilde{\mathcal H}, L_2(\mathbb P^{\tilde Y})) \le N(\eta, \overline{\mathcal B}(0, R^\ast), \Vert \cdot \Vert )
\end{eqnarray*}
where $N_{[ \ ]}(2 \eta, \widetilde{\mathcal H},  L_2(\mathbb P^{\tilde Y})) $ is the $(2\eta)$-bracketing number of the class $\widetilde{\mathcal H}$ with respect to  $L_2(\mathbb P^{\tilde Y})$, and $N(\eta, T, d)$ is the covering number of some set $T$ equipped with a metric $d$.  Here, $\Vert \cdot \Vert $ denotes the Euclidean norm. As stated in page 147 of \cite{aadbookE2}, $N(\eta, T, d) \le D(\eta, T, d)$ where the latter denotes the packing number which is the maximum of $\eta$-separated points in $T$. By \cite[Problems and Complements (7) in page 143]{aadbookE2} we know that 
$$
D(\eta, \overline{\mathcal B}(0, R^\ast), \Vert \cdot \Vert )  \le \left(\frac{3 R^\ast}{\eta} \right)^p.
$$
Therefore, 
\begin{eqnarray}\label{BrackEntr}
N_{[ \ ]}(\eta , \widetilde{\mathcal H}, L_2(\mathbb P^{\tilde Y})) \le \left(\frac{6 R^\ast}{\eta} \right)^p.
\end{eqnarray}
Now, we will derive an upper bound for the bracketing number of the class $\widetilde{\mathcal{H}} \cdot H $; i.e., the class of functions of the form $\tilde h_\beta \times H$.  Let $[\tilde L, \tilde U]$ be an $\eta$-bracket of $\widetilde{\mathcal{H}}$. Since $\Vert \tilde{h}_\beta \Vert_{\infty} \le 1$, it can be easily shown that we can assume without loss of generality that $\tilde L \ge -1 $ and $\tilde U \le 1$. In fact, if we replace $\tilde L$ and $\tilde U$  with $\tilde{L}' = \tilde L \vee (-1)$ and $ \tilde U'= \tilde U \wedge 1$ we can show that $[\tilde{L}', \tilde{U}']$ is also an $\eta$-bracket. Define
\begin{eqnarray*}
L =  \tilde{L}  H, \ \ \text{and} \ \  U =  \tilde{U}  H.
\end{eqnarray*}
Then, $U- L =  H (\tilde U - \tilde L) \ge 0$ since $H \ge 0$ and $\tilde U - \tilde L \ge 0$ by the definition of a bracket.  Also,  
\begin{eqnarray*}
&& \int (U(y) - L(y))^2 d\mathbb P^{\tilde Y}(y)    \\
&& =   \int (\tilde U(y) - \tilde L(y))^2  H^2(y) d\mathbb P^{\tilde Y}(y) \\
&& \le  2 \int \big(\tilde U(y)  - \tilde L(y)\big)  H^2(y) d\mathbb P^{\tilde Y}(y),\ \textrm{using the fact that $0 \le \tilde U - \tilde L \le 2$}  \\
&& \le 2 \left(\int \big(\tilde U(y)  - \tilde L(y)\big)^2 d\mathbb P^{\tilde Y}(y)\right)^{1/2} \left(\int H^4(y) d\mathbb P^{\tilde Y}(y)\right)^{1/2}, \\
&& \  \  \ \ \textrm{using the Cauchy-Schwarz inequality} \\
&&  \le 2 \eta  \left(\int H^4(y) d\mathbb P^{\tilde Y}(y)  \right)^{1/2}  = K \eta
\end{eqnarray*}
where $K = 2 \big(\int H^4(y) d\mathbb P^{\tilde Y}(y) \big)^{\frac{1}{2}} < \infty.$ In fact, we compute
\begin{eqnarray*}
\int  H(y)^4  d\mathbb P^{\tilde Y}(y)  & = &  \left(B \alpha d^{-\alpha} (2^{\alpha -2} \vee 1) \right)^4\int   (\vert y \vert^{\alpha -1} +  C^{\alpha -1})^4 \left(\int f(y - \beta^\top_0 x)  f^X(x) dx \right)  dy  
\end{eqnarray*}
with
\begin{eqnarray*}
\int f(y - \beta^\top_0 x)  f^X(x) dx &  =  &  c_\alpha \int \exp(-d^{-\alpha} \vert y- \beta^\top_0 x \vert^\alpha) f^X(x) dx \\
& \le & c_\alpha \exp(d^{-\alpha} C^\alpha)\exp(-d^{-\alpha} 2^{1-\alpha} \vert y \vert^{\alpha}) 
\end{eqnarray*}
using again the fact that $  \vert a - b \vert^\alpha \ge 2^{1-\alpha} \vert a \vert^\alpha  - \vert b\vert^\alpha $  and $- \vert \beta^\top_0 x \vert  \ge - \Vert \beta_0 \Vert B \ge - C$.  Therefore,
\begin{eqnarray*}
\int  H(y)^4  d\mathbb P^{\tilde Y}(y)  & \lesssim &  \int   (\vert y \vert^{\alpha -1} +  C^{\alpha -1})^4 \exp(-d^{-\alpha} 2^{1-\alpha} \vert y \vert^{\alpha}) dy < \infty.
\end{eqnarray*}
From the calculations  we conclude that  for any given $\eta > 0$,   the number of $(K \eta)^{1/2}$-brackets of the class $\widetilde{\mathcal H} \cdot H$ is at most the number of $\eta$-brackets of the class $\widetilde {\mathcal H}$. More precisely, for any $\eta > 0$
\begin{eqnarray*}
N_{[ \ ]}( K^{1/2} \eta^{1/2}, \widetilde{\mathcal H} \cdot H, L_2(\mathbb P^{\tilde Y}))  \le N_{[ \ ]}( \eta, \widetilde{\mathcal H}, L_2(\mathbb P^{\tilde Y}))
\end{eqnarray*}
or equivalently
\begin{eqnarray}\label{BrackIneq}
N_{[ \ ]}( \eta, \widetilde{\mathcal H} \cdot H, L_2(\mathbb P^{\tilde Y}))  \le  N_{[ \ ]}( \eta^2/K, \widetilde{\mathcal H}, L_2(\mathbb P^{\tilde Y})).
\end{eqnarray}
Note that $H$ is an envelope for the class $\widetilde{\mathcal{H}} \cdot H $. Denote the bracketing integral of this class  (with respect to $L_2(\mathbb{P}^{\tilde{Y}})$) by 
\begin{eqnarray*}
J_{[ \ ]}(\eta)= \int_0^\eta \sqrt{1 + \log N_{[ \ ]}(t \Vert H\Vert, \widetilde{\mathcal H} \cdot H, L_2(\mathbb P^{\tilde Y}))}  dt 
\end{eqnarray*}
where 
\begin{eqnarray*}
\Vert H\Vert^2  =   \int H^2(y)  d\mathbb P^{\tilde Y}(y)  \lesssim    \int (\vert y \vert^{\alpha -1} + C^{\alpha -1} )^2 \exp(-d^{-\alpha} 2^{1-\alpha} \vert y\vert^\alpha) dy.
\end{eqnarray*} 
By the inequalities in (\ref{BrackIneq}) and (\ref{BrackEntr}), we have for all $t > 0$
\begin{eqnarray*}
N_{[ \ ]}(t \Vert H\Vert, \widetilde{\mathcal H} \cdot H, L_2(\mathbb P^{\tilde Y})) & \le &  N_{[ \ ]}(t^2 \Vert H \Vert^2 K^{-1}, \widetilde{\mathcal H}, L_2(\mathbb  P^{\tilde Y})) \\
& \le &  \left( \frac{6 K R^\ast}{t^2 \Vert  H \Vert^2 }\right)^p
\end{eqnarray*}
At the cost of increasing $R^\ast$ so that 
$$
\Vert H \Vert \le \sqrt{6 K R^\ast}
$$
we have for all $\eta \in (0,1]$  that
\begin{eqnarray*}
J_{[ \ ]}(\eta) & \le&  \int_0^\eta \sqrt{1 +  p \log \left( \frac{6 K R^\ast}{t^2 \Vert  H \Vert^2} \right) } dt  \\
& \le & \eta  + \sqrt{p} \int_0^\eta \sqrt{\log \left(\frac{6 K R^\ast}{t^2 \Vert  H \Vert^2}\right)}    dt  \\
& = & \eta  + \sqrt{2p} \int_0^\eta \sqrt{\log \left(\frac{\sqrt{ 6 K R^\ast}}{t \Vert  H \Vert}\right)}    dt  \\
& \le & \eta  +  \frac{2\sqrt{2p} (6 K R^\ast)^{1/4}}{\Vert H \Vert^{1/2}} \sqrt{\eta} \\
& = &  \sqrt{\eta} \left(\sqrt \eta + \frac{2\sqrt{2p} (6 K R^\ast)^{1/4}}{\Vert H \Vert^{1/2}}  \right).
\end{eqnarray*}
By \cite[Theorem 2.14.16]{aadbookE2}, it follows that
\begin{eqnarray}\label{ImpIneq}
\mathbb E_{\mathbb P^{\tilde Y}}\left[ \Vert \mathbb G^{\tilde Y}_n \Vert_{\widetilde{\mathcal H} \cdot H}\right] 
 \lesssim  J_{[ \ ]}(1) \Vert H \Vert   =  \left(1 + \frac{2\sqrt{2p} (6 K R^\ast)^{1/4}}{\Vert H \Vert^{1/2}}  \right) \Vert H \Vert \lesssim 1
\end{eqnarray}
for all $n \ge 1$. Recall that 
\begin{eqnarray*}
\widetilde{A}_n(\beta) & =   &  \frac{1}{\sqrt n} \int \tilde h_\beta(y) H(y) d\mathbb G_n^{\tilde Y}(y)   \\
& \le & \frac{1}{\sqrt n}  \Vert \mathbb G^{\tilde Y}_n \Vert_{\widetilde{\mathcal H} \cdot H}.
\end{eqnarray*}
It follows from (\ref{ImpIneq}) and the Markov's inequality that 
\begin{eqnarray*}
\sup_{\beta \in \overline{\mathcal B}(0, R^\ast)} \vert \widetilde{\mathbb A}_n(\beta) \vert = O_{\mathbb P^{\tilde X}  \otimes \mathbb P^{\tilde Y}}\left(\frac{1}{\sqrt n} \right)  = o_{\mathbb P^{\tilde X}  \otimes \mathbb P^{\tilde Y}}(1).
\end{eqnarray*}
Finally, note that
\begin{eqnarray*}
\int \vert \log f(y) \vert^2 d\mathbb P^{\tilde Y}(y)  & = & \int \left( \log c_\alpha -  d^{-\alpha} \vert y \vert^\alpha \right)^2 f^{\tilde{Y}}(y) dy \\
& \le & 2 \left((\log c_\alpha)^2 +   d^{-2\alpha}   \int \vert y \vert^{2\alpha} f^{\tilde{Y}}(y) dy \right)  \\
& \le &  2 (\log c_\alpha)^2  \\
&& \  + \  2  d^{-2\alpha} c_\alpha \exp(d^{-\alpha} (\lVert\beta_0\rVert B)^\alpha) \int \vert y \vert^{2\alpha} \exp(-d^{-\alpha} 2^{1-\alpha} \vert y \vert^\alpha) dy \\
& \le &  2 (\log c_\alpha)^2  \\
&& \  + \ 2  d^{-2\alpha} c_\alpha \exp(d^{-\alpha} C^\alpha) \int \vert y \vert^{2\alpha} \exp(-d^{-\alpha} 2^{1-\alpha} \vert y \vert^\alpha) dy \\
& <  & \infty
\end{eqnarray*}
where above we used again the fact that $f^{\tilde{Y}}(y) = \int f(y- \beta_0^\top x) f^X(x) dx \le \sup_{x \in \overline{\mathcal B}(0, B)}  f(y-\beta_0^\top x)  \le c_\alpha \exp(-d^{-\alpha} 2^{1-\alpha} \vert y \vert^\alpha + d^{-\alpha}(\lVert\beta_0\rVert B)^\alpha).$   Now by the Central Limit Theorem, we have that 
$$
\int \log{f(y)} \  d(\mathbb P^{\tilde Y}_n  - \mathbb P^{\tilde Y})(y)  = O_{ \mathbb P^{\tilde Y}}\left(\frac{1}{\sqrt n} \right) =  o_{\mathbb  P^{\tilde Y}}(1).
$$
From the expression in (\ref{An}) and the calculations above, we conclude that
\begin{eqnarray*}
\sup_{\beta \in \overline{\mathcal B}(0, R^\ast)} \vert \mathbb A_n(\beta) \vert = O_{\mathbb P^{\tilde X}  \otimes \mathbb P^{\tilde Y}}\left(\frac{1}{\sqrt n} \right)  = o_{\mathbb P^{\tilde X}  \otimes \mathbb P^{\tilde Y}}(1).
\end{eqnarray*}

\bigskip
\bigskip

Now, we turn to the second empirical process $\mathbb B_n$.  First, note that for $\beta \in \overline{\mathcal{B}}(0, R^\ast)$
\begin{eqnarray*}
\mathbb B_n(\beta) =  \int  \log \left(\frac{\int f(y- \beta^\top x) d\mathbb P^{\tilde X}_n(x)}{\int f(y- \beta^\top x)d\mathbb P^{\tilde X}(x)}\right) d \mathbb P^{\tilde Y}(y). 
\end{eqnarray*}
Note that if  
\begin{eqnarray}\label{comp}
\int f(y- \beta^\top x) d\mathbb P^{\tilde X}_n(x) \ge \int f(y- \beta^\top x) d\mathbb P^{\tilde X}(x)
\end{eqnarray}
then
\begin{eqnarray*}
0 \le  \log \left(\frac{\int f(y- \beta^\top x) d\mathbb P^{\tilde X}_n(x)}{\int f(y- \beta^\top x)d\mathbb P^{\tilde X}(x)}\right) 
  & =  &  \log \left(1 +  \frac{\int f(y- \beta^\top x) d(\mathbb P^{\tilde X}_n  - \mathbb P^{\tilde X})(x)}{\int f(y- \beta^\top x)d\mathbb P^{\tilde X}(x)} \right) \\
& \le &    \frac{\int f(y- \beta^\top x) d(\mathbb P^{\tilde X}_n  - \mathbb P^{\tilde X})(x)}{\int f(y- \beta^\top x)d\mathbb P^{\tilde X}(x)},
\end{eqnarray*}
and when the inequality in (\ref{comp}) is reversed, then
\begin{eqnarray*}
  \frac{\int f(y- \beta^\top x) d(\mathbb P^{\tilde X}_n  - \mathbb P^{\tilde X})(x)}{\int f(y- \beta^\top x)d\mathbb P^{\tilde X}_n(x)} \le  \log \left(\frac{\int f(y- \beta^\top x) d\mathbb P^{\tilde X}_n(x)}{\int f(y- \beta^\top x)d\mathbb P^{\tilde X}(x)}\right)  \le 0.
\end{eqnarray*}
It follows that
\begin{eqnarray*}
\vert \mathbb B_n(\beta) \vert & \le  &  \int  \left \vert  \frac{\int f(y- \beta^\top x) d(\mathbb P^{\tilde X}_n  - \mathbb P^{\tilde X})(x)}{\int f(y- \beta^\top x)d\mathbb P^{\tilde X}(x)}   \right \vert d \mathbb P^{\tilde Y}(y) \\
&&  +  \int  \left \vert  \frac{\int f(y- \beta^\top x) d(\mathbb P^{\tilde X}_n  - \mathbb P^{\tilde X})(x)}{\int f(y- \beta^\top x)d\mathbb P^{\tilde X}_n(x)}   \right \vert d \mathbb P^{\tilde Y}(y)  \\
& := & \int \mathbb C_{n, \beta, 1}(y) d \mathbb P^{\tilde Y}(y) +  \int \mathbb C_{n, \beta, 2}(y) d \mathbb P^{\tilde Y}(y).  
\end{eqnarray*}
For $y \in \mathbb R$, we have that  
\begin{eqnarray*}
\mathbb C_{n,\beta, 1}(y) = \left \vert \int g_{\beta, y}(x)  d (\mathbb P^{\tilde X}_n  - \mathbb P^{\tilde X})(x) \right  \vert 
\end{eqnarray*}
with 
\begin{eqnarray*}
 g_{\beta, y}(x) =  \frac{f(y-\beta^\top x)}{\int f(y- \beta^\top z) d\mathbb P^{\tilde X}(z)}  = \frac{f(y-\beta^\top x)}{\int f(y- \beta^\top z) f^X(z) dz}.
\end{eqnarray*}
Let $\mathcal G_y$ be the class of such functions. We compute now the gradient of $\beta \mapsto  g_{\beta, y}(x)$.
\begin{eqnarray*}
\nabla  g_{\beta, y}(x)  & = &   - \frac{f'(y - \beta^\top x) x}{\int f(y- \beta^\top z) f^X(z) dz} + \frac{f(y - \beta^\top x) \int  
 z f'(y- \beta^\top z) f^X(z) dz }{\left(\int f(y- \beta^\top z) f^X(z) dz\right)^2}.    
\end{eqnarray*}
Let $y$ be such that $\vert y \vert >  C$ where we recall that $C = B R^\ast$. We have that 
\begin{eqnarray*}
\frac{\vert f'(y - \beta^\top x) \vert }{\int f(y- \beta^\top z) f^X(z) dz} &=  &  \frac{\alpha d^{-\alpha} \vert y - \beta^\top x \vert^{\alpha -1} \exp(- d^{-\alpha} \vert y - \beta^\top x\vert^\alpha)}{\int \exp(- d^{-\alpha} \vert y - \beta^\top z\vert^\alpha)  f^X(z) dz }  \\
& \le &  \alpha d^{-\alpha} \vert y - \beta^\top x \vert^{\alpha -1} \exp\Big( d^{-\alpha} (\sup_{z \in \overline{\mathcal{B}}(0, B)} \vert y - \beta^\top z\vert^\alpha  - \vert y - \beta^\top x\vert^\alpha) \Big) \\
& = & \alpha d^{-\alpha} \vert y - \beta^\top x \vert^{\alpha -1} \exp\left( d^{-\alpha}  \vert y \vert^\alpha \left( \left \vert 1 - \frac{\beta^\top z_y}{y}  \right \vert^\alpha - \left \vert  1 - \frac{\beta^\top x}{y}  \right \vert^\alpha \right) \right) 
\end{eqnarray*}
where $z_y$ denotes the vector where the supremum of $\vert y - \beta^\top z\vert^\alpha $ is achieved in $\overline{\mathcal{B}}(0, B)$.  
Using Taylor expansion, it follows that
\begin{eqnarray*}
 \left \vert 1 - \frac{\beta^\top z_y}{y}  \right \vert^\alpha \le \left(1 + \frac{C}{\vert y \vert}  \right)^\alpha = 1 + \alpha (1 + \theta^\ast_1)^{\alpha -1} \frac{C}{\vert y \vert}
\end{eqnarray*}
and 
\begin{eqnarray*}
 \left \vert 1 - \frac{\beta^\top x}{y}  \right \vert^\alpha \ge \left(1 - \frac{C}{\vert y \vert}  \right)^\alpha = 1 - \alpha (1 -\theta^\ast_2)^{\alpha -1} \frac{C}{\vert y \vert}
\end{eqnarray*}
where $\theta^\ast_1$ and $\theta^\ast_2$ depend on $\beta$ and $y$ and belong to  $(0, C/\vert y \vert) \subset (0,1)$ by our assumption about $y$.  Hence, 
\begin{eqnarray*}
\frac{\vert f'(y - \beta^\top x) \vert }{\int f(y- \beta^\top z) f^X(z) dz} 
& \le   & \alpha d^{-\alpha} \vert y - \beta^\top x \vert^{\alpha -1} \exp\left( \alpha d^{-\alpha} C \vert y \vert^{\alpha - 1} ((1 + \theta^\ast_1)^{\alpha -1} + (1 - \theta^\ast_2)^{\alpha -1})\right) \\
& \le & \alpha  d^{-\alpha}\vert y - \beta^\top x \vert^{\alpha -1} \exp\left( \alpha d^{-\alpha}   \tilde C \vert y \vert^{\alpha - 1} \right)
\end{eqnarray*}
where $\tilde C = C( 1 + 2^{\alpha -1})$. Now, let $y:  \vert y \vert \le C$. Then,
\begin{eqnarray*}
\sup_{\vert y \vert \le C, x \in \overline{\mathcal{B}}(0, B), \beta \in \overline{\mathcal{B}}(0, R^\ast)}  \frac{\vert f'(y - \beta^\top x) \vert }{\int f(y- \beta^\top z) f^X(z) dz}  \le \frac{\sup_{\vert y \vert \le C, \vert t \vert \le C} \vert f'(y -t) \vert }{\inf_{\vert y \vert \le C, \vert t\vert \le C} f(y- t)},
\end{eqnarray*}
where
\begin{eqnarray*}
 \sup_{\vert y \vert \le C, \vert t \vert \le C} \vert f'(y - t) \vert \le c_\alpha \alpha d^{-\alpha}  (2C)^{\alpha-1} 
 \end{eqnarray*}
and 
\begin{eqnarray*}
 \inf_{\vert y \vert \le C, \vert t \vert \le C} f(y -t)  \ge c_\alpha \exp(-d^{-\alpha} (2C)^\alpha)
\end{eqnarray*}
using the triangle inequality, the fact that the functions $u \mapsto u^{\alpha -1}$ and $u \mapsto u^\alpha$ are increasing on $(0, \infty)$ and $\exp(-x) \le 1$ for all $x \ge 0$.  Hence,
$$
\sup_{y:  \vert y \vert \le C, x \in \overline{\mathcal{B}}(0, B), \beta \in \overline{\mathcal{B}}(0, R^\ast)}  \frac{\vert f'(y - \beta^\top x) \vert }{\int f(y- \beta^\top z) f^X(z) dz}  \le M
$$
with
\begin{eqnarray*}
M := \alpha d^{-\alpha} (2C)^{\alpha -1} \exp(d^{-\alpha} (2C)^\alpha).
\end{eqnarray*}
It follows that for all $y \in \mathbb R$ and $\beta: \Vert \beta \Vert \le R^\ast$ and $x: \Vert x \Vert \le B$ 
\begin{eqnarray*}
\frac{ \vert f'(y - \beta^\top x) \vert }{\int f(y- \beta^\top z) f^X(z) dz} \le G_1(y)
\end{eqnarray*}
where
\begin{eqnarray}\label{G1}
G_1(y)  = 
\begin{cases}
 M,  \text{ if $\vert y \vert \le C $}   \\
 (2^{\alpha-2} \vee 1) \alpha d^{-\alpha} \left( \vert y \vert^{\alpha -1}  + C^{\alpha - 1}  \right) \exp\left( \alpha d^{-\alpha}   \tilde C \vert y \vert^{\alpha - 1} \right),  \text{ if $\vert y \vert > C $}.
\end{cases}
\end{eqnarray}
Now, we turn to the second term in the gradient $\nabla g_{\beta, y}(x)$ and define the function
\begin{eqnarray*}
k_{\beta, y}(x)  = \frac{f(y - \beta^\top x) \int  
 z \vert f'(y- \beta^\top z)  \vert  f^X(z) dz }{\left(\int f(y- \beta^\top z) f^X(z) dz\right)^2}.     
\end{eqnarray*}
We have that 
\begin{eqnarray}\label{kbeta}
k_{\beta, y}(x)  & = &  \frac{f(y - \beta^\top x)}{\int f(y- \beta^\top z) f^X(z) dz} \cdot \frac{\int z \lvert f'(y - \beta^\top z)\rvert f^X(z) dz}{\int f(y- \beta^\top z) f^X(z) dz}.
\end{eqnarray}
Using similar arguments as above we can show that 
\begin{eqnarray*}
\frac{f(y - \beta^\top x)}{\int f(y- \beta^\top z) f^X(z) dz} \le G_2(y)
\end{eqnarray*}
where 
\begin{eqnarray}\label{G2}
G_2(y)  = 
\begin{cases}
M'=  \exp(d^{-\alpha} (2C)^{\alpha}) ,  \text{ if $\vert y \vert \le C $}   \\
\exp\left( \alpha d^{-\alpha}   \tilde C \vert y \vert^{\alpha - 1} \right), \  \text{ if $\vert y \vert > C $}.
\end{cases}
\end{eqnarray}
Also, using the fact that
\begin{eqnarray*}
&& \frac{\int \Vert z \Vert \cdot \vert f'(y - \beta^\top z) \vert f^X(z) dz}{\int f(y- \beta^\top z) f^X(z) dz}  \le   B \frac{\sup_{z \in \overline{\mathcal B}(0, B)} \vert f'(y - \beta^\top z)  \vert }{\int f(y- \beta^\top z) f^X(z) dz}  \\
\end{eqnarray*}
we can use similar arguments as above to show that 
\begin{eqnarray*}
\frac{\int \Vert z \Vert \cdot \vert f'(y - \beta^\top z) \vert f^X(z) dz}{\int f(y- \beta^\top z) f^X(z) dz} \le B G_1(y). \ \ \ 
\end{eqnarray*}
It follows that
\begin{eqnarray*}
\Vert \nabla g_{\beta, y}  \Vert \le G(y) = B G_1(y)\left(1+G_2(y)\right) \ \ \ 
\end{eqnarray*}
and hence for a fixed $y$ and for all $x \in \overline{\mathcal B}(0, B)$ and $\beta, \beta'  \in \overline{\mathcal B}(0, R^\ast)$, we have that
\begin{eqnarray*}
 \vert g_{\beta, y}(x)  - g_{\beta', y}(x)  \vert \le \Vert \beta' - \beta  \Vert \ G(y).
\end{eqnarray*}
Consider now the new class $\widetilde{\mathcal G}_y$ of $\tilde{g}_{\beta,y}$ defined as
\begin{align*}
    \tilde{g}_{\beta,y} (x) = 
    \frac{g_{\beta,y} (x) -g_{0,y} (x) }{(R^\ast \vee 1) G(y)} =  \frac{g_{\beta,y} (x) - 1 }{(R^\ast \vee 1) G(y)}.
\end{align*}
Since this class admits the constant 1 as an envelope, we can use similar arguments as above to show that 

\begin{eqnarray*}
\mathbb E_{\mathbb P^{\tilde X}}\left[ \Vert \mathbb G^{\tilde X}_n \Vert_{\widetilde{\mathcal G}_y} \right]  = \mathbb E_{\mathbb P^{\tilde X}} \left[\sup_{\tilde g \in \widetilde{\mathcal G}_y} \left \vert \int \tilde{g}(x) d\mathbb G^{\tilde X}_n(x) \right \vert  \right]  \lesssim 1.   
\end{eqnarray*}
We have now 
\begin{align*}
\mathbb{C}_{n,\beta,1}(y) &= \left \vert  \int (R^\ast \vee 1) \tilde{g}_{\beta,y}(x) G(y) \ d(\mathbb{P}_n^{\tilde{X}} -\mathbb{P}^{\tilde{X}})(x) + \int 1 \ d(\mathbb{P}_n^{\tilde{X}} -\mathbb{P}^{\tilde{X}})(x)\right \vert   \\
    &= (R^\ast \vee 1)\frac{1}{\sqrt{n}}\left \vert  \int  \tilde{g}_{\beta,y}(x) G(y) \ d\mathbb{G}_n^{\tilde{X}}(x)\right \vert.
\end{align*}
Then,
\begin{eqnarray*}
\mathbb E_{\mathbb P^{\tilde X}}\left[\sup_{\beta \in \overline{\mathcal B}(0, R^\ast)} \int \mathbb C_{n, \beta, 1}(y)  d\mathbb P^{\tilde Y}(y)  \right] 
&\le  &  \mathbb E_{\mathbb P^{\tilde X} } \left[ \int \sup_{\beta \in \overline{\mathcal B}(0, R^\ast)} \mathbb C_{n, \beta, 1}(y)  d\mathbb P^{\tilde Y}(y) \right ]  \\
& \lesssim  &  \frac{1}{\sqrt n} \mathbb E_{\mathbb P^{\tilde X}} \left[\int \sup_{\tilde g \in \widetilde{\mathcal G}_y} \left \vert  \int \tilde g(x) d \mathbb G^{\tilde X}_n(x) \right \vert G(y)  d\mathbb P^{\tilde Y}(y) \right] \\
& = & \frac{1}{\sqrt n}  \int  \mathbb E_{\mathbb P^{\tilde X} } \left[\sup_{\tilde g \in \widetilde{\mathcal G}_y} \left \vert  \int \tilde g(x) d \mathbb G^{\tilde X}_n(x) \right \vert \right] G(y)  d\mathbb P^{\tilde Y}(y) \\
& \lesssim &  \frac{1}{\sqrt n} \int G(y)  \int f(y- \beta^\top_0 x) f^X(x) dx dy  \\
& \lesssim &  \frac{1}{\sqrt n}  \int G(y) \sup_{x \in \overline{\mathcal B}(0, B)} \exp(-d^{-\alpha} \vert y - \beta_0^\top x \vert^\alpha )  dy  \\
& \lesssim &  \frac{1}{\sqrt n}   \int G(y)  \exp(-d^{-\alpha} 2^{1-\alpha} \vert y \vert^\alpha)  dy < \infty.
\end{eqnarray*}
By Markov's inequality, we conclude that 
$$
\sup_{\beta \in \overline{\mathcal B}(0, R^\ast)} \int \mathbb C_{n, \beta, 1}(y)  d\mathbb P^{\tilde Y}(y) = O_{\mathbb P^{\tilde X}}\left(\frac{1}{\sqrt n} \right).  
$$
Now, we turn to the second term  $\int \mathbb C_{n, \beta, 2}(y) d\mathbb P^{\tilde Y}(y)$.  For $y \in \mathbb R$, we have that 
\begin{eqnarray*}
\mathbb C_{n, \beta,2}(y) & = & \left \vert \int \frac{f(y - \beta^\top x)}{n^{-1} \sum_{j=1}^n  f(y-\beta^\top X_j) } d(\mathbb P^{\tilde X}_n(x)  - \mathbb P^{\tilde X}(x))  \right \vert  \\
& = &  \left \vert \int s_{\beta, y}(x)  d(\mathbb P^{\tilde X}_n(x)  - \mathbb P^{\tilde X}(x))  \right \vert
\end{eqnarray*}
where the function $s_{\beta, y}$ can be written as 
\begin{eqnarray*}
s_{\beta, y}(x)  =  \frac{f(y - \beta^\top x)}{\sum_{j=1}^k w_j f(y-\beta^\top x_j) }
\end{eqnarray*}
for some integer $k \ge 1$, weights $w_i \in [0,1]$ such that $\sum_{j=1}^k w_j =1$ and $x_1, \ldots, x_k \in \overline{\mathcal B}(0, B)$.  To avoid a cumbersome notation, we will not explicitly write the dependence on $k$, $w_j$ and $x_j$ for $j=1, \ldots, k$.   The gradient of the partial function $\beta \mapsto s_{\beta, y}(x)$ is given by
\begin{eqnarray*}
\nabla s_{\beta, y}(x)  & =  &  - \frac{f'(y - \beta^\top x) x}{\sum_{j=1}^k w_j f(y-\beta^\top x_j) }  +  \frac{f(y - \beta^\top x) \sum_{j=1}^k w_j f'(y-\beta^\top x_j) x_j }{\left(\sum_{j=1}^k w_j f(y-\beta^\top x_j)\right)^2 }  \\ 
& = & - \frac{f'(y - \beta^\top x) x}{\sum_{j=1}^k w_j f(y-\beta^\top x_j) }  +  \frac{f(y - \beta^\top x)}{\sum_{j=1}^k w_j f(y-\beta^\top x_j)}  \frac{\sum_{j=1}^k w_j f'(y-\beta^\top x_j) x_j }{\sum_{j=1}^k w_j f(y-\beta^\top x_j)}.  
\end{eqnarray*}
It is easy to see that the functions involved in the gradient $\nabla s_{\beta, y}(x) $ have a very structure as the ones involved in $k_{\beta, y}(x)$ given above in (\ref{kbeta}).  Thus, we shall omit the proof of the fact that 
$$
\sup_{\beta \in \overline{\mathcal B}(0, R^\ast)} \int \mathbb C_{n, \beta, 2}(y)  d\mathbb P^{\tilde Y}(y) = O_{\mathbb P^{\tilde X}}\left(\frac{1}{\sqrt n} \right)  
$$
which in turn implies that $\sup_{\beta \in \overline{\mathcal B}(0, R^\ast)}  \vert \mathbb B_n(\beta) \vert = O_{\mathbb P^{\tilde X}}\left(\frac{1}{\sqrt n} \right)$.  This finishes the proof that the uniform consistency in (\ref{UC}) holds true.  \\
\bigskip

\par \noindent Next, we need to show that the regression model based on combining the matched and unmatched variables is identifiable.  Firstly, we show that $\ell(\beta) \leq \ell(\beta_0)$ for all $\beta \in \mathbb{R}^p$. Recall that
\begin{align*}
    \ell(\beta) &= \frac{1}{\lambda + 1} \int \log \left( \int f(y - \beta^\top x) d \mathbb P^{\tilde{X}}(x)  \right)  d\mathbb P^{\tilde Y}(y) +  \    \frac{\lambda}{\lambda + 1}  \int \log f(y - \beta^\top x)  d\mathbb P(x, y) \\
    &= \frac{1}{\lambda + 1} \int \log \left( \int f(y - \beta^\top x) f^X(x)dx \right) \int f(y - \beta_0^\top x)f^X(x)dx dy \\
    & \ \ +  \frac{\lambda}{\lambda + 1}  \int \log \left(f(y - \beta^\top x) \right) f(y - \beta_0^\top x)f^X(x)dx dy
\end{align*}
Using Jensen's inequality applied to the convex function $-\log $, we have for the matched part
\begin{align*}
    & \int \log\left( \frac{f(y - \beta^\top x)}{f(y - \beta^\top_0 x)} \right) f(y - \beta_0^\top x)f^X(x)dx dy \\
    & \leq  \log \left(\int  \frac{f(y - \beta^\top x)}{f(y - \beta_0^\top x)}  f(y - \beta_0^\top x)f^X(x)dx dy \right) \\
     & = \log \left(\int f(y - \beta^\top x) f^X(x)dx dy \right)\\ 
    &= 0 
\end{align*}
since $\int f(y - \beta^\top x) f^X(x)dx dy = 1$. 
Similarly, we have for the unmatched part
\begin{align*}
    &\int \log \left( \frac{\int f(y - \beta^\top x) f^X(x)dx}{\int f(y - \beta_0^\top x) f^X(x)dx} \right) \int f(y - \beta_0^\top x) f^X(x) dxdy\\
    &\leq \log \left( \int  \frac{\int f(y - \beta^\top x) f^X(x)dx}{\int f(y - \beta_0^\top x) f^X(x)dx}  \int f(y - \beta_0^\top x)f^X(x)dx dy\right)\\
    & = \log \left(\int f(y - \beta^\top x) f^X(x)dx dy \right) \\
    & = 0.
\end{align*}
We conclude that $\ell(\beta)\leq \ell(\beta_0)$ for all $\beta \in \mathbb{R}^p$.

Next, we show that the previous inequality is strict, that is $\ell(\beta) < \ell(\beta_0)$ for all $\beta \in \mathbb R^p \setminus \{\beta_0\}$. Let $\varphi$ be strictly convex function. Then, for any integrable random variable  $W$ which belongs to the domain of $\varphi$ almost surely the Jensen's inequality 
$$
\mathbb E[\varphi(W)] \ge \varphi(\mathbb E(W))
$$
is an equality if and only if $\mathbb P(W = \mathbb E(W)) = 1$.  Since $- \log $ is strictly convex on $(0, \infty)$, it holds that 
$$
\int \log\left( \frac{f(y - \beta^\top x)}{f(y - \beta^\top_0 x)} \right) f(y - \beta_0^\top x)f^X(x)dx dy = 0
$$
if and only if
$$
\mathbb{P}\left(\frac{f(Y - \beta^\top X)}{f(Y - \beta_0^\top X)} = \mathbb E \left[ \frac{f(Y - \beta^\top X)}{f(Y - \beta_0^\top X)}\right]  \right) = 1. 
$$
Since $\displaystyle \mathbb E \left[ \frac{f(Y - \beta^\top X)}{f(Y - \beta_0^\top X)}\right] =1$, we conclude that 
\begin{eqnarray*}
f(y - \beta^\top_0 x)  = f(y - \beta^\top x)
\end{eqnarray*}
for almost every $(x,y)  \in \mathcal X \times \mathbb R$, where $\mathcal X$ denotes the support of $X$.    Now, fix $x \in \mathcal X$ and write $t =y - \beta^\top_0 x $  and $a =  (\beta_0 - \beta)^\top x$. Then, the preceding inequality can be re-written as 
\begin{eqnarray}\label{cyc}
f(t)  = f(t + a)
\end{eqnarray}
for almost all $t \in \mathbb R$. This implies that $a =0$, Indeed, suppose that $a \ne 0$. Without loss of generality, we can assume that $a > 0$.  Then,
\begin{eqnarray*}
\int_{-\infty}^\infty f(t)dt & =  & \int_{-\infty}^a  f(t) 
dt + \sum_{k=1}^\infty \int_{ka}^{(k+1)a}    f(t) dt,
\end{eqnarray*}
where $\int_{ka}^{(k+1)a}    f(t) dt = \int_0^a f(t - ka) dt = \int_0^a f(t) dt$, since (\ref{cyc}) implies that $f(t) = f(t - ka)$ for all $k \in \mathbb N$.  This implies that $\int_0^a f(t) dt =0$ because otherwise we would have $\int_{\mathbb R} f(t) dt = \infty$, which is impossible.  It follows that $a =0$ and hence
$$
(\beta_0 - \beta)^\top x =0
$$
for almost all $x \in \mathcal X$.  This means that
$$
\mathbb P( u^\top  X  =0)  =1
$$
with $u = \beta_0  - \beta$.   By Assumption (A2), this implies that $u =0$ and hence $\beta = \beta_0$. 
 We conclude that $\ell(\beta) < \ell(\beta_0)$ for all $\beta \ne  \beta_0$.  In fact, if there existed $\beta  \ne \beta_0$ such that  $\ell(\beta) = \ell(\beta_0)$, then by the already proved  fact that $\ell(\beta) \le \ell(\beta_0)$, we must have equality for both the matched and unmatched parts. In particular we must have that 
 \begin{align*}
    & \int \log f(y - \beta^\top x) f(y - \beta_0^\top x)f^X(x)dx dy  =  \int \log f(y - \beta^\top_0 x) f(y - \beta_0^\top x)f^X(x)dx dy
 \end{align*}
and we just proved above that this is impossible. 

Now, let $r > 0$. we show next that 
$$
\sup_{\beta \in \mathcal O_r} \ell(\beta) < \ell(\beta_0).
$$


As shown above, there exists $R^\ast$ so that with probability 1 the maximizer  of the empirical likelihood can be restricted to the ball $\overline{\mathcal{B}}(0, R^\ast)$ for $n$ and $m$ large enough.  Therefore, showing the result of the theorem is equivalent to showing that
$$
\sup_{\beta \in \mathcal O_r \cap \overline{\mathcal{B}}(0, R^\ast)} \ell(\beta) < \ell(\beta_0).
$$
To avoid trivialities, we can take $r \in (0,1)$ so that $\mathcal O_r \cap \overline{\mathcal{B}}(0, R^\ast) \ne \emptyset$. In fact, if theorem is shown for $r$ small, then it will continue to hold for larger values of $r$. Hence, we assume in what follows that $r$. Using the fact that $\mathcal{O}_r \subset \overline{\mathcal{O}}_{r/2} =  \{ \beta \in \mathbb R^p:  \Vert \beta - \beta_0 \Vert \ge r/2 \}$, it follows that
\begin{eqnarray*}
\sup_{\beta \in \mathcal O_r \cap \overline{\mathcal{B}}(0, R^\ast)} \ell(\beta)  \le \sup_{\beta \in \overline{\mathcal O}_{r/2} \cap \overline{\mathcal{B}}(0, R^\ast)} \ell(\beta).
\end{eqnarray*}
The set $\overline{\mathcal O}_{r/2} \cap \overline{\mathcal{B}}(0, R^\ast)$ is compact as it is the intersection of a closed set and a compact set. Since the function $\beta \mapsto \ell(\beta)$ is continuous, it follows that $\beta \mapsto \ell(\beta)$ attains its supremum over this set at some $\beta^\ast \ne \beta_0$. Using the result obtained above, it holds that 
$$
\sup_{\beta \in \overline{\mathcal O}_{r/2} \cap \overline{\mathcal{B}}(0, R^\ast)} \ell(\beta) = \ell(\beta^\ast) < \ell(\beta_0)
$$
and the claim of the proposition follows.  


By \cite[Corollary 3.2.3]{aadbookE2}, it follows that $\widehat \beta_{n,m}  \to_{\mathbb P \otimes \mathbb P^{\tilde X} \otimes \mathbb P^{\tilde Y}} \beta_0$. This finishes the proof.
\end{proof}

\subsection{Proofs for Section \ref{WC}}

\begin{proof}[Proof of Theorem \ref{AsympNorm}]
To simplify the notation, we will write
$$
\frac{\partial}{\partial\beta}{\ell}_{n,m}(\beta)|_{\beta = \tilde \beta}
$$
and 
$$
\frac{\partial^2}{\partial\beta \partial \beta^\top}{\ell}_{n,m}(\beta)|_{\beta = \tilde \beta}
$$
as
$\dot{\ell}_{n, m}(\tilde{\beta})$ and  $\ddot{\ell}_{n, m}(\tilde{\beta})$  respectively. Since $\widehat{\beta}_{n,m}$ is maximizer of $\ell_{n,m}(\beta)$, we have then
\begin{align*}
     0 &= \dot{\ell}_{n,m}(\widehat{\beta}_{n,m}) \\
    &= \dot{\ell}_{n,m}(\beta_0) + \ddot{\ell}_{n,m}(\beta_0) (\widehat{\beta}_{n,m} - \beta_0)  + o_{\mathbb P \otimes \mathbb P^{\tilde X} \otimes \mathbb P^{\tilde Y}}\left(\Vert \widehat{\beta}_{n,m} - \beta_0\Vert \right).
\end{align*}
Let us assume for now that the matrix $\ddot{\ell}_{n,m}(\beta_0)$ is invertible (this will be proved below in Theorem \ref{Remainder2}). By Theorem \ref{Consis}, we know that $\widehat{\beta}_{n,m}\overset{\mathbb P \otimes \mathbb P^{\tilde X} \otimes \mathbb P^{\tilde Y}}{\rightarrow}\beta_0$ and hence
\begin{eqnarray*}
    \sqrt{m+n}(\widehat{\beta}_{n,m} - \beta_0) = -\left(\ddot{\ell}_{n,m}(\beta_0)\right)^{-1}\sqrt{m + n} \ \dot{\ell}_{n,m}(\beta_0) + o_{\mathbb P \otimes \mathbb P^{\tilde X} \otimes \mathbb P^{\tilde Y}}(1).
\end{eqnarray*}
We have that
\begin{eqnarray}\label{score}
    \dot{\ell}_{n,m}(\beta_0) &  =   &  - \frac{1}{n+m} \sum_{j=1}^n \frac{\int x f'(\widetilde Y_j - \beta_0^\top x) d\mathbb P^{\tilde X}_n(x)}{\int  f(\widetilde Y_j - \beta_0^\top x) d\mathbb P^{\tilde X}_n(x)}  - \frac{1}{n+m} \sum_{k=1}^m  \frac{ X_k f'(Y_k- \beta_0^\top X_k)}{f(Y_k- \beta_0^\top X_k)}  \notag \\
    &  = &   - \frac{n}{n+m} \frac{1}{n} \sum_{j=1}^n \frac{\int x f'(\widetilde Y_j - \beta_0^\top x) d\mathbb P^{\tilde X}_n(x)}{\int  f(\widetilde Y_j - \beta_0^\top x) d\mathbb P^{\tilde X}_n(x)}  - \frac{m}{n+m} \int \frac{ x f'(y- \beta_0^\top x)}{f(y- \beta_0^\top x)} d\mathbb P_m(x, y) \notag. \\
    && 
    \end{eqnarray}
We start with the second term and note that 
\begin{eqnarray*}
\int \frac{ x f'(y- \beta_0^\top x)}{f(y- \beta_0^\top x)} d\mathbb P(x, y) =  \int  x f'(y- \beta_0^\top x) f^X(x) dx dy  = \mathbb E(X) \int f'(t) dt = 0
\end{eqnarray*}
using the change of variable $t = y - \beta_0^\top x$.  Using the same change of variable, we compute
\begin{eqnarray*}
 \int \frac{ x x^\top (f'(y- \beta_0^\top x))^2}{(f(y- \beta_0^\top x))^2} d\mathbb P(x, y)  & =  &     \left(\int\frac{( f'(t))^2}{f(t)}dt\right) \int (xx^\top) f^X(x) dx  \\
 & = &  \left(\int\frac{( f'(t))^2}{f(t)}dt\right)  \mathbb E[X X^\top]  = \Sigma_2.
\end{eqnarray*}
By the Central Limit Theorem, it follows that 
\begin{eqnarray*}
\sqrt m \int \frac{ x f'(y- \beta_0^\top x)}{f(y- \beta_0^\top x)} d\mathbb P_m(x, y)  \to_d \mathcal{N}(0, \Sigma_2)   
\end{eqnarray*}
and hence
\begin{eqnarray*}
\sqrt{m+n} \frac{m}{n+m} \int \frac{ x f'(y- \beta_0^\top x)}{f(y- \beta_0^\top x)} d\mathbb P_m(x, y)    & = &  \sqrt{\frac{m}{m+n}} \sqrt m  \int \frac{ x f'(y- \beta_0^\top x)}{f(y- \beta_0^\top x)} d\mathbb P_m(x, y)  \\
& \to_d &  \mathcal N\left(0,  \frac{\lambda}{1+\lambda} \Sigma_2 \right). 
\end{eqnarray*}
Note that 
\begin{eqnarray*}
\int \frac{ (f'(t))^2}{f(t)} dt & = &  \alpha^2 d^{-2\alpha} \int   t^{2\alpha-2} f(t) dt =  \alpha^2 d^{-2\alpha} \mathbb E[\epsilon^{2\alpha -2}]. 
\end{eqnarray*}
More explicitly, 
\begin{eqnarray*}
\int   t^{2\alpha-2} f(t) dt & =  &  c_\alpha  \int  t^{2\alpha-2} \exp(-d^{-\alpha} \vert t \vert^\alpha) dt \\
& = &  2 c_\alpha \int_0^\infty  t^{2\alpha-2} \exp(-d^{-\alpha} t^\alpha) dt \\
& = &  \frac{2 c_\alpha}{\alpha}  \int x^{1 - 1/\alpha } \exp(-d^{-\alpha} x)  dx  \\
& = &  \frac{2 c_\alpha}{\alpha} \ d^{2\alpha -1}  \ \Gamma\left(2 - \frac{1}{\alpha}\right).
\end{eqnarray*}
Using the fact that $c_\alpha =  \alpha / (2d \Gamma(1/\alpha))$ we obtain
\begin{eqnarray*}
\int \frac{(f'(t))^2}{f(t)} dt  =   \frac{\alpha^2}{d^2}  \frac{\Gamma\left(2 - \frac{1}{\alpha}\right)}{\Gamma\left(\frac{1}{\alpha}\right)}.
\end{eqnarray*}
If $\epsilon \sim \mathcal{N}(0,\sigma^2)$, then the integral specializes to $1/\sigma^2$. 

\bigskip

Now, we turn to the first term in (\ref{score}).  It follows from Theorem \ref{Remainder} that 
\begin{eqnarray*}
\frac{1}{\sqrt n} \sum_{j=1}^n \frac{\int x f'(\widetilde Y_j - \beta_0^\top x) d\mathbb P^{\tilde X}_n(x)}{\int  f(\widetilde Y_j - \beta_0^\top x) d\mathbb P^{\tilde X}_n(x)}  \to_d \mathcal N(0, \Gamma_1  +  \Gamma_2) 
\end{eqnarray*}
which implies that 
\begin{eqnarray*}
\sqrt{m+n} \frac{n}{n+m} \frac{1}{n} \sum_{j=1}^n \frac{\int x f'(\widetilde Y_j - \beta_0^\top x) d\mathbb P^{\tilde X}_n(x)}{\int f(\widetilde Y_j - \beta_0^\top x) d\mathbb P^{\tilde X}_n(x)}  & = &  \sqrt{\frac{n}{n+m}} \frac{1}{\sqrt n} \sum_{j=1}^n \frac{\int x f'(\widetilde Y_j - \beta_0^\top x) d\mathbb P^{\tilde X}_n(x)}{\int f(\widetilde Y_j - \beta_0^\top x) d\mathbb P^{\tilde X}_n(x)}  \\
& \to_d &  \mathcal{N}\left(0,  \frac{1}{\lambda +1} (\Gamma_1 + \Gamma_2) \right). 
\end{eqnarray*}
Since the matched and unmatched samples are independent, we conclude the following weak convergence result
\begin{eqnarray*}
\sqrt{m+n} \ \dot{\ell}_{n,m}(\beta_0) \to_d \mathcal N\left(0,     \frac{1}{\lambda +1} (\Gamma_1 + \Gamma_2) + \frac{\lambda}{\lambda +1} \Sigma_2 \right).    
\end{eqnarray*}
Finally, it follows from Theorem \ref{Remainder2} that 
\begin{eqnarray*}
\ddot{\ell}_{n, m}(\beta_0)  \to_{\mathbb P \otimes \mathbb P^{\tilde X} \otimes \mathbb P^{\tilde Y}} -\frac{1}{1+\lambda} \Gamma_1 -\frac{\lambda}{1+\lambda} \Sigma_2.
\end{eqnarray*}
Note that $\Sigma_2$ is positive definite since $X$ admits an absolutely continuous distribution. This in turn implies that the matrix
$$
\frac{1}{1+\lambda} \Gamma_1 +  \frac{\lambda}{1+\lambda} \Sigma_2
$$
is also positive definite. In fact, if $u \in \mathbb R^p$ is an eigenvector corresponding to the eigenvalue $0$, then we must have 
$$
\frac{1}{1+\lambda} u^\top \Gamma_1 u +  \frac{\lambda}{1+\lambda} u^\top \Sigma_2 u = 0
$$
and hence $u^\top \Gamma_1 u = u^\top \Sigma_2 u = 0$ since $\Gamma_1$ and $\Sigma_2$ are semi-positive. Therefore, we must have $u = 0$. By the weak convergence obtained above, the continuous mapping theorem (applied to the map $M \mapsto M^{-1}$ defined on the space of positive definite matrices $M \in \mathbb R^{p \times p}$) and Slutsky's theorem we conclude that 
$$
\sqrt{m+n} (\widehat \beta_{m,n} - \beta_0)  \to_d \left(\frac{1}{1+\lambda} \Gamma_1 +  \frac{\lambda}{1+\lambda} \Sigma_2\right)^{-1}   N\left(0,     \frac{1}{\lambda +1} (\Gamma_1 + \Gamma_2) + \frac{\lambda}{\lambda +1} \Sigma_2 \right)
$$
which yields the result.

\end{proof}

\begin{proof}[Proof of Theorem \ref{gain}.]
Since $\epsilon \sim \mathcal{N}(0, \sigma_\epsilon^2)$, the matched MLE and the OLSE are equal. Then, the statistical gain $G$ is defined as 
\begin{align*}
   G^{-1}= \frac{\text{Vol}_{\text{SSL}}}{\text{Vol}_{\text{OLS}}} &= \sqrt{\frac{\det{\Tilde{\Sigma}_{\text{SSL}}}}{\det{\Sigma_{\text{OLS}}}}} = \frac{1}{\sqrt{\frac{\det{\Sigma_{\text{OLS}}}}{\det{\Tilde{\Sigma}_{\text{SSL}}}}}} \\
     & = \frac{1}{\sqrt{\det{\left(\Sigma_2^{-1}\left(\frac{1}{\lambda} \Gamma_1 + \Sigma_2\right) \left(\frac{1}{\lambda}\left(\Gamma_1+\Gamma_2\right)+\Sigma_2\right)^{-1}\left(\frac{1}{\lambda} \Gamma_1 +\Sigma_2\right)\right)}}}\\
    & = \frac{\sqrt{\det{\left(\mathbbm{1}_{p\times p}+
    \frac{1}{\lambda}\Sigma_2^{-1}\left(\Gamma_1+\Gamma_2\right) \right)}}}{\det{\left(\mathbbm{1}_{p\times p} + \frac{1}{\lambda}\Sigma_2^{-1}\Gamma_1\right)}}.
\end{align*}
Since $\Sigma_2$ is positive definite, we have $\Sigma_2 = V_2 \Lambda_2 V_2^\top$ and $\Sigma_2 = \Sigma_2^{\frac{1}{2}}\Sigma_2^{\frac{1}{2}}$ with $\Sigma_2^{\frac{1}{2}} = V_2 \Lambda_2^{\frac{1}{2}} V_2^\top$. Then it follows 
\begin{align*}
    \Sigma_2^{-1}\Gamma_1 = \Sigma_2^{-\frac{1}{2}}\left( \Sigma_2^{-\frac{1}{2}}\Gamma_1 \Sigma_2^{-\frac{1}{2}}\right)\Sigma_2^{\frac{1}{2}}.
\end{align*}
Since $\Gamma_1$ and $\Sigma_2^{\frac{1}{2}}$ are symmetric, the matrix $M :=  \Sigma_2^{-\frac{1}{2}}\Gamma_1 \Sigma_2^{-\frac{1}{2}}$  is also symmetric. Also, the fact that $\Gamma_1$ is positive semi-definite implies that $M$ is positive semi-definite. It follows that $M = V_M \Lambda_M V_M^\top$ by its eigenvalue decomposition. Thus, we have
\begin{align*}
    \Sigma_2^{-1}\Gamma_1 = \left(\Sigma_2^{-\frac{1}{2}}V_M\right) \Lambda_M \left(\Sigma_2^{\frac{1}{2}}V_M\right)^\top := P \Lambda_M P^{-1}.
\end{align*}
It follows that
\begin{align*}
    \det{\left(\mathbbm{1}_{p\times p}+\frac{1}{\lambda}\Sigma_2^{-1}\Gamma_1\right)} =  \det{\left(\mathbbm{1}_{p\times p}+\frac{\Lambda_M}{\lambda}\right)} = \prod_{i=1}^p (1+\frac{k_i}{\lambda})
\end{align*}
with $k_1 \geq k_2 \geq \cdots \geq k_p \geq 0$ the ordered eigenvalues of $M=  \Sigma_2^{-\frac{1}{2}}\Gamma_1 \Sigma_2^{-\frac{1}{2}}$. Similarly, we have 
\begin{align*}
    \det{\left(\mathbbm{1}_{p\times p}+
    \frac{1}{\lambda}\Sigma_2^{-1}\left(\Gamma_1+\Gamma_2\right) \right)} = \prod_{i = 1}^p(1+\frac{h_i}{\lambda})
\end{align*}
with $h_1 \geq h_2 \geq \cdots \geq h_p \geq 0$ the ordered eigenvalues of $N :=  \Sigma_2^{-\frac{1}{2}}\left(\Gamma_1+\Gamma_2\right) \Sigma_2^{-\frac{1}{2}}$. Therefore, we can write that
\begin{align*}
    \frac{\text{Vol}_{\text{SSL}}}{\text{Vol}_{\text{OLS}}}
    = \frac{\sqrt{\prod_{i = 1}^p(1+\frac{h_i}{\lambda})}}{\prod_{i=1}^p(1+\frac{k_i}{\lambda})}.
\end{align*}
We start with the case where  $X\sim \mathcal{N}(\mu_X,\sigma_X^2 \mathbbm{1}_{p\times p})$ with $\mu_X \in \mathbb{R}^p\setminus \{0\}$ and $\sigma_X > 0$. For the sake of a less cumbersome notation, we will write $\mu$ for $\mu_X$. We can have that
\begin{align*}
    \Sigma_2 &= \frac{1}{\sigma_\epsilon^2}\left(\mu\mu^\top + \sigma_X^2\mathbbm{1}_{p\times p}\right) \\
    &= \frac{1}{\lVert\mu\rVert}\begin{pmatrix}
        \vrule & \vrule & & \vrule \\
        \mu &  v_1 & \cdots &  v_{p-1}\\
        \vrule & \vrule & & \vrule \\
    \end{pmatrix} \begin{pmatrix}
        \frac{\lVert\mu\rVert^2 + \sigma_X^2}{\sigma_\epsilon^2} & 0 & \cdots & 0 \\
        0 & \frac{\sigma_X^2}{\sigma_\epsilon^2}  & \cdots & 0\\
        \vdots & \vdots & & \vdots \\
        0 & 0 & \cdots & \frac{\sigma_X^2}{\sigma_\epsilon^2}
    \end{pmatrix} \left(\frac{1}{\lVert\mu\rVert}\begin{pmatrix}
        \vrule & \vrule & & \vrule \\
        \mu &  v_1 & \cdots &  v_{p-1}\\
        \vrule & \vrule & & \vrule \\
    \end{pmatrix} \right)^\top
\end{align*}
where $\frac{\mu}{\lVert\mu\rVert},\,\frac{v_1}{\lVert\mu\rVert},\,\cdots,\,\frac{v_{p-1}}{\lVert\mu\rVert}$ form an orthonormal basis of $\mathbb{R}^p$. It follows that 
\begin{align*}
    \Sigma_2^{-\frac{1}{2}} 
    &= \frac{1}{\lVert\mu\rVert}\begin{pmatrix}
        \vrule & \vrule & & \vrule \\
        \mu &  v_1 & \cdots &  v_{p-1}\\
        \vrule & \vrule & & \vrule \\
    \end{pmatrix} \begin{pmatrix}
        \frac{\sigma_\epsilon}{\sqrt{\lVert\mu\rVert^2 + \sigma_X^2}} & 0 & \cdots & 0 \\
        0 & \frac{\sigma_\epsilon}{\sigma_X}  & \cdots & 0\\
        \vdots & \vdots & & \vdots \\
        0 & 0 & \cdots & \frac{\sigma_\epsilon}{\sigma_X}
    \end{pmatrix} \left(\frac{1}{\lVert\mu\rVert}\begin{pmatrix}
        \vrule & \vrule & & \vrule \\
        \mu &  v_1 & \cdots &  v_{p-1}\\
        \vrule & \vrule & & \vrule \\
    \end{pmatrix} \right)^\top \\
    &= \frac{\sigma_\epsilon}{\lVert\mu\rVert^2}\left(\frac{1}{\sqrt{\lVert\mu\rVert^2+\sigma_X^2}}\mu\mu^\top+\frac{1}{\sigma_X}(\lVert\mu\rVert^2\mathbbm{1}_{p\times p}-\mu\mu^\top)\right) \\
    &= \frac{\sigma_\epsilon}{\lVert\mu\rVert^2}(\frac{1}{\sqrt{\lVert\mu\rVert^2+\sigma_X^2}}-\frac{1}{\sigma_X})\mu\mu^\top + \frac{\sigma_\epsilon}{\sigma_X}\mathbbm{1}_{p\times p}.
\end{align*}
Since $Y \sim \mathcal{N}(\beta_0^\top\mu, \sigma_X^2\lVert\beta_0\rVert^2 + \sigma_\epsilon^2)$, we have $f^Y(y,\beta) = \frac{1}{\sqrt{2\pi(\sigma_X^2\lVert\beta\rVert^2+\sigma_\epsilon^2)}}\exp{\left(-\frac{(y-\beta^\top \mu)^2}{2(\sigma_X^2\lVert\beta\rVert^2+\sigma_\epsilon^2)}\right)}$. Hence,
\begin{align*}
    \nabla_\beta f^Y(y,\beta) &= - \frac{\sigma_X^2\beta}{\sigma_X^2 \lVert \beta \rVert^2+\sigma_\epsilon^2} \frac{1}{\sqrt{2\pi(\sigma_X^2\lVert\beta\rVert^2+\sigma_\epsilon^2)}}\exp{\left(-\frac{(y-\beta^\top \mu)^2}{2(\sigma_X^2\lVert\beta\rVert^2+\sigma_\epsilon^2)}\right)} \\
    & \ \ \ +   \frac{1}{\sqrt{2\pi(\sigma_X^2\lVert\beta\rVert^2+\sigma_\epsilon^2)}} \left(\frac{y - \beta^\top\mu}{\sigma_X^2\lVert\beta\rVert^2+\sigma_\epsilon^2}\mu + \frac{(y-\beta^\top\mu)^2}{(\sigma_X^2\lVert\beta\rVert^2+\sigma_\epsilon^2)^2} \sigma_X^2 \beta\right) \exp{\left(-\frac{(y-\beta^\top \mu)^2}{2(\sigma_X^2\lVert\beta\rVert^2+\sigma_\epsilon^2)}\right)} \\
    & = f^Y(y,\beta) \left(\frac{y - \beta^\top\mu}{\sigma_X^2\lVert\beta\rVert^2+\sigma_\epsilon^2}\mu +\left( \frac{(y-\beta^\top\mu)^2}{(\sigma_X^2\lVert\beta\rVert^2+\sigma_\epsilon^2)^2} - \frac{1}{\sigma_X^2 \lVert \beta \rVert^2+\sigma_\epsilon^2}\right)\sigma_X^2\beta\right).
\end{align*}
Thus,
\begin{align*}
     &\frac{\left(\nabla_\beta f^Y(y,\beta)|_{\beta=\beta_0}\right)\left(\nabla_\beta f^Y(y,\beta)|_{\beta=\beta_0}\right)^\top}{f^Y(y)} \\
     &= f^Y(y) \left(\frac{y - \beta_0^\top\mu}{\sigma_X^2\lVert\beta_0\rVert^2+\sigma_\epsilon^2}\mu +\left( \frac{(y-\beta_0^\top\mu)^2}{(\sigma_X^2\lVert\beta_0\rVert^2+\sigma_\epsilon^2)^2} - \frac{1}{\sigma_X^2 \lVert \beta_0 \rVert^2+\sigma_\epsilon^2}\right)\sigma_X^2\beta_0\right) \\
     & \ \ \ \ \left(\frac{y - \beta_0^\top\mu}{\sigma_X^2\lVert\beta_0\rVert^2+\sigma_\epsilon^2}\mu +\left( \frac{(y-\beta_0^\top\mu)^2}{(\sigma_X^2\lVert\beta_0\rVert^2+\sigma_\epsilon^2)^2} - \frac{1}{\sigma_X^2 \lVert \beta_0 \rVert^2+\sigma_\epsilon^2}\right)\sigma_X^2\beta_0\right)^\top.
\end{align*}
Therefore, we can write that
\begin{align*}
    \Gamma_1 &= \int \frac{\left(\int x f'(y- \beta^\top_0 x) f^X(x) dx \right)\left(\int x^\top f'(y- \beta^\top_0 x) f^X(x) dx \right) }{f^Y(y)} dy \\
    &=\int \frac{\left(\nabla_\beta f^Y(y,\beta)|_{\beta=\beta_0}\right)\left(\nabla_\beta f^Y(y,\beta)|_{\beta=\beta_0}\right)^\top}{f^Y(y)} dy \\
    &= \int \left(\frac{y-\beta_0^\top \mu}{\sigma_X^2\lVert\beta_0\rVert^2+\sigma_\epsilon^2}\mu + \sigma_X^2\left(\left(\frac{y-\beta_0^\top \mu}{\sigma_X^2\lVert\beta_0\rVert^2+\sigma_\epsilon^2}\right)^2-\frac{1}{\sigma_X^2\lVert\beta_0\rVert^2 + \sigma_\epsilon^2}\right)\beta_0\right) \\
    & \ \ \ \  \left(\frac{y-\beta_0^\top \mu}{\sigma_X^2\lVert\beta_0\rVert^2+\sigma_\epsilon^2}\mu + \sigma_X^2\left(\left(\frac{y-\beta_0^\top \mu}{\sigma_X^2\lVert\beta_0\rVert^2+\sigma_\epsilon^2}\right)^2-\frac{1}{\sigma_X^2\lVert\beta_0\rVert^2 + \sigma_\epsilon^2}\right)\beta_0\right)^\top f^Y(y) dy \\
    & = \mathbb{E}_Y[\left(Y-\beta_0^\top\mu\right)^2]\frac{\mu\mu^\top}{\left(\sigma_X^2\lVert\beta_0\rVert^2+\sigma_\epsilon^2\right)^2}+\mathbb{E}_Y[\left(\frac{(Y-\beta_0^\top\mu)^2}{(\sigma_X^2\lVert\beta_0\rVert^2+\sigma_\epsilon^2)^2}-\frac{1}{\sigma_X^2\lVert\beta_0\rVert^2+\sigma_\epsilon^2}\right)^2] \sigma_X^4 \beta_0 \beta_0^\top \\
    & \ \ \ + \mathbb{E}_Y[\left(\frac{(Y-\beta_0^\top\mu)^2}{(\sigma_X^2\lVert\beta_0\rVert^2+\sigma_\epsilon^2)^2}-\frac{1}{\sigma_X^2\lVert\beta_0\rVert^2+\sigma_\epsilon^2}\right)\left(\frac{Y-\beta_0^\top\mu}{\sigma_X^2\lVert\beta_0\rVert^2+\sigma_\epsilon^2}\right)]\sigma_X^2\left(\beta_0\mu^\top+\mu\beta_0^\top\right)\\
    &= \frac{1}{\sigma_X^2\lVert\beta_0\rVert^2+\sigma_\epsilon^2}\mu\mu^\top + \frac{2\sigma_X^4}{\left(\sigma_X^2\lVert\beta_0\rVert^2+\sigma_\epsilon^2\right)^2}\beta_0\beta_0^\top.
\end{align*}
Similarly, 
\begin{align*}
    \Gamma_2 &= \int \left(\int \frac{f^\epsilon(y-\beta_0^\top x) (\nabla_\beta f^Y(y,\beta)|_{\beta = \beta_0})}{f^Y(y)} dy\right) \left(\int \frac{f^\epsilon(y-\beta_0^\top x) (\nabla_\beta f^Y(y,\beta)|_{\beta = \beta_0})}{f^Y(y)} dy\right)^\top f^X(x) dx \\
    &= \int \left(\frac{\beta_0^\top(x-\mu)}{\sigma_X^2\lVert\beta_0\rVert^2+\sigma_\epsilon^2}\mu + \left(\frac{\sigma_\epsilon^2}{(\sigma_X^2\lVert\beta_0\rVert^2+\sigma_\epsilon^2)^2} -\frac{1}{\sigma_X^2\lVert\beta_0\rVert^2+\sigma_\epsilon^2} + \frac{(\beta_0^\top(x-\mu))^2}{(\sigma_X^2\lVert\beta_0\rVert^2+\sigma_\epsilon^2)^2} \right) \sigma_X^2\beta_0\right) \\
    & \ \ \ \ \left(\frac{\beta_0^\top(x-\mu)}{\sigma_X^2\lVert\beta_0\rVert^2+\sigma_\epsilon^2}\mu + \left(\frac{\sigma_\epsilon^2}{(\sigma_X^2\lVert\beta_0\rVert^2+\sigma_\epsilon^2)^2} -\frac{1}{\sigma_X^2\lVert\beta_0\rVert^2+\sigma_\epsilon^2} + \frac{(\beta_0^\top(x-\mu))^2}{(\sigma_X^2\lVert\beta_0\rVert^2+\sigma_\epsilon^2)^2} \right) \sigma_X^2\beta_0 \right)^\top f^X(x) dx \\
    & = \frac{\mathbb{E}_X[(\beta_0^\top(X-\mu))^2]}{(\sigma_X^2\lVert\beta_0\rVert^2+\sigma_\epsilon^2)^2}\mu\mu^\top + \left(\frac{\sigma_\epsilon^2}{(\sigma_X^2\lVert\beta_0\rVert^2+\sigma_\epsilon^2)^2}-\frac{1}{\sigma_X^2\lVert\beta_0\rVert^2+\sigma_\epsilon^2}\right)^2\sigma_X^4\beta_0\beta_0^\top\\
    & \ \ \ + 2\left(\frac{\sigma_\epsilon^2}{(\sigma_X^2\lVert\beta_0\rVert^2+\sigma_\epsilon^2)^2}-\frac{1}{\sigma_X^2\lVert\beta_0\rVert^2+\sigma_\epsilon^2}\right) \frac{\mathbb{E}_X[(\beta_0^\top (X-\mu))^2]}{(\sigma_X^2\lVert\beta_0\rVert^2+\sigma_\epsilon^2)^2}\sigma_X^4 \beta_0\beta_0^\top + \frac{\mathbb{E}_X[(\beta_0^\top (X-\mu))^4]}{(\sigma_X^2\lVert\beta_0\rVert^2+\sigma_\epsilon^2)^4}\sigma_X^4 \beta_0\beta_0^\top \\
    &= \frac{\sigma_X^2\lVert\beta_0\rVert^2}{(\sigma_X^2\lVert\beta_0\rVert^2+\sigma_\epsilon^2)^2}\mu\mu^\top +   \frac{2\sigma_X^8\lVert\beta_0\rVert^4}{(\sigma_X^2\lVert\beta_0\rVert^2+\sigma_\epsilon^2)^4}\beta_0\beta_0^\top.
\end{align*}
It follows that
\begin{align*}
    \Gamma_1+\Gamma_2 = \left(\frac{1}{\sigma_X^2\lVert\beta_0\rVert^2+\sigma_\epsilon^2}+\frac{\sigma_X^2\lVert\beta_0\rVert^2}{(\sigma_X^2\lVert\beta_0\rVert^2+\sigma_\epsilon^2)^2}\right)\mu\mu^\top+\left( \frac{2\sigma_X^4}{\left(\sigma_X^2\lVert\beta_0\rVert^2+\sigma_\epsilon^2\right)^2} + \frac{2\sigma_X^8\lVert\beta_0\rVert^4}{(\sigma_X^2\lVert\beta_0\rVert^2+\sigma_\epsilon^2)^4}  \right) \beta_0\beta_0^\top.
\end{align*}
The calculations above imply that
\begin{align*}
    M &= \Sigma_2^{-\frac{1}{2}} \Gamma_1 \Sigma_2^{-\frac{1}{2}} \\
    & = \left(a\mu\mu^\top + b \mathbbm{1}_{p\times p}\right)\left(c\mu\mu^\top + d \beta_0\beta_0^\top\right)\left(a\mu\mu^\top + b \mathbbm{1}_{p\times p}\right) \\
    & = \left(c(a\lVert\mu\rVert^2+b)^2+a^2d(\mu^\top\beta_0)^2\right)\mu\mu^\top + \left(abd(\mu^\top\beta_0)\right)\left(\beta_0\mu^\top + \mu\beta_0^\top\right) + b^2d\beta_0\beta_0^\top\\
\end{align*}
with 
\begin{align*}
    a = \frac{\sigma_\epsilon}{\lVert\mu\rVert^2}\left(\frac{1}{\sqrt{\lVert\mu\rVert^2+\sigma_X^2}}-\frac{1}{\sigma_X}\right), \ \text{and} \ 
    b = \frac{\sigma_\epsilon}{\sigma_X}
\end{align*}
\begin{align*}
    c = \frac{1}{\sigma_X^2\lVert\beta_0\rVert^2+\sigma_\epsilon^2}, \ \text{and} \ 
    d = \frac{2\sigma_X^4}{\left(\sigma_X^2\lVert\beta_0\rVert^2+\sigma_\epsilon^2\right)^2}.
\end{align*}
similarly,
\begin{align*}
    N &= \Sigma_2^{-\frac{1}{2}} (\Gamma_1 + \Gamma_2)\Sigma_2^{-\frac{1}{2}} \\
    & = \left(a\mu\mu^\top + b \mathbbm{1}_{p\times p}\right)\left(e\mu\mu^\top + f \beta_0\beta_0^\top\right)\left(a\mu\mu^\top + b \mathbbm{1}_{p\times p}\right) \\
    & = \left(e(a\lVert\mu\rVert^2+b)^2+a^2f(\mu^\top\beta_0)^2\right)\mu\mu^\top + \left(abf(\mu^\top\beta_0)\right)\left(\beta_0\mu^\top + \mu\beta_0^\top\right) + b^2f\beta_0\beta_0^\top\\
\end{align*}
with
\begin{align*}
    e = \frac{1}{\sigma_X^2\lVert\beta_0\rVert^2+\sigma_\epsilon^2}+\frac{\sigma_X^2\lVert\beta_0\rVert^2}{(\sigma_X^2\lVert\beta_0\rVert^2+\sigma_\epsilon^2)^2}, \ \text{and} \
    f = \frac{2\sigma_X^4}{\left(\sigma_X^2\lVert\beta_0\rVert^2+\sigma_\epsilon^2\right)^2} + \frac{2\sigma_X^8\lVert\beta_0\rVert^4}{(\sigma_X^2\lVert\beta_0\rVert^2+\sigma_\epsilon^2)^4}. 
\end{align*}
We need to calculate $\prod_{i=1}^p(1+\frac{h_i}{\lambda})$ where we recall that $h_i, i=1, \ldots, p$ are the eigenvalues of the matrix $N$ which is of the form 
$$N = A\mu\mu^\top + B(\beta_0\mu^\top + \mu\beta_0^\top) + C\beta_0\beta_0^\top.$$ 
Now, note that if $v$ is the eigenvector whose corresponding eigenvalue $h$ is not equal to 0, then
\begin{align*}
    Nv = \left(A\mu^\top v + B\beta_0^\top v\right) \mu + \left(B\mu^\top v + C\beta_0^\top v\right) \beta_0 = hv
\end{align*}
which means that $v \in \operatorname{span}\{\mu, \beta_0\}$. We can write $v = \alpha \mu + \gamma \beta_0$, and hence
\begin{align*}
   & \left( A\mu\mu^\top + B(\beta_0\mu^\top + \mu\beta_0^\top) + C\beta_0\beta_0^\top\right) \left(\alpha \mu + \gamma \beta_0\right) \\
   & \ = \alpha \left(A \lVert\mu\rVert^2 + B(\mu^\top \beta_0)\right) \mu + \gamma \left(A(\mu^\top\beta_0)  + B\lVert\beta_0\rVert^2 \right)\mu \\
   & \ \ \ + \alpha \left(B\lVert\mu\rVert^2+C(\mu^\top\beta_0)\right) \beta_0 + \gamma \left(B(\mu^\top\beta_0)+C\lVert\beta_0\rVert^2\right)\beta_0 \\
   & \ = h \alpha \mu + h \gamma \beta_0.
\end{align*}
We can re-write the above equation as
\begin{align*}
    \begin{pmatrix}
        A \lVert\mu\rVert^2 + B(\mu^\top \beta_0) & A(\mu^\top\beta_0)  + B\lVert\beta_0\rVert^2\\
        B\lVert\mu\rVert^2+C(\mu^\top\beta_0) & B(\mu^\top\beta_0)+C\lVert\beta_0\rVert^2 
    \end{pmatrix}\begin{pmatrix}
       \alpha \\
       \gamma
    \end{pmatrix} = h \begin{pmatrix}
       \alpha \\
       \gamma
    \end{pmatrix}
    .
\end{align*}
Thus, the nonzero eigenvalues $h_i$'s are also eigenvalues of the above matrix, and
\begin{align*}
    h^2 - \left(A\lVert\mu\rVert^2 + C\lVert\beta_0\rVert^2 + 2B(\mu^\top\beta_0)\right)h + (AC -B^2)\left(\lVert\mu\rVert^2\lVert\beta_0\rVert^2 - (\mu^\top\beta_0)^2\right) = 0.
\end{align*}
Using the Vieta's formulas applied to quadratic polynomials we obtain
\begin{align*}
&\prod_{i=1}^p(1+\frac{h_i}{\lambda})\\
&= 1 + \frac{h_1+h_2}{\lambda} + \frac{h_1h_2}{\lambda^2}  \\
&=1+\frac{1}{\lambda}\left(A\lVert\mu\rVert^2 + C\lVert\beta_0\rVert^2 + 2B(\mu^\top\beta_0)\right) + \frac{1}{\lambda^2}(AC -B^2)\left(\lVert\mu\rVert^2\lVert\beta_0\rVert^2 - (\mu^\top\beta_0)^2\right)
\end{align*}
 When the above results are applied to $M = \Sigma_2^{-\frac{1}{2}} \Gamma_1 \Sigma_2^{-\frac{1}{2}}$, it follows that
 \begin{align*}
     &\prod_{i=1}^p (1+\frac{k_i}{\lambda}) \\
     & = 1 + \frac{1}{\lambda} \bigg(\big(c(a\lVert\mu\rVert^2+b)^2+a^2d(\mu^\top\beta_0)^2\big) \lVert\mu\rVert^2 + b^2d \lVert\beta_0\rVert^2 + 2\left(abd(\mu^\top\beta_0)\right) (\mu^\top\beta_0 )\bigg) \\
     & \ \ \ +\frac{1}{\lambda^2} \bigg(\big(c(a\lVert\mu\rVert^2+b)^2+a^2d(\mu^\top\beta_0)^2\big) b^2d - \big(abd(\mu^\top\beta_0)\big)^2\bigg)\left(\lVert\mu\rVert^2\lVert\beta_0\rVert^2 - (\mu^\top\beta_0)^2\right) \\
      & = 1 + \frac{1}{\lambda} \bigg(\big(c(a\lVert\mu\rVert^2+b)^2+a^2d(\mu^\top\beta_0)^2\big) \lVert\mu\rVert^2 + b^2d \lVert\beta_0\rVert^2 + 2\left(abd(\mu^\top\beta_0)\right) (\mu^\top\beta_0 )\bigg) \\
     & \ \ \ +\frac{1}{\lambda^2} \big(b^2dc(a\lVert\mu\rVert^2+b)^2 \big)\left(\lVert\mu\rVert^2\lVert\beta_0\rVert^2 - (\mu^\top\beta_0)^2\right).
 \end{align*}
Now, note that
\begin{align*}
&\big(c(a\lVert\mu\rVert^2+b)^2+a^2d(\mu^\top\beta_0)^2\big) \lVert\mu\rVert^2 \\
& = \bigg(\frac{1}{\frac{\sigma_X^2\lVert\beta_0\rVert^2}{\sigma_\epsilon^2}+1}\frac{1}{\lVert\mu\rVert^2 + \sigma_X^2} + 2\frac{(\mu^\top\beta_0)^2}{\lVert\mu\rVert^4}\left(\frac{1}{\sqrt{\lVert\mu\rVert^2+\sigma_X^2}}-\frac{1}{\sigma_X}\right)^2\frac{\frac{\sigma_X^4}{\sigma_\epsilon^2}}{(\frac{\sigma_X^2\lVert\beta_0\rVert^2}{\sigma_\epsilon^2}+1)^2} \bigg)\lVert\mu\rVert^2
\end{align*}
\begin{align*}
    b^2d \lVert\beta_0\rVert^2 = 2 \frac{\frac{\sigma_X^2\lVert\beta_0\rVert^2}{\sigma_\epsilon^2}}{(\frac{\sigma_X^2\lVert\beta_0\rVert^2}{\sigma_\epsilon^2}+1)^2}
\end{align*}
\begin{align*}
    2\left(abd(\mu^\top\beta_0)\right) (\mu^\top\beta_0 ) = 4 \frac{(\mu^\top\beta_0)^2}{\lVert\mu\rVert^2}\left(\frac{1}{\sqrt{\lVert\mu\rVert^2+\sigma_X^2}} - \frac{1}{\sigma_X}\right)\frac{\frac{\sigma_X^3}{\sigma_\epsilon^2}}{(\frac{\sigma_X^2\lVert\beta_0\rVert^2}{\sigma_\epsilon^2}+1)^2}
\end{align*}
\begin{align*}
    b^2dc(a\lVert\mu\rVert^2+b)^2 &=   2 \frac{\frac{\sigma_X^2}{\sigma_\epsilon^2}}{(\frac{\sigma_X^2\lVert\beta_0\rVert^2}{\sigma_\epsilon^2}+1)^2}\frac{1}{\frac{\sigma_X^2\lVert\beta_0\rVert^2}{\sigma_\epsilon^2}+1}\frac{1}{\lVert\mu\rVert^2 + \sigma_X^2} \\
    &=2 \frac{\frac{\sigma_X^2}{\sigma_\epsilon^2}}{(\frac{\sigma_X^2\lVert\beta_0\rVert^2}{\sigma_\epsilon^2}+1)^3}\frac{1}{\lVert\mu\rVert^2 + \sigma_X^2}.
\end{align*}
Therefore,
\begin{align*}
    \prod_{i=1}^p (1+\frac{k_i}{\lambda}) &=  1 + \frac{1}{\lambda} \bigg( \frac{1}{\eta+1} \frac{\lVert\mu\rVert^2}{\lVert\mu\rVert^2 + \sigma_X^2} + 2 \frac{\eta}{(\eta+1)^2}  \\
    & \ \ \ + \frac{2}{(\eta+1)^2} \frac{(\mu^\top\beta_0)^2}{\lVert\mu\rVert^2}\frac{\sigma_X^3}{\sigma_\epsilon^2} (\frac{1}{\sqrt{\lVert\mu\rVert^2+\sigma_X^2}} -\frac{1}{\sigma_X}) (\sigma_X(\frac{1}{\sqrt{\lVert\mu\rVert^2 + \sigma_X^2}}-\frac{1}{\sigma_X}) +2)  \bigg) \\
    & \ \ \ + \frac{2}{\lambda^2} \bigg( \lVert\mu\rVert^2\lVert\beta_0\rVert^2 -(\mu^\top\beta_0)^2\bigg)\frac{\frac{\sigma_X^2}{\sigma_\epsilon^2}}{(\eta+1)^3} \frac{1}{\lVert\mu\rVert^2+\sigma_X^2}  \\
    &= 1 + \frac{1}{\lambda} \bigg( \frac{1}{\eta+1} \frac{\lVert\mu\rVert^2}{\lVert\mu\rVert^2 + \sigma_X^2} + 2 \frac{\eta}{(\eta+1)^2} - 2 \frac{(\mu^\top\beta_0)^2}{(\eta+1)^2}\frac{\sigma_X^2}{\sigma_\epsilon^2}\frac{1}{\lVert\mu\rVert^2+ \sigma_X^2}  \bigg) \\
    & \ \ \ + \frac{2}{\lambda^2} \bigg( \lVert\mu\rVert^2\lVert\beta_0\rVert^2 -(\mu^\top\beta_0)^2\bigg)\frac{\frac{\sigma_X^2}{\sigma_\epsilon^2}}{(\eta+1)^3} \frac{1}{\lVert\mu\rVert^2+\sigma_X^2}  
\end{align*}
with $\eta = \frac{\sigma_X^2\lVert\beta_0\rVert^2}{\sigma_\epsilon^2}$. Similarly, we can also get results for $N = \Sigma_2^{-\frac{1}{2}} (\Gamma_1+\Gamma_2) \Sigma_2^{-\frac{1}{2}}$. \\

 From the calculations above, we can conclude that if $\epsilon \sim \mathcal{N}(0,\sigma_\epsilon^2)$ and $X \sim \mathcal{N}(\mu, \sigma_X^2\mathbbm{1}_{p\times p})$ and if $\eta := \frac{\sigma_X^2\lVert\beta_0\rVert^2}{\sigma_\epsilon^2}$, then the volume ratio is given by
\begin{align*}
    \frac{\text{Vol}_{\text{SSL}}}{\text{Vol}_{\text{OLS}}}  &= \frac{\sqrt{R_N}}{R_D}
\end{align*}
with 
\begin{align*}
    R_N & = 1 + \frac{1}{\lambda} \bigg( \frac{1}{\eta+1} \frac{\lVert\mu\rVert^2}{\lVert\mu\rVert^2 + \sigma_X^2}(1+\frac{\eta}{\eta+1}) + 2 \frac{\eta}{(\eta+1)^2} (1+\frac{\eta^2}{(\eta+1)^2}) \\
    & \ \ \ - 2 \frac{(\mu^\top\beta_0)^2}{(\eta+1)^2}\frac{\sigma_X^2}{\sigma_\epsilon^2}\frac{1}{\lVert\mu\rVert^2+ \sigma_X^2} (1+\frac{\eta^2}{(\eta+1)^2})  \bigg) \\
    & \ \ \ + \frac{2}{\lambda^2} \bigg( \lVert\mu\rVert^2\lVert\beta_0\rVert^2 -(\mu^\top\beta_0)^2\bigg)\frac{\frac{\sigma_X^2}{\sigma_\epsilon^2}}{(\eta+1)^3} \frac{1}{\lVert\mu\rVert^2+\sigma_X^2} (1+\frac{\eta^2}{(\eta+1)^2})(1+\frac{\eta}{\eta+1})
\end{align*}
and
\begin{align*}
    R_D &= 1 + \frac{1}{\lambda} \bigg( \frac{1}{\eta+1} \frac{\lVert\mu\rVert^2}{\lVert\mu\rVert^2 + \sigma_X^2} + 2 \frac{\eta}{(\eta+1)^2} - 2 \frac{(\mu^\top\beta_0)^2}{(\eta+1)^2}\frac{\sigma_X^2}{\sigma_\epsilon^2}\frac{1}{\lVert\mu\rVert^2+ \sigma_X^2}  \bigg) \\
    & \ \ \ + \frac{2}{\lambda^2} \bigg( \lVert\mu\rVert^2\lVert\beta_0\rVert^2 -(\mu^\top\beta_0)^2\bigg)\frac{\frac{\sigma_X^2}{\sigma_\epsilon^2}}{(\eta+1)^3} \frac{1}{\lVert\mu\rVert^2+\sigma_X^2}.  
\end{align*}
If $X \sim \mathcal{N}(\mu, \Sigma_X)$ with $\Sigma_X$ positive definite, then we can also re-write the linear model as
\begin{align*}
    Y = \beta_0^\top X +\epsilon = \alpha_0^\top Z + \epsilon
\end{align*}
where $\alpha_0 = \Sigma_X^{\frac{1}{2}}\beta_0$ and $Z \sim \mathcal{N}(\theta, \mathbbm{1}_{p\times p})$ with $\theta = \Sigma_X^{-\frac{1}{2}}\mu$. By the previous results, we have for $\lVert\beta_0\rVert \neq 0$ and $\lVert\mu\rVert \neq 0$,
 \begin{align*}
     & \frac{\text{Vol}_{\text{SSL}}}{\text{Vol}_{\text{OLS}}} \nonumber \\
     &=  \frac{\sqrt{1+\frac{1}{\lambda}\left(\frac{1}{1+\eta}\frac{1}{1+\rho}(1+\frac{\eta}{\eta+1})+\frac{2\eta}{(\eta+1)^2}(1-\frac{\zeta^2}{1+\rho})(1+\frac{\eta^2}{(\eta+1)^2})\right) + \frac{1}{\lambda^2}\frac{1-\zeta^2}{1+\rho}\frac{2\eta}{(\eta+1)^3}(1+\frac{\eta^2}{(\eta+1)^2})(1+\frac{\eta}{\eta+1}) }}{1+\frac{1}{\lambda}\left(\frac{1}{1+\eta}\frac{1}{1+\rho}+\frac{2\eta}{(\eta+1)^2}(1-\frac{\zeta^2}{1+\rho})\right) + \frac{1}{\lambda^2}\frac{1-\zeta^2}{1+\rho}\frac{2\eta}{(\eta+1)^3}}
\end{align*}
with $\zeta:= \frac{\theta^\top\alpha_0}{\lVert\theta\rVert\lVert\alpha_0\rVert} =\frac{\mu^\top\beta_0}{\sqrt{\beta_0^\top \Sigma_X \beta_0}\sqrt{\mu^\top \Sigma_X^{-1}\mu}}$, $\rho :=\frac{1}{\lVert\theta\rVert^2} =\frac{1}{\mu^\top \Sigma_X^{-1}\mu}$ and $\eta := \frac{\lVert\alpha_0\rVert^2}{\sigma_\epsilon^2} =\frac{\beta_0^\top \Sigma_X\beta_0}{\sigma_\epsilon^2}$.

In the special case where $\mu = 0 \in \mathbb{R}^p$, it can be easily established that the relative efficiency is given by
\begin{align*}
     \frac{\text{Vol}_{\text{SSL}}}{\text{Vol}_{\text{OLS}}}= \frac{\sqrt{1+\frac{2}{\lambda}\left(\frac{\eta}{(\eta+1)^2} + \frac{\eta^3}{(\eta+1)^4}\right)}}{1+\frac{2}{\lambda}\frac{\eta}{(\eta+1)^2}} \quad \text{with $\eta =\frac{\beta_0^\top \Sigma_X\beta_0}{\sigma_\epsilon^2}$.} 
\end{align*}

\end{proof}

\section*{Appendix B: Auxiliary results and their proofs}

\begin{theorem}\label{Remainder}
Under the assumptions (A0)-(A4), it holds that
\begin{eqnarray*}
 \frac{1}{n}\sum_{j=1}^n \frac{\int x f'(\tilde Y_j - \beta^\top_0 x) d\mathbb P^{\tilde {X}}_n(x)}{\int f(\tilde Y_j - \beta^\top_0 x) d\mathbb P^{\tilde {X}}_n(x)} & = &  \frac{1}{n}\sum_{j=1}^n \frac{\int x f'(\tilde Y_j - \beta^\top_0 x) d\mathbb P^{\tilde {X}}(x)}{\int f(\tilde Y_j - \beta^\top_0 x) d\mathbb P^{\tilde {X}}(x)}  + \int \psi(x) d(\mathbb P^{\tilde X}_n - \mathbb P^{\tilde X}) (x)   \\
&&   \  + \  o_{\mathbb P^{\tilde X} \otimes\mathbb P^{\tilde Y} }\left(\frac{1}{\sqrt n} \right),
\end{eqnarray*}
where
$$
\psi(x) =  -   \int \frac{\int z f'(y- \beta^\top_0 z) d \mathbb P^{\tilde X}(z)}{\int  f(y- \beta^\top_0 z) d \mathbb P^{\tilde X}(z)} f(y- \beta^\top_0 x)  dy. 
$$
Furthermore, as $n \to \infty$ we have that
\begin{eqnarray*}
\sqrt n  \frac{1}{n}\sum_{j=1}^n \frac{\int x f'(\tilde Y_j - \beta^\top_0 x) d\mathbb P^{\tilde {X}}(x)}{\int f(\tilde Y_j - \beta^\top_0 x) d\mathbb P^{\tilde {X}}(x)}  =  \sqrt n \int  \frac{\int x f'(y - \beta^\top_0 x) d\mathbb P^{\tilde{X}}(x)}{\int f(y - \beta^\top_0 x)  d\mathbb P^{\tilde {X}}(x)} d(\mathbb P^{\tilde Y}_n - \mathbb P^{\tilde Y})(y)  \to_d   W_1    
\end{eqnarray*}
and 
\begin{eqnarray*}
\sqrt n \int \psi(x)  d(\mathbb P^{\tilde X}_n - \mathbb P^{\tilde X}) (x)  \to_d   W_2    
\end{eqnarray*}
where $W_1$ and $W_2$ are independent centered Gaussian $p$-dimensional vectors with covariance matrices 
\begin{eqnarray}\label{Gamma1}
\Gamma_1 =  \int  \frac{\left(\int x f'(y- \beta^\top_0 x) f^X(x) dx \right)\left(\int x^\top f'(y- \beta^\top_0 x) f^X(x) dx \right) }{f^Y(y) } dy
\end{eqnarray}
and 
\begin{eqnarray}\label{Gamma2}
\hspace{-0.8cm}\Gamma_2 &= & \int \left(\int \frac{ f(y- \beta^\top_0 x) \int z f'(y- \beta^\top_0 z) f^X(z)dz  }{f^Y(y)} dy\right) \left(\int \frac{ f(y- \beta^\top_0 x)  \int z^\top f'(y- \beta^\top_0 z) f^X(z)dz}{f^Y(y)} dy\right) f^X(x)dx \notag \\
&& 
\end{eqnarray}
respectively. 

\end{theorem}

\bigskip
\bigskip

\begin{proof}
We have that 
\begin{eqnarray*}
\sqrt n \frac{1}{n}\sum_{j=1}^n \frac{\int x f'(\tilde Y_j - \beta^\top_0 x) d\mathbb P^{\tilde {X}}(x)}{\int f(\tilde Y_j - \beta^\top_0 x) d\mathbb P^{\tilde {X}}(x)}  =   \int   \frac{\int x f'(y - \beta^\top_0 x) d\mathbb P^{\tilde {X}}(x)}{\int f(y - \beta^\top_0 x) d\mathbb P^{\tilde {X}}(x)} d \mathbb G^{\tilde{Y}}_n(y)
\end{eqnarray*}
since
\begin{eqnarray*}
   \mathbb E\left[\frac{\int x f'(Y - \beta^\top_0 x) d\mathbb P^{\tilde {X}}(x)}{\int f(Y - \beta^\top_0 x) d\mathbb P^{\tilde {X}}(x)} \right]   & = &  \int \frac{\int x f'(y - \beta^\top_0 x) d\mathbb P^{\tilde {X}}(x)}{\int f(y - \beta^\top_0 x) d\mathbb P^{\tilde {X}}(x)}  d\mathbb P^{\tilde Y}(y) \\
   & = &  \int \frac{\int x f'(y - \beta^\top_0 x) d\mathbb P^{\tilde {X}}(x)}{\int f(y - \beta^\top_0 x) d\mathbb P^{\tilde {X}}(x)}  \int f(y - \beta^\top_0 x) d\mathbb P^{\tilde {X}}(x) dy \\
   & = &  \int x f'(y- \beta^\top_0 x) d\mathbb P^{\tilde X}(x) dy  \\
   & = & \mathbb E(X) \int f'(t) dt = 0
 \end{eqnarray*}
using the change of variable $t = y-\beta^\top_0 x$. Then,
 \begin{eqnarray*}
\sqrt n \frac{1}{n}\sum_{j=1}^n \frac{\int x f'(\tilde Y_j - \beta^\top_0 x) d\mathbb P^{\tilde {X}}(x)}{\int f(\tilde Y_j - \beta^\top_0 x) d\mathbb P^{\tilde {X}}(x)}  \to_d \mathcal{N}(0, \Gamma_1)   
\end{eqnarray*}
where $\Gamma_1$ is given in (\ref{Gamma1}).

Also, we can write that
\begin{eqnarray}\label{Decomp}
&& \frac{1}{n}\sum_{j=1}^n \frac{\int x f'(\tilde Y_j - \beta^\top_0 x) d\mathbb P^{\tilde {X}}_n(x)}{\int f(\tilde Y_j - \beta^\top_0 x) d\mathbb P^{\tilde {X}}_n(x)}  -  \frac{1}{n}\sum_{j=1}^n \frac{\int x f'(\tilde Y_j - \beta^\top_0 x) d\mathbb P^{\tilde {X}}(x)}{\int f(\tilde Y_j - \beta^\top_0 x) d\mathbb P^{\tilde {X}}(x)} \notag \\
&& =  \int \left(\frac{\int x f'(y - \beta_0^\top x) d\mathbb P^{\tilde X}_n(x)}{\int f(y - \beta_0^\top x) d\mathbb P^{\tilde X}_n(x)}  - \frac{\int x f'(y - \beta_0^\top x) d\mathbb P^{\tilde X}(x)}{\int f(y - \beta_0^\top x) d\mathbb P^{\tilde X}(x)} \right) d\mathbb P^{\tilde Y}_n(y) \notag \notag \\
&&  =  \int  \left( \frac{\int x f'(y - \beta_0^\top x) d\mathbb P^{\tilde X}_n(x)}{\int f(y - \beta_0^\top x) d\mathbb P^{\tilde X}_n(x)}  - \frac{\int x f'(y - \beta_0^\top x) d\mathbb P^{\tilde X}(x)}{\int f(y - \beta_0^\top x) d\mathbb P^{\tilde X}(x)} \right) d\mathbb P^{\tilde Y}(y)  \notag\\
&&  +  \  \frac{1}{\sqrt n} \int \left(\frac{\int x f'(y - \beta_0^\top x) d\mathbb P^{\tilde X}_n(x)}{\int  f(y - \beta_0^\top x) d\mathbb P^{\tilde X}_n(x)}  - \frac{\int x f'(y - \beta_0^\top x) d\mathbb P^{\tilde X}(x)}{\int f(y - \beta_0^\top x) d\mathbb P^{\tilde X}(x)} \right)d\mathbb G^{\tilde Y}_n(y), \notag \\
&&
\end{eqnarray}
with $\mathbb G^{\tilde Y}_n  = \sqrt n (\mathbb P^{\tilde Y}_n  - \mathbb P^{\tilde Y})$. Now, put
\begin{eqnarray*}
\Delta_n(y) : =  \frac{\int x f'(y - \beta_0^\top x) d\mathbb P^{\tilde X}_n(x)}{\int f(y - \beta_0^\top x) d\mathbb P^{\tilde X}_n(x)}  -  \frac{\int x f'(y - \beta_0^\top x) d\mathbb P^{\tilde X}(x)}{\int f(y - \beta_0^\top x) d\mathbb P^{\tilde X}(x)}. 
\end{eqnarray*}
Then,
\begin{eqnarray*} 
\Delta_n(y) & = &  \frac{\int x f'(y - \beta_0^\top x) d(\mathbb P^{\tilde X}_n - \mathbb P^{\tilde X})(x)}{\int f(y - \beta_0^\top x) d\mathbb P^{\tilde X}_n(x)}  \\
&& \ +  \ \frac{\int x f'(y - \beta_0^\top x) d\mathbb P^{\tilde X}(x)}{\int f(y- \beta^\top_0 x) d \mathbb P^{\tilde X}_n(x)} -  \frac{\int x f'(y - \beta_0^\top x) d\mathbb P^{\tilde X}(x)}{\int f(y- \beta^\top_0 x) d \mathbb P^{\tilde X}(x)}  \\
& = &  \frac{\int x f'(y - \beta_0^\top x) d(\mathbb P^{\tilde X}_n - \mathbb P^{\tilde X})(x)}{\int f(y - \beta_0^\top x) d\mathbb P^{\tilde X}_n(x)}  \\
&&  - \frac{\int x f'(y - \beta_0^\top x) d\mathbb P^{\tilde X}(x) \cdot \int f(y- \beta^\top_0 x) d(\mathbb P^{\tilde X}_n - \mathbb P^{\tilde X})(x)}{\int f(y - \beta_0^\top x) d\mathbb P^{\tilde X}_n(x) \cdot\int f(y - \beta_0^\top x) d\mathbb P^{\tilde X}(x)}  \\
& := &   \frac{1}{\sqrt n} \int m_n(x, y)  d \mathbb G^{\tilde X}_n(x)    
\end{eqnarray*}
with $\mathbb G^{\tilde X}_n =  \sqrt n (\mathbb P^{\tilde X}_n - \mathbb P^{\tilde X})$ and 
\begin{eqnarray*}
m_n(x, y) & = &  \frac{x f'(y- \beta^\top_0 x)}{\int f(y- \beta^\top_0 z) d\mathbb P^{\tilde X}_n(z)}  -  \frac{\int z f'(y- \beta^\top_0 z) d\mathbb P^{\tilde X}(z) f(y- \beta^\top_0 x)}{ \int f(y- \beta^\top_0 z) d\mathbb P^{\tilde X}_n(z)   \int f(y- \beta^\top_0 z) d\mathbb P^{\tilde X}(z) } \\
& = &  \frac{1}{\int f(y- \beta^\top_0 z) d\mathbb P^{\tilde X}_n(z)}  \left( x f'(y- \beta^\top_0 x) -  \frac{\int z f'(y- \beta^\top_0 z) d \mathbb P^{\tilde X}(z)}{\int  f(y- \beta^\top_0 z) d \mathbb P^{\tilde X}(z)} f(y- \beta^\top_0 x)  \right). 
\end{eqnarray*}
Thus, using the fact that $d \mathbb P^{\tilde Y}(y) = f^Y(y) dy$ and the decomposition in (\ref{Decomp}), we get that
\begin{eqnarray*}\label{ImpDecom}
&& \sqrt n \left \{\frac{1}{n}\sum_{j=1}^n \frac{\int x f'(\tilde Y_j - \beta^\top_0 x) d\mathbb P^{\tilde {X}}_n(x)}{\int f(\tilde Y_j - \beta^\top_0 x) d\mathbb P^{\tilde {X}}_n(x)}  -  \frac{1}{n}\sum_{j=1}^n \frac{\int x f'(\tilde Y_j - \beta^\top_0 x) d\mathbb P^{\tilde {X}}(x)}{\int f(\tilde Y_j - \beta^\top_0 x) d\mathbb P^{\tilde {X}}(x)} \right \} \notag \\
&& =  \int  \left(\int m_n(x, y) f^Y(y) dy \right)  d \mathbb G^{\tilde X}_n(x)  +  \frac{1}{\sqrt n} \int  \int m_n(x, y) d \mathbb G^{\tilde X}_n(x)  d \mathbb G^{\tilde Y}_n(y).
\end{eqnarray*}
We will show next that  $\int  \left(\int m_n(x, y) f^Y(y) dy \right)  d \mathbb G^{\tilde X}_n(x)$ converges weakly to a Gaussian distribution and that 
$ \int  \int m_n(x, y) d \mathbb G^{\tilde X}_n(x)  d \mathbb G^{\tilde Y}_n(y) = O_{\mathbb P^{\tilde X} \otimes \mathbb P^{\tilde Y}}(1)$.  Note that $m_n(x, y) = \tilde m(x, y) + D_n(x, y)$ where 
\begin{eqnarray*}
\tilde m(x, y)  & =   &      \frac{1}{\int f(y- \beta^\top_0 z) d\mathbb P^{\tilde X}(z)}  \left( x f'(y- \beta^\top_0 x) -  \frac{\int z f'(y- \beta^\top_0 z) d \mathbb P^{\tilde X}(z)}{\int  f(y- \beta^\top_0 z) d \mathbb P^{\tilde X}(z)} f(y- \beta^\top_0 x)  \right)
\end{eqnarray*}
and
\begin{eqnarray*}
  D_n(x, y) &=&  m_n(x, y) -  \tilde m(x, y) \\
 & = &  - \frac{1}{\sqrt n}  \frac{ x f'(y- \beta^\top_0 x) -  \frac{\int z f'(y- \beta^\top_0 z) d \mathbb P^{\tilde X}(z)}{\int  f(y- \beta^\top_0 z) d \mathbb P^{\tilde X}(z)} f(y- \beta^\top_0 x)}{\int f(y- \beta^\top_0 z) d\mathbb P^{\tilde X}_n(z) \int f(y- \beta^\top_0 z) d\mathbb P^{\tilde X}(z)}  \ \int f(y- \beta^\top_0z) d\mathbb G^{\tilde X}_n(z) .
\end{eqnarray*}
Using again the fact that $\int f(y- \beta^\top_0 x) d\mathbb P^{\tilde X}(x) = f^Y(y)$, we compute 
\begin{eqnarray*}
\int \tilde{m}(x, y) f^Y(y) dy  & = &   \int   \left( x f'(y- \beta^\top_0 x) -  \frac{\int z f'(y- \beta^\top_0 z) d \mathbb P^{\tilde X}(z)}{\int  f(y- \beta^\top_0 z) d \mathbb P^{\tilde X}(z)} f(y- \beta^\top_0 x)  \right) dy  \\
& =  &  -   \int \frac{\int z f'(y- \beta^\top_0 z) d \mathbb P^{\tilde X}(z)}{\int  f(y- \beta^\top_0 z) d \mathbb P^{\tilde X}(z)} f(y- \beta^\top_0 x)  dy
\end{eqnarray*}
since $\int  f'(y- \beta^\top_0 x) dy =  \int f'(t) dt = 0$. Hence,
\begin{eqnarray*}
\mathbb P^{\tilde X} \left( \int \tilde{m}(\cdot, y) f^Y(y) dy \right) &  =   &   - \int  \int \frac{\int z f'(y- \beta^\top_0 z) d \mathbb P^{\tilde X}(z)}{\int  f(y- \beta^\top_0 z) d \mathbb P^{\tilde X}(z)} f(y- \beta^\top_0 x)  dy d\mathbb P^{\tilde X}(x)  \\
& = & - \int  \int z f'(y- \beta^\top_0 z) d \mathbb P^{\tilde X}(z) dy \\
& = & 0.
\end{eqnarray*}
Then,
\begin{eqnarray*}
\int \left(\int \tilde{m}(x, y)f^Y(y) dy\right) d\mathbb G^{\tilde X}_n(x)  \to_d \mathcal{N}(0, \Gamma_2)
\end{eqnarray*}
with $\Gamma_2$ the same matrix in (\ref{Gamma2}).   Now, it remains to show that
\begin{itemize}
\item  $\sqrt n \int \left(\int D_n(x, y) f^Y(y) dy\right)  d\mathbb G^{\tilde X}_n(x)  = O_{\mathbb P^{\tilde{X}}}(1)$, 

\item $\int \left(\int m_n(x, y) d\mathbb G^{\tilde X}_n(x) \right) d\mathbb G^{\tilde Y}_n(y) = O_{\mathbb P^{\tilde{X}} \otimes \mathbb P^{\tilde Y}}(1)$.

\end{itemize}
We have that
\begin{eqnarray*}
&&  -\sqrt n \int D_n(x, y) f^Y(y) dy   \\
&& = \int \frac{\int f(y- \beta_0^\top z) d\mathbb G^{\tilde X}_n(z) }{\int f(y- \beta^\top_0 z) d\mathbb P^{\tilde X}_n(z)}  \left(x f'(y- \beta_0^\top x) -  \frac{\int z f'(y- \beta^\top_0 z) f^X(z) dz }{\int  f(y- \beta^\top_0 z) f^X(z) dz}  f(y-\beta_0^\top x)       \right) dy \\
&& : = d_n(x).
\end{eqnarray*}
Note that $d_n$ is almost every where infinitely differentiable in each component $x_i, i =1, \ldots, x_p$. Also, using Leibniz formula, we get for any integer $k \ge 0$ and $i, j \in \{1,  \ldots, p \} $
\begin{eqnarray*}
\frac{\partial^k [x_i f'(y - \beta^\top_0 x)]}{\partial x_j^k}  & = &  \sum_{r=0}^k \binom{k}{r} \frac{\partial^r x_i}{\partial x_j^r} (-1)^{k-r} \beta_{0j}^{k-r} f^{(k-r +1)}(y-\beta^\top_0 x) \\
& = & x_i  (-1)^{k} \beta_{0j}^{k} f^{(k+1)}(y-\beta^\top_0 x)  +  k \mathds{1}_{i=j} (-1)^{k-1} \beta_{0j}^{k-1} f^{(k)}(y-\beta^\top_0 x).
\end{eqnarray*}
Hence,  for any integers $k_1, k_2 \ge 0$
\begin{eqnarray*}
&& \frac{\partial^{k_1 + k_2} [x_i f'(y - \beta^\top_0 x)]}{\partial x_1^{k_1}\partial x_2^{k_2} } \\
&&  =  \frac{\partial^{k_2}}{\partial x_2^{k_2} } \left[x_i  (-1)^{k_1} \beta_{01}^{k_1} f^{(k_1+1)}(y-\beta^\top_0 x)  +  k_1 \mathds{1}_{i=1} (-1)^{k_1-1} \beta_{01}^{k_1-1} f^{(k_1)}(y-\beta^\top_0 x) \right]  \\
&& = x_i  (-1)^{k_1 +k_2} \beta_{01}^{k_1}\beta_{02}^{k_2}f^{(k_1+ k_2 + 1)}(y-\beta^\top_0 x) +  k_2  \mathds{1}_{i=2} (-1)^{k_1 + k_2 -1}  \beta_{01}^{k_1}  \beta_{02}^{k_2-1}  f^{(k_1+ k_2)}(y-\beta^\top_0 x)  \\ 
&& \  + \  k_1 \mathds{1}_{i=1} (-1)^{k_1 + k_2-1} \beta_{01}^{k_1-1} \beta_{02}^{k_2} f^{(k_1 + k_2)}(y-\beta^\top_0 x) \\
&&  = x_i  (-1)^{k_1 +k_2} \beta_{01}^{k_1}\beta_{02}^{k_2}f^{(k_1+ k_2 + 1)}(y-\beta^\top_0 x)  \\
&&   \ +  \   ( k_1 \mathds{1}_{i=1}  \beta_{02} +  k_2 \mathds{1}_{i=2}\beta_{01})  (-1)^{k_1 + k_2 -1} \beta_{01}^{k_1-1}  \beta_{02}^{k_2-1} f^{(k_1+ k_2)}(y-\beta^\top_0 x)
\end{eqnarray*}
which can be easily generalized for any vector $k= (k_1, \ldots, k_p) \in \mathbb N^p_0$:
\begin{eqnarray*}
&& D^k( x_i f'(y - \beta^\top_0 x)): =  \frac{\partial^{k_1 + \ldots + k_p} [x_i f'(y - \beta^\top_0 x)]}{\partial x_1^{k_1} \ldots  \partial x_p^{k_p}}  \\
&& = x_i  (-1)^{k_.} (\prod_{j=1}^p \beta_{0j}^{k_j}) f^{(k_. + 1)}(y-\beta^\top_0 x) \\
&& \ + \  \Big( \sum_{j=1}^p k_j \mathds{1}_{i=j} \prod_{\substack{l=1 \\ l \ne j}}^p\beta_{0l}\Big)  (-1)^{k_. -1}  \left(\prod_{j=1}^p \beta_{0j}^{k_j -1}  \right)  f^{(k_.)}(y-\beta^\top_0 x) 
\end{eqnarray*}
with $k_. = \sum_{j=1}^p k_j.$  Also, it is easy to show that 
\begin{eqnarray*}
D^k(f(y - \beta^\top_0 x)) = (-1)^{k_.} \left(\prod_{j=1}^p \beta_{0j}^{k_j} \right) f^{(k_.)}(y-\beta_0^\top x). 
\end{eqnarray*}
Using the same notation as in \cite[Section 2.7]{aadbookE2} for the H\"older classes, define for a function $g: \mathcal X \mapsto \mathbb R$ and $\gamma > 0$
\begin{eqnarray}\label{normgamma}
\Vert g \Vert_\gamma  =  \max_{0 \le k_. \le \underline{\gamma}} \sup_{x} \vert D^k g(x) \vert +  \max_{k_. = \underline{\gamma}}  \sup_{x, y}   \frac{\vert D ^k g(x)  -  D^k g(y)  \vert }{\Vert x- y \Vert^{\gamma - \underline{\gamma}}}  
\end{eqnarray}
where $\underline{\gamma}$ is the greatest integer strictly smaller than $\gamma$.  In the following, we will take $\gamma =p+1$ and hence $\underline{\gamma} = p$. 

Recall that $\mathcal X \subseteq \overline{\mathcal B}(0, B)$ and $\beta_0 \in \overline{\mathcal B}(0, R^\ast)$.   Then, for $i \in \{1, \ldots, p \}$ and $k = (k_1, \ldots, k_p) \in \mathbb N^p_0$
\begin{eqnarray}\label{Dk}
 \vert D^k (x_i f'(y-\beta_0^\top x)) \vert & \le  &  B (R^\ast)^{k_. } \ \vert f^{(k_.+1)}(y-\beta_0^\top x) \vert \notag \\
&& + \  p \ (\max_{1 \le j \le p} k_j)  (R^\ast)^{p-1 + \sum_{j=1}^p (k_j - 1) }  \   \vert f^{(k_.)}(y-\beta_0^\top x) \vert \notag \\
&& =  B (R^\ast)^{k_. } \ \vert f^{(k_.+1)}(y-\beta_0^\top x) \vert \notag \\
&& \ + \  p \  (\max_{1 \le j \le p} k_j)  (R^\ast)^{k_. -1}  \   \vert f^{(k_.)}(y-\beta_0^\top x) \vert \notag \\
&&  \le   B (R^\ast \vee 1)^p \ \vert f^{(k_.+1)}(y-\beta_0^\top x) \vert \notag \\
&& \ + \   p^2  (R^\ast \vee 1)^{p-1}  \vert f^{(k_.)}(y-\beta_0^\top x) \vert
\end{eqnarray}
using the fact that $k_. \le p$. Also,
\begin{eqnarray}\label{Dk2}
 \vert D^k (f(y-\beta_0^\top x) )\vert \le       (R^\ast)^{k_.} \vert f^{(k_.)}(y-\beta_0^\top x) \vert  \le  (R^\ast \vee 1)^{p} \ \vert f^{(k_.)}(y-\beta_0^\top x) \vert.  
\end{eqnarray}
Assumption (A4) implies that for any integer $m \ge 0$, there exists a real constant $A_{m, \alpha} > 0$ such that for all $t \in \mathbb R$
\begin{eqnarray}\label{boundderiv}
 \vert f^{(m)}(t) \vert & \le  & A_{m, \alpha} \left(\sum_{j=0}^{m (\alpha-1)} \vert t \vert^{j}  \right) \exp(- d^{-\alpha} \vert t \vert^\alpha). 
\end{eqnarray}
To show the inequality in (\ref{boundderiv}), we will first show that there exists $P_{m, \alpha}$, a polynomial of degree $m (\alpha -1)$ such that 
\begin{eqnarray}\label{FormDerivm}
f^{(m)}(t)  = P_{m, \alpha}(t) \exp(-d^{-\alpha} t^\alpha)    
\end{eqnarray}
for $t > 0$.  We will show this by induction. If $\alpha =1$,   then  it is clear that $P_{m, 1}(t) = (-1)^m c_\alpha \alpha d^{-m}$. Now, we suppose that $\alpha > 1$.   For $m=0$, (\ref{FormDerivm}) is  satisfied  with $P_{m, \alpha} \equiv c_\alpha$. Suppose that it is satisfied for $m \ge 1$.  Then, for $t > 0$
\begin{eqnarray*}
f^{(m+1)}(t)  & = & \left(P'_{m, \alpha}(t)  - d^{-\alpha} \alpha t^{\alpha -1} P_{m,\alpha}(t) \right) \exp(-d^{-\alpha} t^\alpha)  \\
& = &  P_{m+1,\alpha}(t)  \exp(-d^{-\alpha} t^\alpha),
\end{eqnarray*}
where $P_{m+1,\alpha}$ is the sum of  polynomials of degree $m (\alpha-1) - 1$ and $m (\alpha-1)  + \alpha -1 = (m+1)(\alpha -1)$ respectively. This implies that 
$P_{m+1,\alpha}$ is a polynomial of degree $(m+1)(\alpha -1)$, and hence (\ref{FormDerivm}) is satisfied for $m+1$. Thus, for $j = 0, \ldots, m (\alpha -1)$, there exist real numbers $a_{j, m,\alpha}$ such that 
\begin{eqnarray*}
P_{m,\alpha}(t)  =  \sum_{j=0}^{m (\alpha -1)}  a_{j, m,\alpha}  t^j.
\end{eqnarray*}
Taking
$$
A_{m, \alpha}  = \max_{ 0 \le j \le m (\alpha -1)}  \vert a_{j, m,\alpha} \vert
$$
and using symmetry implies that the claimed identity in (\ref{boundderiv}) holds.

Now, for $i \in \{1, \ldots, p \}$, let $g_i(x)$ denote the $i$-th component of $d_n$.  Then, using the expression in (\ref{Dk}) and (\ref{boundderiv}) it holds for $k = (k_1, \ldots, k_p) \in \mathbb N^p_0: 0 \le k_. \le \underline \gamma = p$ that
\begin{eqnarray*}
&&  \sup_{x \in \overline{\mathcal B}(0, B)}  \vert D^k g_i(x)  \vert \\
&& \le   \int   \frac{ \vert \int f(y- \beta_0^\top z) d\mathbb G^{\tilde X}_n(z)  \vert }{\int f(y- \beta_0^\top z) d\mathbb P^{\tilde X}_n(z)} \times \\
 && \  \  \ \  \  \  \   \Bigg [B (R^\ast \vee 1)^{p}  A_{k_. +1, \alpha}  \sup_{x  \in \overline{\mathcal B}(0, B)} \left \{\sum_{j=0}^{(k_. +1)(\alpha -1)} \vert y - \beta^\top_0 x \vert^{j} \exp(-d^{-\alpha} \vert y - \beta^\top_0 x \vert^\alpha)  \right \}  \\
 && \   \  \  \  \ \  + \  p^2  (R^\ast \vee 1)^{p-1}  A_{k_., \alpha} \sup_{x  \in \overline{\mathcal B}(0, B)} \left \{\sum_{j=0}^{k_. (\alpha-1)} \vert y - \beta^\top_0 x \vert^{j} \exp(-d^{-\alpha} \vert y - \beta^\top_0 x \vert^\alpha)  \right \}    \\
 &&  \  \  \   \ \ \ \ \   +    \left \vert \frac{\int z_j f'(y- \beta^\top_0 z) f^X(z) dz }{\int  f(y- \beta^\top_0 z) f^X(z) dz}  \right \vert \\
 && \ \ \ \ \ \ \ \ \ \times \ \  (R^\ast \vee 1)^{p}  A_{k_., \alpha} \sup_{x  \in \overline{\mathcal B}(0, B)} \left \{\sum_{j=0}^{k_. (\alpha-1)} \vert y - \beta^\top_0 x \vert^{j} \exp(-d^{-\alpha} \vert y - \beta^\top_0 x \vert^\alpha)  \right \} \Bigg ] dy.
\end{eqnarray*}
Using the fact that $ 1 \le m \le p+1$ we can write that
\begin{eqnarray*}
\sum_{j=0}^{m (\alpha -1)} \vert y - \beta^\top_0 x \vert^{j} & \le   &  \sum_{j=0}^{(p+1)(\alpha -1)} (2^{j  - 1} \vee 1) (\vert y \vert^{j} +  C^{j} ) \\
& \le & (2^{(p+1)(\alpha -1)  - 1} \vee 1) \sum_{j=0}^{(p+1)(\alpha -1)} \vert y \vert^{j}  + D_\alpha
\end{eqnarray*}
with $C =  R^\ast B$ and $D_\alpha =  (2^{(p+1)(\alpha -1)  - 1} \vee 1) \sum_{j=0}^{(p+1)(\alpha-1)} C^{j}$. As done above in the proof of consistency we can show that 
\begin{eqnarray}\label{ImpBlock}
\frac{\sup_{x \in \overline{\mathcal B}(0, B)}f(y-\beta_0^\top x) }{\int f(y-\beta^\top_0 z) d\mathbb P^{\tilde X}_n(z)}  \le  G_2(y) =  M' \mathds{1}_{\vert y \vert \le C} +  \exp(\alpha d^{-\alpha} \tilde{C} \vert y \vert^{\alpha -1} ) \mathds{1}_{\vert y \vert > C};
\end{eqnarray}
see also the expression $G_2$ in (\ref{G2}). 
Also, we have that 
\begin{eqnarray*}
 \int z  f'(y- \beta^\top_0 z)  f^X(z)dz  = \int z  \frac{f'(y- \beta^\top_0 z)}{\sqrt{f(y- \beta^\top_0 z)}}  \sqrt{f(y- \beta^\top_0 z)} f^X(z)dz.    
\end{eqnarray*}
By the Cauchy-Schwarz inequality, it follows that 
\begin{eqnarray*}
 \frac{\int \Vert z \Vert\vert f'(y- \beta^\top_0 z) \vert f^X(z)dz}{\int f(y- \beta^\top_0 z) f^X(z) dz} & \le   &  \int \Vert z \Vert^2 \frac{(f'(y- \beta^\top_0 z))^2}{f(y- \beta^\top_0 z)} f^X(z) dz,
\end{eqnarray*}
where 
\begin{eqnarray*}
\frac{(f'(y- \beta^\top_0 z))^2}{f(y- \beta^\top_0 z)} &  =  &   \frac{\alpha^2 c^2_\alpha d^{-2\alpha}\vert y - \beta_0^\top z\vert^{2(\alpha -1)} \exp(-2 d^{-\alpha} \vert y- \beta_0^\top z \vert^\alpha)}{c_\alpha  \exp(- d^{-\alpha} \vert y- \beta_0^\top z \vert^\alpha)}   \\
& = &  \alpha^2 c_\alpha d^{-2\alpha} \vert y - \beta_0^\top z\vert^{2(\alpha -1)} \exp(- d^{-\alpha} \vert y- \beta_0^\top z \vert^\alpha) \\
& \le  &   \alpha^2 c_\alpha d^{-2\alpha} ( 2^{2(\alpha -1) -1}  \vee 1) (\vert y \vert^{2(\alpha -1)}  + C^{2(\alpha -1)}) \exp(- d^{-\alpha} 2^{1-\alpha} \vert y \vert^\alpha +  d^{-\alpha} C^\alpha)
\end{eqnarray*}
Thus, 
\begin{eqnarray}\label{ImpBlock2}
&&\frac{\int \vert z_j \vert  \ \vert f'(y- \beta^\top_0 z) \vert f^X(z)dz}{\int f(y- \beta^\top_0 z) f^X(z) dz} \notag \\
&& \leq \frac{\int \Vert z \Vert \  \vert f'(y- \beta^\top_0 z) \vert f^X(z)dz}{\int f(y- \beta^\top_0 z) f^X(z) dz} \notag \\
 && \le  
 \mathbb E[\Vert X \Vert^2] \  \alpha^2 c_\alpha d^{-2\alpha}  ( 2^{2(\alpha -1) -1}  \vee 1) (\vert y \vert^{2(\alpha -1)}  + C^{2(\alpha -1)}) \exp(- d^{-\alpha} 2^{1-\alpha} \vert y \vert^\alpha +  d^{-\alpha} C^\alpha)  \notag \\
 &&  \lesssim (\vert y \vert^{2(\alpha -1)}  + C^{2(\alpha -1)}) \exp(- d^{-\alpha} 2^{1-\alpha} \vert y \vert^\alpha).
\end{eqnarray}
Note that this upper bound is sharper than the one obtained in (\ref{G1}).   Then, there exists $\tilde D_\alpha > 0$ such that 
\begin{eqnarray*}
&&  \sup_{x \in \overline{\mathcal B}(0, B)}  \vert D^k g_i(x)  \vert \\
&& \lesssim  \int \left \vert \int f(y- \beta_0^\top z) d\mathbb G^{\tilde X}_n(z) \right \vert \Big(\sum_{j=0}^{(p+1)(\alpha-1)} \vert y \vert^{j}  + \tilde{D}_\alpha \Big) \Big(1+ 
(\vert y \vert^{2(\alpha -1)}  + C^{2(\alpha -1)}) \exp(- d^{-\alpha} 2^{1-\alpha} \vert y \vert^\alpha) \Big) \\
&&  \ \  \ \  \  \  \ \times \  \left(M' \mathds{1}_{\vert y \vert \le C} +  \exp(\alpha d^{-\alpha} \tilde{C} \vert y \vert^{\alpha -1} ) \mathds{1}_{\vert y \vert > C}\right) dy \\
&& \lesssim  \int \left \vert \int f(y- \beta_0^\top z) d\mathbb G^{\tilde X}_n(z) \right \vert \Big(\sum_{j=0}^{(p+1)(\alpha-1)} \vert y \vert^{j}  +\sum_{j=0}^{(p+1)(\alpha-1)} \vert y \vert^{j+2(\alpha-1)}  +\tilde{D}'_\alpha \Big) \\
&&  \ \  \ \  \  \  \ \times \  \left(M' \mathds{1}_{\vert y \vert \le C} +  \exp(\alpha d^{-\alpha} \tilde{C} \vert y \vert^{\alpha -1} ) \mathds{1}_{\vert y \vert > C}\right) dy \\
&&  \lesssim  \int \left \vert \int f(y- \beta_0^\top z) d\mathbb G^{\tilde X}_n(z) \right \vert \Big(\sum_{j=0}^{(p+3)(\alpha-1)} \vert y \vert^{j}  + \tilde{D}''_\alpha \Big)  \times \left(M' \mathds{1}_{\vert y \vert \le C} +  \exp(\alpha d^{-\alpha} \tilde{C} \vert y \vert^{\alpha -1} ) \mathds{1}_{\vert y \vert > C}\right) dy
\end{eqnarray*}
for some constant $\tilde D'_\alpha$. Above, note that we have used the fact that $ \exp(- d^{-\alpha} 2^{1-\alpha} \vert y \vert^\alpha) \le 1$.
$$
$$
The goal now is to show that for $1 \le i  \le p$ 
$\sup_{x \in \overline{\mathcal B}(0, B)}  \vert D^k g_i(x)  \vert = O_{\mathbb P^{\tilde X}}(1)$ for all $k = (k_1, \ldots, k_p) \in \mathbb N^p_0 $ such that $0 \le k_. \le p$.  For a fixed $y \in \mathbb R$, denote by $\sigma^2(y) = \text{var}(f(y- \beta_0^\top \tilde{X}))$. More specifically, we have that
\begin{eqnarray*}
  \sigma^2(y)  & =   &  \int f^2(y - \beta_0^\top x) f^X(x) dx - 
 \left(\int f(y-\beta_0^\top x) f^X(x) dx\right)^2   \\
 & =  &  \int f^2(y - \beta_0^\top x) f^X(x) dx   - (f^Y(y))^2.
\end{eqnarray*}
Note that above we used the fact that $\tilde X$ and $X$ have the same distribution. Let $\gamma_\alpha = \alpha d^{-\alpha} \tilde C$. For a fixed $j  \in \{0, \ldots, (p+3)(\alpha-1) \} $ we have that 
\begin{eqnarray*}
&& \int \left \vert \int f(y- \beta^\top_0 x) d\mathbb G^{\tilde X}_n(x) \right \vert  \vert y \vert^{j} \exp(\gamma_\alpha  \vert y \vert^{\alpha -1} )  \mathds{1}_{\vert y \vert > C} dy  \\
&& =  \int \left \vert \int \frac{f(y- \beta^\top_0 x)}{\sigma(y)} d\mathbb G^{\tilde X}_n(x) \right \vert  \sigma(y) \ \vert y \vert^{j} \exp(\gamma_\alpha  \vert y \vert^{\alpha -1} ) \mathds{1}_{\vert y \vert > C}  dy \\
&& \le \int \left \vert \int \frac{f(y- \beta^\top_0 x)}{\sigma(y)} d\mathbb G^{\tilde X}_n(x) \right \vert  \left(\int f^2(y - \beta_0^\top x) f^X(x) dx \right)^{1/2} \vert y \vert^{j} \exp(\gamma_\alpha  \vert y \vert^{\alpha -1} ) dy
\end{eqnarray*}
Since $ \vert y - \beta_0^\top x \vert^\alpha \ge 2^{1-\alpha} \vert y \vert^\alpha -  (\Vert \beta_0 \Vert B)^\alpha$ it follows that
\begin{eqnarray*}
\left(\int f^2(y - \beta_0^\top x) f^X(x) dx\right)^{1/2}  & \le  &  c_\alpha \exp\big(- 2^{1-\alpha} d^{-\alpha} \vert y\vert^\alpha + d^{-\alpha} (\Vert \beta_0 \Vert B)^\alpha\big) \\
& \le &  c_\alpha \exp\big(- 2^{1-\alpha} d^{-\alpha} \vert y\vert^\alpha + d^{-\alpha} C^{\alpha} \big).
\end{eqnarray*}
Furthermore, for $K > 0$ the Chebyshev's inequality implies that
\begin{eqnarray*}
\mathbb P \left(\left \vert \int \frac{f(y- \beta^\top_0 x)}{\sigma(y)} d\mathbb G^{\tilde X}_n(x) \right \vert >  K \right  )  \le   \frac{1}{K^2}.   
\end{eqnarray*}
Thus, with probability $\ge 1- K^{-2}$ we have that 
\begin{eqnarray*}
&& \int \left \vert \int f(y- \beta_0^\top z) d\mathbb G^{\tilde X}_n(z) \right \vert \Big( \sum_{j=0}^{(p+3)(\alpha-1)} \vert y \vert^{j}  + \tilde D''_\alpha \Big)  \exp(\gamma_\alpha \vert y \vert^{\alpha -1} ) \mathds{1}_{\vert y \vert > C} dy \\   
&& \lesssim \int \left(\sum_{j=0}^{(p+3)(\alpha-1)} \vert y \vert^{j} + \tilde D''_\alpha\right)  \exp\big(- 2^{1-\alpha} d^{-\alpha} \vert y\vert^\alpha + \gamma_\alpha  \vert y \vert^{\alpha -1})\mathds{1}_{\vert y \vert > C} dy  \\
&& :=  M_{1}(K, \alpha)  < \infty.
\end{eqnarray*}
Also, 
\begin{eqnarray*}
&& \int \left \vert \int f(y- \beta_0^\top z) d\mathbb G^{\tilde X}_n(z) \right \vert \Big( \sum_{j=0}^{(p+3)(\alpha -1)} \vert y \vert^{j}  + \tilde D''_\alpha \Big) \mathds{1}_{\vert y \vert \le C} dy  \\
&& \le K c_\alpha  \exp( d^{-\alpha} C^\alpha) \int \exp(-2^{1-\alpha} d^{-\alpha} \vert y \vert^\alpha) \left( \sum_{j=0}^{(p+3)(\alpha -1)} \vert y \vert^{j}  + \tilde D''_\alpha  \right)  \mathds{1}_{\vert y \vert \le C} dy \\
&& \le K c_\alpha  \exp( d^{-\alpha} C^\alpha) \int \left( \sum_{j=0}^{(p+3)(\alpha -1)} \vert y \vert^{j}  + \tilde D''_\alpha  \right)  \mathds{1}_{\vert y \vert \le C} dy \\
&& := M_2(K, \alpha) < \infty. \\
\end{eqnarray*}
Hence, there exists  $M(K, \alpha) > 0$ depending on  $M_1(K, \alpha)$ and  $M_2(K, \alpha)$ such that for all $1 \le i  \le p$
\begin{eqnarray*}
 \sup_{x \in \overline{\mathcal B}(0, B)}  \vert D^k g_i(x)  \vert  \lesssim M(K, \alpha)
\end{eqnarray*}
with probability $\ge 1 - 2/K^2$ for $K > 1/\sqrt 2$. As for the second term in the definition (\ref{normgamma}), recall that $\gamma = p+1$ and $ \underline \gamma =p$, and hence 
\begin{eqnarray*}
\max_{k_. =  p} \sup_{x, y \in \overline{\mathcal B}(0, B)}   \frac{D^k g_i(x)  - D^k g_i(y)}{\Vert x - y \Vert}    \le \max_{s_. =  p+1} \sup_{z \in \overline{\mathcal B}(0, B)}   \vert D^{s} g_i(z)  \vert 
\end{eqnarray*}
here $s = (s_1, \ldots, s_p) \in \mathbb N^p_0$. Using the same calculations as before, we conclude that for $i \in \{1, \ldots, p \}$
\begin{eqnarray}\label{Normgi}
\Vert g_i \Vert_{\gamma}  =  O_{\mathbb P^{\tilde X}}(1)
\end{eqnarray}
for $\gamma = p+1$. \\

\medskip
Now, we handle the case $k_. = 0$, or equivalently $k_1 = \ldots = k_p = 0$.  Using the same arguments as above, we can show again that 
\begin{eqnarray}\label{supnorm}
\Vert d_n \Vert_\infty = O_{\mathbb P^{\tilde X}}(1).
\end{eqnarray}
Recall that our first goal is to show that  
$$\sqrt n \int \left(\int D_n(x, y)  f^Y(y) dy\right)  d\mathbb G_n^{\tilde X}(x) = - \int d_n(x) d\mathbb G_n^{\tilde X}(x)  = O_{\mathbb P^{\tilde X}}(1).$$ 
From  (\ref{Normgi}) and  (\ref{supnorm}),  we conclude that with probability tending to $1$, $d_n$ belongs to the H\"older space $C^{p+1}(\overline{\mathcal{B}}(0, B))_D$ for some constant $D > 0$ which depends on $\alpha, B$, $R^\ast$ and the dimension $p$. Since all the elements of  $C^{p+1}(\overline{\mathcal B}(0, B))_D$ have a supremum norm bounded above by $D$, this class admits $D$ also as an envelope.  It follows from \cite[Corollary 2.7.2]{aadbookE2} that for all $\eta > 0$
\begin{eqnarray}\label{unientropholder}
 \log N\left(\eta,  C^{p+1}(\overline{\mathcal{B}}(0, B))_D, L_2(Q)   \right) \le L \left(\frac{1}{\eta}   \right)^{p/(p+1)}
\end{eqnarray}
for all probability measures $Q$ on $\mathbb R^p$ and $L > 0$ some constant depending (through $D$) on $\alpha, B, R^\ast$ and $p$. Using (\ref{unientropholder}), the value of the uniform entropy of $C^{p+1}(\overline{\mathcal B}(0, B))_D$ at $1$ satisfies
\begin{eqnarray*}
J(1,  C^{p+1}(\overline{\mathcal B}(0, B))_D)  & \le  &  \int_0^1\left(1  +  L \left(\frac{1}{\eta}\right)^{p/(p+1)}   \right)^{1/2} d\eta  \\
& \le   &  1 +  \sqrt L \int_0^1 \frac{1}{\eta^{p/(2(p+1))}} d\eta  =   1 + \frac{2 (p+1) \sqrt L}{p +2}. 
\end{eqnarray*}
By \cite[2.14.1]{aadbookE2}, we conclude that
$$
\mathbb E[\Vert \mathbb G^{\tilde X}_n \Vert_{C^{p+1}(\overline{\mathcal B}(0, B))_D)}]  \lesssim  J(1, C^{p+1}(\overline{\mathcal B}(0, B))_D))  = O(1).
$$
Now, for $M > 0$
\begin{eqnarray}\label{FinalRes}
\mathbb P\left( \left \vert  \int d_n(x)  d\mathbb G_n^{\tilde X}(x) \right \vert > M \right) & = & \mathbb P\left( \left \vert  \int d_n(x)  d\mathbb G_n^{\tilde X}(x) \right \vert > M, d_n \notin C^{p+1}(\overline{\mathcal B}(0, B))_D) \right)  \notag \\
&& \ + \  \mathbb P\left( \left \vert  \int d_n(x)  d\mathbb G_n^{\tilde X}(x) \right \vert > M, d_n \in C^{p+1}(\overline{\mathcal B}(0, B))_D) \right)   \notag \\
& \leq  & o(1)  +  \frac{1}{M}  \mathbb E[\Vert \mathbb G^{\tilde X}_n \Vert_{C^{p+1}(\overline{\mathcal B}(0, B))_D)}],  \
 \textrm{by the Markov's inequality},  \notag\\
& \to &  0, \ \ \textrm{as $M \to \infty$}. 
\end{eqnarray}
This shows that $\int d_n(x) d\mathbb G^{\tilde X}_n(x) = O_{\mathbb P^{\tilde X}}(1)$. 

\bigskip
\bigskip

Next, we will show that 
$\int \left(m_n(x, y) d\mathbb G_n^{\tilde Y}(y)\right) d\mathbb G^{\tilde X}_n(x) = O_{\mathbb  P^{\tilde X} \times \mathbb  P^{\tilde Y}}(1)$.  We have that 
\begin{eqnarray*}
&& s_n(x):= \int m_n(x, y) d\mathbb G_n^{\tilde Y}(y)   \\
&& = \int  \frac{1}{\int f(y-\beta_0^\top z) d\mathbb P^{\tilde X}_n(z)} \left( x f'(y- \beta_0^\top x) - \frac{\int  z   f'(y- \beta^\top_0 z) f^X(z)dz}{\int f(y- \beta^\top_0 z) f^X(z) dz}  f(y-\beta_0^\top x) \right) d \mathbb G^{\tilde{Y}}_n(y),
\end{eqnarray*}
which is a function in $x \in \mathcal X$ but also random as it involves the responses in the unmatched sample.  Let $s_{j,n}$ be the $j$-th component of $s_n$. The main idea is to show that, with probability tending to $1$, $x \mapsto s_{j,n}(x)$ belongs to a \lq\lq nice\rq\rq \ class of functions $\mathcal S$ so that  $\mathbb E[\Vert \mathbb G^{\tilde X}_n  \Vert_{\mathcal S} ] \lesssim 1$. As done above, we will next show that 
\begin{eqnarray*}
s_{n, j}  \in C^{p+1}(\overline{\mathcal B}(0, B))_D \equiv \mathcal S
\end{eqnarray*}
with large probability, and where $D > 0$ is a constant not necessarily equal to the one exhibited above.   Let $k = (k_1, \ldots, k_p) \in \mathbb N^p_0$. We start with the case where $k_i = 0$ for $i = 1, \ldots, p$.  The vector $x \in \mathbb R^p$ can be seen a parameter indexing the function
\begin{eqnarray*}
y &\mapsto & m_n(x, y)\\
& = &  \frac{1}{\int f(y-\beta_0^\top z) d\mathbb P^{\tilde X}_n(z)} \left( x f'(y- \beta_0^\top x) - \frac{\int  z   f'(y- \beta^\top_0 z) f^X(z)dz}{\int f(y- \beta^\top_0 z) f^X(z) dz}  f(y-\beta_0^\top x) \right)\\
&:=&  q_x(y)
\end{eqnarray*}
For $j \in \{1, \ldots, p\}$, let us denote by $q_{x, j}$ the $j$-th component of $q_x$; i.e., 
\begin{eqnarray*}\label{qxj}
q_{x, j}(y)  =  \frac{1}{\int f(y-\beta_0^\top z) d\mathbb P^{\tilde X}_n(z)} \left( x_j f'(y- \beta_0^\top x) - \frac{\int  z_j   f'(y- \beta^\top_0 z) f^X(z)dz}{\int f(y- \beta^\top_0 z) f^X(z) dz}  f(y-\beta_0^\top x) \right), \ y \in \mathbb R.
\end{eqnarray*}
Computing the gradient of this function with respect of $x$ yields 
\begin{eqnarray*}
\nabla q_{x, j}(y)  & =   &   \frac{1}{\int f(y-\beta_0^\top z) d\mathbb P^{\tilde X}_n(z)} \left(\mathbb{I}_j  f'(y- \beta_0^\top x)   -  x_j \beta_0 f''(y- \beta_0^\top x) \right)  \\
&& \  +  \ \beta_0 \frac{f'(y- \beta_0^\top x)}{\int f(y-\beta_0^\top z) d\mathbb P^{\tilde X}_n(z)}  \frac{\int  z_j   f'(y- \beta^\top_0 z) f^X(z)dz}{\int f(y- \beta^\top_0 z) f^X(z) dz}     \\
& = &  a_n(y)  + b_n(y)
\end{eqnarray*}
where $\mathbb I_j = (0, 0, \ldots, 1, 0, \ldots, 0)$ with $1$ at the $j$-th position.  Note that
\begin{eqnarray*}
 \frac{\vert f'(y-\beta^\top_0 x)\vert}{\int f(y-\beta_0^\top z) d\mathbb P^{\tilde X}_n(z)} \le \frac{\alpha d^{-\alpha} \vert y - \beta_0^\top x \vert^{\alpha -1} f(y-\beta^\top_0 x)}{\int f(y-\beta_0^\top z) d\mathbb P^{\tilde X}_n(z)}  
\end{eqnarray*}
and 
\begin{eqnarray*}
 \frac{\vert f''(y-\beta^\top_0 x) \vert}{\int f(y-\beta_0^\top z) d\mathbb P^{\tilde X}_n(z)} \le \frac{\alpha d^{-\alpha} \Big((\alpha-1) \mathds{1}_{\alpha \ge 2} \ \vert y - \beta_0^\top x \vert^{\alpha -2} +  d^{-\alpha} \alpha \ \vert y - \beta_0^\top x \vert^{2(\alpha -1)} \Big) f(y-\beta^\top_0 x)}{\int f(y-\beta_0^\top z) d\mathbb P^{\tilde X}_n(z)}  
\end{eqnarray*}
Using the inequality in (\ref{ImpBlock}), and 
$$\vert y - \beta_0^\top x \vert^{\alpha -1}  \le (2^{\alpha -2} \vee 1) (\vert y \vert^{\alpha-1}  +  C^{\alpha-1})$$
for all $\alpha \ge 1$, 
$$\vert y - \beta_0^\top x \vert^{\alpha -2}  \le (2^{\alpha -3} \vee 1) (\vert y \vert^{\alpha-2}  +  C^{\alpha-2})$$
for all $\alpha \ge 2$, and 
$$\vert y - \beta_0^\top x \vert^{2(\alpha -1)}  \le (2^{2(\alpha -1) -1} \vee 1) (\vert y \vert^{2(\alpha-1)}  +  C^{2(\alpha-1)})$$
for all $\alpha \ge 1$, we can write that
\begin{eqnarray*}
 && \frac{\Vert \mathbb{I}_j f'(y-\beta^\top_0 x)\Vert}{\int f(y-\beta_0^\top z) d\mathbb P^{\tilde X}_n(z)} \\
 &&\le \alpha d^{-\alpha} (2^{\alpha -2} \vee 1) (\vert y \vert^{\alpha-1}  +  C^{\alpha-1}) \left( M'\mathds{1}_{\vert y \vert \le C} + \exp(\alpha d^{-\alpha} \tilde C \vert y \vert^{\alpha -1}) \mathds{1}_{\vert y \vert > C} \right),
 \end{eqnarray*}
 and 
\begin{eqnarray*}
 && \frac{\Vert x_j \beta_0 f''(y-\beta^\top_0 x) \Vert}{\int f(y-\beta_0^\top z) d\mathbb P^{\tilde X}_n(z)} \\
 && \le C  \alpha d^{-\alpha} \left((\alpha -1) \mathds{1}_{\alpha \ge 2} (2^{\alpha -3} \vee 1) (\vert y \vert^{\alpha-2}  +  C^{\alpha-2}) + \alpha d^{-\alpha} (2^{2\alpha -3} \vee 1) (\vert y \vert^{2(\alpha-1)}  +  C^{2(\alpha-1)}) \right) \\
 && \ \times \left( M'\mathds{1}_{\vert y \vert \le C} + \exp(\alpha d^{-\alpha} \tilde C \vert y \vert^{\alpha -1}) \mathds{1}_{\vert y \vert > C} \right)
 \end{eqnarray*}
and hence
\begin{eqnarray*}
\Vert a_n(y) \Vert \le A_\alpha \left(\vert y \vert^{2(\alpha -1)} + \vert y \vert^{\alpha -1}  + \vert y\vert^{\alpha -2} \mathds{1}_{\alpha \ge 2}  +  B_\alpha \right) \times \left( M'  \mathds{1}_{\vert y \vert \le C}  + \exp(\alpha d^{-\alpha} \tilde C \vert y \vert^{\alpha -1}) \mathds{1}_{\vert y \vert > C}\right) 
\end{eqnarray*}
for some constants $A_\alpha > 0, B_\alpha > 0$.  For the term $b_n(x)$, note that we can show as in the proof of consistency that
\begin{eqnarray*}
\frac{\vert f'(y- \beta_0^\top x) \vert}{\int f(y-\beta_0^\top z) d\mathbb P^{\tilde X}_n(z)}  &\le &  G_1(y) \\
&= & M  \mathds{1}_{\vert y \vert \le C}  + (2^{\alpha-2} \vee 1)\alpha d^{-\alpha}(\lvert y\rvert ^{\alpha-1} + C^{\alpha-1})\exp(\alpha d^{-\alpha} \tilde C \vert y \vert^{\alpha -1}) \mathds{1}_{\vert y \vert > C};  
\end{eqnarray*}
see also the expression of $G_1$ in (\ref{G1}).  Also, using the inequality proved in (\ref{ImpBlock2}),  we have that  
\begin{eqnarray*}
\Vert b_n(y) \Vert & \le  &  R^\ast \ \mathbb E[\Vert X \Vert^2] \  \alpha^2 c_\alpha d^{-2\alpha}   (2^{2(\alpha -1)-1} \vee 1) (\vert y \vert^{2(\alpha -1)}  + C^{2(\alpha -1)}) \exp(- d^{-\alpha} 2^{1-\alpha} \vert y \vert^\alpha + d^{-\alpha} C^\alpha)  \\
&&  \times \  \left(M  \mathds{1}_{\vert y \vert \le C}  + (2^{\alpha-2} \vee 1)\alpha d^{-\alpha}(\lvert y\rvert ^{\alpha-1} + C^{\alpha-1})\exp(\alpha d^{-\alpha} \tilde C \vert y \vert^{\alpha -1}) \mathds{1}_{\vert y \vert > C} \right) \\
& \lesssim &   (\vert y \vert^{2(\alpha -1)}  + C^{2(\alpha -1)}) \left(M  \mathds{1}_{\vert y \vert \le C}  + (2^{\alpha-2} \vee 1)\alpha d^{-\alpha}(\lvert y\rvert ^{\alpha-1} + C^{\alpha-1})\exp(\alpha d^{-\alpha} \tilde C \vert y \vert^{\alpha -1}) \mathds{1}_{\vert y \vert > C} \right).
\end{eqnarray*}
This implies that 
\begin{eqnarray*}
\vert q_{x, j}(y)  - q_{x', j}(y)  \vert \le \Vert x - x' \Vert Q_j(y)
\end{eqnarray*}
where 
\begin{align*}
Q_j(y)  =    \tilde A_\alpha \left(\vert y \vert^{3(\alpha -1)} +\vert y \vert^{2(\alpha -1)} +\vert y \vert^{\alpha -1}  + \vert y\vert^{\alpha -2} \mathds{1}_{\alpha \ge 2} + \tilde B_\alpha \right) \times \left( \tilde M  \mathds{1}_{\vert y \vert \le C}  + \exp(\alpha d^{-\alpha} \tilde C \vert y \vert^{\alpha -1}) \mathds{1}_{\vert y \vert > C}\right)
\end{align*}
for some constants $\tilde A_\alpha, \tilde B_\alpha > 0$ and $\tilde M > 0$. Since $Q_j$ has a finite $L_2(\mathbb P^{\tilde Y})$-norm, we can use similar arguments as in the proof of consistency to show that 
\begin{eqnarray*}
\sup_{x \in \overline{\mathcal B}(0, B)} \left \vert \int q_{x, j}(y) d \mathbb G^{\tilde Y}_n(y) \right \vert  = O_{\mathbb P^{\tilde Y}}(1)
\end{eqnarray*}
for all $j \in \{1, \ldots, p \}$. This allows us to conclude that
$$
\sup_{x \in \overline{\mathcal B}(0, B)} \vert s_{n, j}(x) \vert = O_{\mathbb P^{\tilde Y}}(1). 
$$
Let $k = (k_1, \ldots, k_p) \in \mathbb N^p_0$ such that $k_. \le p$ and there exists at least one $j \in \{1, \ldots, p \}$ such that $k_j \ne 0$.  We are going to show that 
\begin{eqnarray}\label{Dksnj}
\sup_{x \in \overline{\mathcal B}(0, B)} \vert D^k s_{n, j}(x)  \vert  =  O_{\mathbb P^{\tilde Y}}(1).   
\end{eqnarray}
It is clear that
\begin{eqnarray*}
D^k s_{n, j}(x)  =   \int    D^k q_{x, j}(y) d\mathbb G^{\tilde Y}_n(y).  
\end{eqnarray*}
Using the calculations above we can write that
\begin{eqnarray*}
 D^k q_{x, j}(y) &= &   \frac{1}{\int f(y-\beta_0^\top z) d\mathbb P^{\tilde X}_n(z)} \bigg \{ x_j  (-1)^{k_.} (\prod_{l=1}^p \beta_{0l}^{k_l}) f^{(k_. + 1)}(y-\beta^\top_0 x) \\
&& \ +  \ \Big(  \  \Big( \sum_{l=1}^p k_l \mathds{1}_{j=l} \prod_{\substack{r=1 \\ r \ne l \\ }}^p\beta_{0r}\Big)  (-1)^{k_. -1}  (\prod_{l=1}^p \beta_{0l}^{k_l -1}  )  f^{(k_.)}(y-\beta^\top_0 x) \\ 
&& \ -  \  \frac{\int z_j f'(y- \beta_0^\top z) f^X(z) dz}{\int f(y- \beta_0^\top z) f^X(z) dz} (-1)^{k_.} (\prod_{l=1}^p \beta_{0l}^{k_l}) f^{(k_.)}(y-\beta^\top_0 x)\bigg \}.
\end{eqnarray*}
As done above for the case $k_j =0, 1 \le j \le p$, we will consider $y \mapsto D^k q_{x, j}(y)$ as a function which index by $x$. Next, we compute its gradient
\begin{eqnarray*}
\nabla [D^k q_{x, j}(y)] & =  & \frac{1}{\int f(y-\beta_0^\top z) d\mathbb P^{\tilde X}_n(z)} \bigg \{\mathbb I_j  (-1)^{k_.} (\prod_{l=1}^p \beta_{0l}^{k_l}) f^{(k_. + 1)}(y-\beta^\top_0 x) \\
&& \ +\  x_j \beta_0 (-1)^{k_. +1} (\prod_{l=1}^p \beta_{0l}^{k_l}) f^{(k_. + 2)}(y-\beta^\top_0 x)   \\
&&  \ + \ \beta_0 \Big( \sum_{l=1}^p k_l \mathds{1}_{j=l} \prod_{\substack{r=1 \\ r \ne l}}^p\beta_{0r}\Big)  (-1)^{k_.}  (\prod_{l=1}^p \beta_{0l}^{k_l -1}  )  f^{(k_.+1)}(y-\beta^\top_0 x) \\ 
&& \ +  \  \beta_0 \frac{\int z_j f'(y- \beta_0^\top z) f^X(z) dz}{\int f(y- \beta_0^\top z) f^X(z) dz} (-1)^{k_.} (\prod_{l=1}^p \beta_{0l}^{k_l}) f^{(k_.+1)}(y-\beta^\top_0 x)\bigg \}.     
\end{eqnarray*}
Using the bound in (\ref{boundderiv}), we can find a real constant $\tilde C_\alpha > 0$ depending on $\alpha, p, R^\ast, \beta_0$ and $B$ such that 
\begin{eqnarray*}
\Vert \nabla [D^k q_{x, j}(y)] \Vert  & \le  & \tilde C_\alpha \frac{f(y- \beta_0^\top x)}{\int f(y- \beta_0^\top z) d \mathbb P^{\tilde X}_n(z)}  \Bigg \{ \sum_{l=0}^{(k. +2)(\alpha -1)} \vert y - \beta_0^\top x \vert^{l } \\
&& \ +  \ \frac{\int \vert z_j \vert  \vert f'(y- \beta_0^\top z) \vert f^X(z) dz}{\int f(y- \beta_0^\top z) f^X(z) dz}  \sum_{l=0}^{(k. +1)(\alpha-1)} \vert y - \beta_0^\top x \vert^{l} \Bigg \}.
\end{eqnarray*}
By the inequalities in (\ref{ImpBlock}) and (\ref{ImpBlock2}) it follows that
\begin{eqnarray*}
\Vert \nabla [D^k q_{x, j}(y)] \Vert  & \le  & \tilde C_\alpha \left(M' \mathds{1}_{\vert y \vert \le C } +  \exp(\gamma_\alpha \vert y \vert^{\alpha -1} \mathds{1}_{\vert y \vert > C}  \right) \ \times \ \Bigg \{ \sum_{l=0}^{(k. +2)(\alpha-1)} \vert y - \beta_0^\top x \vert^{l} \\
&& \ + \  \left(\sum_{l=0}^{(k. +1)(\alpha-1)} \vert y - \beta_0^\top x \vert^{l}\right) (\vert y \vert^{2(\alpha -1)} +  C^{2(\alpha -1)}) \exp(-d^{-\alpha} 2^{1-\alpha} \vert y \vert^\alpha)    \Bigg \} \\
&&:= R_j(y).
\end{eqnarray*}
Note that the envelope $R_j$ has a finite $L_2(\mathbb P^{\tilde Y})$-norm. Hence, by similar arguments as above, we can show that (\ref{Dksnj}) holds true.  The same arguments allow us to show the stronger statement  
\begin{eqnarray*}
\Vert s_{n, j} \Vert_\gamma = O_{\mathbb P^{\tilde Y}}(1)
\end{eqnarray*}
for all $j \in \{1, \ldots, p \}$, with $\gamma = p+1$. This means that there exists a constant $D > 0$ such that for all $j \in \{1, \ldots, p \}$
\begin{eqnarray*}
x \mapsto s_{n, j}(x)  \in C^{p+1}(\overline{\mathcal B}(0, B))_D
\end{eqnarray*}
with large probability. We can conclude in the same way as done in (\ref{FinalRes}) that  $\int s_n(x) d\mathbb G^{\tilde X}_n(x) = O_{\mathbb P^{\tilde X} \otimes \mathbb P^{\tilde Y}}(1)$, and the theorem is proved.


\end{proof}

\bigskip
\bigskip

\begin{theorem}\label{Remainder2}
Let $\ddot{\ell}_{n,m}(\beta_0)$ be as in the proof of Theorem \ref{AsympNorm}. Then, under the assumptions (A0)-(A4), it holds that
\begin{eqnarray*}
\ddot{\ell}_{n,m}(\beta_0) \to_{\mathbb P \otimes \mathbb P^{\tilde X} \otimes \mathbb P^{\tilde Y}}  -\frac{1}{1 + \lambda} \Gamma_1 -  \frac{\lambda}{1 + \lambda} \Sigma_2
\end{eqnarray*}
where $\Gamma_1$ is as in (\ref{Gamma1}), and
$$
\Sigma_2 =  \left(\int \frac{(f'(t))^2}{f(t)} dt \right) \mathbb E[X X^\top].
$$

\end{theorem}

\begin{proof}
We have that
\begin{eqnarray*}
\ddot{\ell}_{n,m}(\beta_0)  = I_n + II_n    
\end{eqnarray*}
where
\begin{eqnarray*}
I_n & = &   \frac{1}{n+m} \sum_{j=1}^n \frac{\frac{1}{n}\sum_{i=1}^n f''(\tilde{Y}_j-\beta_0^\top \tilde{X}_i)(\tilde{X}_i \tilde{X}_i^\top)}{\frac{1}{n}\sum_{i=1}^n f(\tilde{Y}_j-\beta_0^\top \tilde{X}_i)} \\
&& \ \ -   \frac{1}{n+m} \sum_{j=1}^n\frac{\left(\frac{1}{n}\sum_{i=1}^n f'(\tilde{Y}_j-\beta_0^\top \tilde{X}_i)(\tilde{X}_i)\right)\left(\frac{1}{n}\sum_{i=1}^n f'(\tilde{Y}_j-\beta_0^\top \tilde{X}_i)(\tilde{X}_i^\top)\right)}{\left(\frac{1}{n}\sum_{i=1}^n f(\tilde{Y}_j-\beta_0^\top \tilde{X}_i)\right)^2}  \\
& = &  \frac{n}{n+m} \frac{1}{n}\sum_{j=1}^n \frac{\int  x x^\top  f''(\tilde{Y}_j  - \beta_0^\top x) d\mathbb P^{\tilde X}_n(x)} {\int  f(\tilde{Y}_j  - \beta_0^\top x) d\mathbb P^{\tilde X}_n(x)} \\
&& \ \ - \frac{n}{n+m} \frac{1}{n} \sum_{j=1}^n  \left(\frac{\int x f'(\tilde{Y}_j  - \beta_0^\top x) d\mathbb P^{\tilde X}_n(x)} {\int  f(\tilde{Y}_j  - \beta_0^\top x) d\mathbb P^{\tilde X}_n(x)} \right) \left(\frac{\int x^\top f'(\tilde{Y}_j  - \beta_0^\top x) d\mathbb P^{\tilde X}_n(x)} {\int  f(\tilde{Y}_j  - \beta_0^\top x) d\mathbb P^{\tilde X}_n(x)} \right) 
\end{eqnarray*}
and 
\begin{eqnarray*}
II_n & = &  \frac{1}{n+m}  \sum_{k=1}^m \frac{f''(Y_k-\beta_0^\top X_k)(X_k X_k^\top)}{f(Y_k - \beta_0^\top X_k)}  - \frac{1}{n+m} \sum_{k=1}^m \frac{\left(f'(Y_k-\beta_0^\top X_k)\right)^2(X_k X_k^\top)}{f(Y_k - \beta_0^\top X_k)^2} \\
& = &  \frac{m}{n+m} \int xx^\top \left(\frac{f''(y-\beta_0^\top x)}{f(y-\beta_0^\top x)} - \left(\frac{f'(y-\beta_0^\top x)}{f(y-\beta_0^\top x )}\right)^2 \right) d\mathbb P_m(x, y). 
\end{eqnarray*}
We start with $II_n$. We have that  
\begin{eqnarray*}
\int xx^\top \frac{f''(y-\beta_0^\top x)}{f(y-\beta_0^\top x)} d\mathbb P(x, y) & = &   \int xx^\top f''(y-\beta_0^\top x) f^X(x) dx dy \\
& = &  \mathbb E[XX^\top] \int f''(t) dt  =0
\end{eqnarray*}
using the change of variable $t = y - \beta_0^\top x$. Also, we have already established above that 
\begin{eqnarray*}
\int xx^\top \frac{(f'(y-\beta_0^\top x))^2}{f(y-\beta_0^\top x)} f^X(x) dx dy = \mathbb E[XX^\top]  \int \frac{(f'(t))^2}{f(t)} dt  = \Sigma_2
\end{eqnarray*}
it follows by the SLLN that $II_n \to_{\mathbb P}  -(\lambda/(1+\lambda)) \Sigma_2$.
Next, we will show that $I_n \to_{\mathbb P^{\tilde X} \otimes \mathbb P^{\tilde Y}}  -(1+ \lambda)^{-1} \Gamma_1$. To this aim, it is enough to show that
\begin{eqnarray}\label{FormerConv}
\frac{1}{n}\sum_{j=1}^n \frac{\int  x x^\top  f''(\tilde{Y}_j  - \beta_0^\top x) d\mathbb P^{\tilde X}_n(x)} {\int  f(\tilde{Y}_j  - \beta_0^\top x) d\mathbb P^{\tilde X}_n(x)} \to_{\mathbb P^{\tilde X} \otimes \mathbb P^{\tilde Y}} 0
\end{eqnarray}
and 
\begin{eqnarray}\label{LatterConv}
\frac{1}{n} \sum_{j=1}^n  \left(\frac{\int x f'(\tilde{Y}_j  - \beta_0^\top x) d\mathbb P^{\tilde X}_n(x)} {\int  f(\tilde{Y}_j  - \beta_0^\top x) d\mathbb P^{\tilde X}_n(x)} \right) \left(\frac{\int x^\top f'(\tilde{Y}_j  - \beta_0^\top x) d\mathbb P^{\tilde X}_n(x)} {\int  f(\tilde{Y}_j  - \beta_0^\top x) d\mathbb P^{\tilde X}_n(x)} \right)\to_{\mathbb P^{\tilde X} \otimes \mathbb P^{\tilde Y}} \Gamma_1.
\end{eqnarray}
We start with showing the convergence in (\ref{FormerConv}). As done in the proof of Theorem \ref{Remainder} we can write that
\begin{eqnarray*}
\frac{1}{n}\sum_{j=1}^n \frac{\int  x x^\top  f''(\tilde{Y}_j  - \beta_0^\top x) d\mathbb P^{\tilde X}_n(x)} {\int  f(\tilde{Y}_j  - \beta_0^\top x) d\mathbb P^{\tilde X}_n(x)} & = &  \frac{1}{n}\sum_{j=1}^n \frac{\int  x x^\top  f''(\tilde{Y}_j  - \beta_0^\top x) d\mathbb P^{\tilde X}(x)} {\int  f(\tilde{Y}_j  - \beta_0^\top x) d\mathbb P^{\tilde X}(x)}  \\
&&  \ \ + \  \frac{1}{\sqrt n} \int \left( \tilde r_n(x, y) f^Y(y) dy \right) d\mathbb G^{\tilde X}_n(x)\\
&&  \ \ + \  \frac{1}{\sqrt n} \int \left[\left( r_n(x, y) - \tilde r_n(x, y) \right)f^Y(y) dy\right] d\mathbb G^{\tilde X}_n(x) \\
&& \ \ + \ \frac{1}{n} \int \int r_n(x, y) d\mathbb G^{\tilde X}_n(x) d\mathbb G^{\tilde Y}_n(y) 
\end{eqnarray*}
with 
\begin{eqnarray*}
r_n(x, y)  =    \frac{1}{\int f(y- \beta^\top_0 z) d\mathbb P^{\tilde X}_n(z)}  \left( xx^\top f''(y- \beta^\top_0 x) -  \frac{\int zz^\top f''(y- \beta^\top_0 z) d \mathbb P^{\tilde X}(z)}{\int  f(y- \beta^\top_0 z) d \mathbb P^{\tilde X}(z)} f(y- \beta^\top_0 x)  \right).   
\end{eqnarray*}
and 
\begin{align*}
\tilde r_n(x, y)  =    \frac{1}{\int f(y- \beta^\top_0 z) d\mathbb P^{\tilde X}(z)}  \left( xx^\top f''(y- \beta^\top_0 x) -  \frac{\int zz^\top f''(y- \beta^\top_0 z) d \mathbb P^{\tilde X}(z)}{\int  f(y- \beta^\top_0 z) d \mathbb P^{\tilde X}(z)} f(y- \beta^\top_0 x)  \right).
\end{align*}
Using similar techniques as in the proof of Theorem \ref{Remainder}, we can show that 
$$
\sqrt{n}\int \left[ \left( r_n(x, y) - \tilde r_n(x, y) \right) f^Y(y) dy \right] d\mathbb G^{\tilde X}_n(x) = O_{\mathbb P^{\tilde X}}(1)
$$
and 
$$
 \int \int r_n(x, y) d\mathbb G^{\tilde X}_n(x) d\mathbb G^{\tilde Y}_n(y) 
 = O_{\mathbb P^{\tilde X} \otimes \mathbb P^{\tilde Y}}(1)
$$
and by SLLN we also have
$$
\frac{1}{\sqrt n} \int \left( \tilde r_n(x, y) f^Y(y) dy \right) d\mathbb G^{\tilde X}_n(x) = o_{\mathbb P^{\tilde X}}(1).
$$
Furthermore, 
\begin{eqnarray*}
\int  \frac{\int  x x^\top  f''(y  - \beta_0^\top x) d\mathbb P^{\tilde X}(x)} {\int  f(y - \beta_0^\top x) d\mathbb P^{\tilde X}(x)}  d\mathbb P^{\tilde Y}(y) =  \int x x^\top  f''(y  - \beta_0^\top x) d\mathbb P^{\tilde X}(x) dy  = \mathbb E[X X^\top]  \int f''(t) dt = 0
\end{eqnarray*}
using the change of variable $t = y - \beta_0^\top$. By the SLLN, it follows that 
$$
\frac{1}{n}\sum_{j=1}^n \frac{\int  x x^\top  f''(\tilde{Y}_j  - \beta_0^\top x) d\mathbb P^{\tilde X}(x)} {\int  f(\tilde{Y}_j  - \beta_0^\top x) d\mathbb P^{\tilde X}(x)}  \to_{\mathbb P^{\tilde Y}}  0
$$
and the convergence in (\ref{FormerConv}) is proved.  Now, we show the convergence in (\ref{LatterConv}). We write that
\begin{eqnarray*}
 && \frac{1}{n} \sum_{j=1}^n  \left(\frac{\int x f'(\tilde{Y}_j  - \beta_0^\top x) d\mathbb P^{\tilde X}_n(x)} {\int  f(\tilde{Y}_j  - \beta_0^\top x) d\mathbb P^{\tilde X}_n(x)} \right) \left(\frac{\int x^\top f'(\tilde{Y}_j  - \beta_0^\top x) d\mathbb P^{\tilde X}_n(x)} {\int  f(\tilde{Y}_j  - \beta_0^\top x) d\mathbb P^{\tilde X}_n(x)} \right)  \\
 && =  \frac{1}{n} \sum_{j=1}^n  \left(\frac{\int x f'(\tilde{Y}_j  - \beta_0^\top x) d\mathbb P^{\tilde X}(x)} {\int  f(\tilde{Y}_j  - \beta_0^\top x) d\mathbb P^{\tilde X}(x)} \right) \left(\frac{\int x^\top f'(\tilde{Y}_j  - \beta_0^\top x) d\mathbb P^{\tilde X}(x)} {\int  f(\tilde{Y}_j  - \beta_0^\top x) d\mathbb P^{\tilde X}(x)} \right) \\
 &&  \  \  +  \  \frac{1}{\sqrt n} \int \left(\int u_n(x, y) f^Y(y) dy \right)  d\mathbb G^{\tilde X}_n(x) \\
 && \ \  +   \  \frac{1}{n} \int \int u_n(x, y)d\mathbb G^{\tilde X}_n(x) d\mathbb G^{\tilde Y}_n(y)
 \end{eqnarray*}
where 
\begin{eqnarray*}
&& u_n(x, y) \\
&& =   \frac{1}{\left(\int  f(y  - \beta_0^\top z) d\mathbb P^{\tilde X}_n(z) \right)^2} \ \bigg \{    \left(\int z f'(y-\beta_0^\top z) d\mathbb P^{\tilde X}(z) \right) x^\top f'(y-\beta_0^\top x)  \\
&& \ \ +  \ x f'(y-\beta_0^\top x) \left(\int z^\top f'(y-\beta_0^\top z) d\mathbb P^{\tilde X}(z)\right) \\
&& \ \  - \ 2 \  f(y- \beta_0^\top x) \frac{\int z f'(y-\beta_0^\top z) d\mathbb P^{\tilde X}(z) \int z^\top f'(y-\beta_0^\top z) d\mathbb P^{\tilde X}(z)}{f^Y(y)} \\
&& \ \  + \  \frac{1}{\sqrt n}  x f'(y-\beta_0^\top x) \left(\int z^\top f'(y-\beta_0^\top z) d\mathbb G^{\tilde X}_n(z)\right) \\
&& \ \ -  \frac{1}{\sqrt n} f(y- \beta_0^\top x) \frac{\int f(y-\beta_0^\top z) d\mathbb G^{\tilde X}_n(z) \int z f'(y-\beta_0^\top z) d\mathbb P^{\tilde X}(z) \int z^\top f'(y-\beta_0^\top z) d\mathbb P^{\tilde X}(z)}{(f^Y(y))^2} \bigg \}.
\end{eqnarray*}
Note that $u_n(x, y)$ is a matrix of dimension $p \times p$. Thus, similar empirical process arguments involving H\"older classes should be now be applied to each entry $(i, j) \in \{1, \ldots, p \} \times \{1, \ldots, p \}$. Although the calculations are a bit more complex, a formal proof is omitted as main idea remains the same. Now, we can apply the SLLN to conclude that 
\begin{eqnarray*}
&& \frac{1}{n} \sum_{j=1}^n  \left(\frac{\int x f'(\tilde{Y}_j  - \beta_0^\top x) d\mathbb P^{\tilde X}(x)} {\int  f(\tilde{Y}_j  - \beta_0^\top x) d\mathbb P^{\tilde X}(x)} \right) \left(\frac{\int x^\top f'(\tilde{Y}_j  - \beta_0^\top x) d\mathbb P^{\tilde X}(x)} {\int  f(\tilde{Y}_j  - \beta_0^\top x) d\mathbb P^{\tilde X}(x)} \right)  \\
&& \to_{\mathbb P^{\tilde Y}} \mathbb E_{\tilde Y} \left[ \frac{\left(\int x f'(\tilde{Y}  - \beta_0^\top x) d\mathbb P^{\tilde X}(x)\right) \left(\int x^\top f'(\tilde{Y}  - \beta_0^\top x) d\mathbb P^{\tilde X}(x) \right)}{\left(\int  f(\tilde{Y}  - \beta_0^\top x) d\mathbb P^{\tilde X}(x)\right)^2}\right] \\
&&  =  \int  \frac{\left(\int x f'(y- \beta^\top_0 x) f^X(x) dx \right)\left(\int x^\top f'(y- \beta^\top_0 x) f^X(x) dx \right) }{f^Y(y) } dy =\Gamma_1  
\end{eqnarray*}
and the proof is complete.

\end{proof}

\bibliographystyle{plainnat}
\bibliography{unlinkedreg2}

\end{document}